\documentclass[11pt]{article}
\usepackage{anysize}
\usepackage{graphicx}
\usepackage{fullpage}
\usepackage[dvipsnames]{xcolor}
\usepackage{amsfonts}
\usepackage{multicol}
\usepackage{verbatim}
\usepackage{url}
\usepackage{amsmath}
\usepackage{amssymb}
\usepackage{amsthm}
\usepackage{enumitem}
\usepackage{mathtools}
\usepackage{pinlabel}
\usepackage{tikz}
\usepackage{tikz-cd}
\newtheorem*{theorem}{Main Theorem}
\newtheorem{lemma}{Lemma}[section]

\newtheorem{prop}[lemma]{Proposition}
\theoremstyle{definition}

\DeclareMathOperator{\aut}{Aut}
\DeclareMathOperator{\baut}{\textbf{Aut}}
\DeclareMathOperator{\homeo}{Homeo}
\DeclareMathOperator{\bhomeo}{\textbf{Homeo}}
\DeclareMathOperator{\diff}{Diff}

\DeclareMathOperator{\thomeo}{Homeo^{\textit{k}}}
\DeclareMathOperator{\bthomeo}{\textbf{Homeo}^{\textbf{k}}}
\DeclareMathOperator{\mcg}{MCG}
\setlist{noitemsep}
\marginsize{1in}{1in}{1in}{1in}
\title{Automorphisms of smooth fine curve graphs}
\author{Katherine Williams Booth}
\date{}
\begin{document}
\newcommand{\fc}[1]{\mathcal{C}^{\dagger}\!\!\left(#1\right)}
\newcommand{\fcn}[2]{\mathcal{C}^{\dagger}\!\!\left(#2; C^{#1}\right)}
\newcommand{\bfcn}[2]{\boldsymbol{\mathcal{C}}^{\boldsymbol{\dagger}}\!\!\left(#2; C^{#1}\right)}
\newcommand{\efcn}[2]{\mathcal{EC}^{\dagger}\!\!\left(#2; C^{#1}\right)}
\newcommand{\befcn}[2]{\boldsymbol{\mathcal{EC}}^{\boldsymbol{\dagger}}\!\!\left(#2; C^{#1}\right)}
\newcommand{\natp}[1]{\mathcal{N}\!\mathcal{A}_{P}\!\left(#1; C^k\right)}
\newcommand{\fnatp}[1]{\mathcal{N}\!\mathcal{A}^{\dagger}_{P}\!\left(#1; C^k\right)}
\newcommand{\kereq }{{\mathrlap{ \,\bigcirc} \sim} \,}
\maketitle

\begin{abstract}
\noindent In this paper, we consider the automorphisms of fine curve graphs restricted to continuously $k$-differentiable curves.   We show that for closed surfaces with genus at least 2,  they are induced by homeomorphisms of the surface. 
\end{abstract}

\section{Introduction} \label{sectionintro}
Fix a smooth structure on a surface $S$.  For an integer $k \geq 1$,  we say that a simple closed curve $\alpha$ on $S$ is $C^k$ if it is a properly embedded one-dimensional $C^k$ submanifold of $S$.  We define the $C^k$-curve graph $\fcn{k}{S}$ to be the graph whose vertices are essential simple closed $C^k$ curves on $S$. There are edges between vertices when the corresponding curves are disjoint.   

The group of $C^k$ diffeomorphisms of the surface, denoted $\diff^k(S)$,  is a group that naturally acts on the $C^k$-curve graph via its action on the surface.  From related work discussed in Section~\ref{introprior}, one might be tempted to guess that the automorphism group of the $C^k$-curve graph, denoted $\aut \fcn{k}{S}$,  is naturally isomorphic to $\diff^k(S)$.  But it follows from work of Le~Roux--Wolff \cite{LRW} that $\aut \fcn{k}{S}$ is strictly larger than $\diff^k(S)$.  

 Let $\homeo(S)$ denote the homeomorphism group of $S$.  We define $\thomeo(S)$ to be the subgroup of $f\in\homeo(S)$ such that both $f$ and $f^{-1}$ map $C^k$ curves to $C^k$ curves.  
We denote the closed oriented connected surface of genus~$g$ by $S_g$.  The main result of this paper is the following: 
\begin{theorem}\label{theoremautc1}
For $g \geq 2$,  the natural map $$ \textstyle \eta: \thomeo(S_g) \rightarrow \aut \fcn{k}{S_g}$$  is an isomorphism.
\end{theorem}

An equivalent statement is that every automorphism of $\fcn{k}{S_g}$ is induced by a (unique) homeomorphism of $S_g$.  This is the first example of an automorphism group of a fine curve graph that is not isomorphic to $\homeo(S_g)$.  

In addition to the above theorem,  we also give examples of two elements from $\thomeo(S)$ in Section \ref{examples}.  In Example 1, we adjust Le Roux--Wolff's example to more generally apply for $C^k$ curves.  In Example 2, we give an element of $\homeo^1(S)$ that is not in $\homeo^k(S)$, for any $k \geq 2$.  We give several more examples and an exact description for $\homeo^1(S)$ in a separate paper~\cite{BoothHomeo1}. 

The remainder of this introduction is organized as follows.  In Section~\ref{introprior}, we discuss some prior work related to automorphisms of curve graphs.   Section~\ref{introproofidea} gives an overview of our proof and outlines the remaining sections of this paper.

\subsection{Related work}\label{introprior}

This section is organized into two broad categories.  First, we discuss the foundational work of Ivanov on the classical curve graph utilizing isotopy classes of curves on a surface.  We then look at the recent advancements concerning fine curve graphs,  which remove isotopies and instead consider every curve on a surface.

\subsubsection*{Curve graphs and Ivanov's metaconjecture}

We start our survey of prior work with the more historical of the two directions.  The curve graph of a surface $S$, denoted by $\mathcal{C}(S)$,  is the graph whose vertices are isotopy classes of essential simple closed curves. There are edges between two vertices whenever there exist representatives from each isotopy class that are disjoint. The curve graph is the 1-skeleton of the curve complex first introduced by Harvey \cite{Harvey} in 1978.  He used this complex as a tool to study the boundary of the mapping class group $\mcg^\pm(S)$,  the group of homeomorphisms of $S$ up to isotopy.  

$\mcg^{\pm}(S)$ naturally acts on $\mathcal{C}(S)$ by the action of an element of the mapping class group on the surface.  Ivanov \cite{Ivanov} showed that for surface with genus at least 3, the natural map $$\mcg^{\pm}(S_g) \rightarrow \aut \mathcal{C}(S_g)$$ is an isomorphism.  
This result directly mirrors our own result, providing a combinatorial model for $\mcg^{\pm}(S_g)$.

Ivanov's result also inspired several others to find other complexes associated to surfaces that also have natural isomorphisms between their automorphism groups and the mapping class group (\cite{BM2}, \cite{Irmak}, \cite{IK}, \cite{Jyothis}, and many others).  This led Ivanov to state the following metaconjecture:

\medskip
\noindent \textit{Every object naturally associated to a surface $S$ and having a sufficiently rich structure has $\mcg^{\pm}(S)$ as its group of automorphisms. Moreover, this can be proved by a reduction to the theorem about the automorphisms of $\mathcal{C}(S)$.}
\medskip

Brendle--Margalit \cite{BM} give a survey of many of these results and a characterization for what ``sufficiently rich" entails for simplicial complexes associated to surfaces. 

\subsubsection*{Fine curve graphs}

More recently in 2012,  Farb--Margalit~\cite{BensonTalk} introduced the fine curve graph $\fc{S}$.  This graph has vertices for every essential closed curve on the surface and edges when the curves are disjoint.  Bowden--Hensel--Webb \cite{BHW} used a modified version with smooth curves and considered the action of $\diff_0(S_g)$ on this hyperbolic graph.  They show that the space of quasi-morphisms is infinite dimensional and thus $\diff_0(S_g)$ is not uniformly perfect when $g$ is at least 1.  

With the topological curve version of $\fc{S}$,   $\homeo(S)$ is the group that naturally acts on the graph.  In a similar vein as Ivanov's result for the classical curve graph,   Long--Margalit--Pham--Verberne--Yao \cite[Theorem 1.1]{LMPVY} showed that for surfaces with genus at least 2,  the natural map
$$ \homeo(S_g) \rightarrow \aut \fc{S_g} $$
 is an isomorphism.

Like the work that branches from Ivanov's seminal result, there are several variants of the fine curve graph which also serve as combinatorial models of $\homeo(S)$.  Kimura--Kuno \cite{KK} show that the argument given by Long--Margalit--Pham--Verberne--Yao can be extended to nonorientable surfaces with genus at least 4.  To expand this even further to all surfaces with genus at least 1,  Le Roux--Wolff \cite{LRW} use a graph of nonseparating curves as vertices with edges when curves are either disjoint or have exactly one topologically transverse intersection.  The author, with Minahan and Shapiro,  \cite{1fine} give a similar result on oriented surfaces for the fine 1-curve graph, which adds edges to the fine curve graph when curves intersect in a single point.  


Based on the similarity of the above results,  we ask the following: 

\medskip 

\noindent \textbf{Question: } \textit{Is there a fine version of Ivanov's metaconjecture that encompasses both the graphs associated to $\homeo(S)$ and the graphs associated to $\homeo^k(S)$?}

\medskip

Our work shows that any attempt for a fine metaconjecture will require more nuance than simply replacing $\mcg^{\pm}(S)$ in Ivanov's result by $\homeo(S)$.  This nuance is also reflected in the fact that $\mcg^{\pm}(S)$ can be defined in many ways, including as the group of diffeomorphisms of $S$ up to isotopy.

\subsection{Overview of the proof and paper outline}\label{introproofidea}
 The proof of our theorem follows the same overarching structure as the work of Long--Margalit--Pham--Verberne--Yao for $\fc{S}$~\cite{LMPVY}.  We use the \emph{extended $C^k$-curve graph}, which we denote by $\efcn{k}{S}$.  The vertices for this graph include both the essential and inessential simple closed $C^k$ curves on $S$.  As with the $C^k$-curve graph, there are edges between vertices when the corresponding curves are disjoint.   Our main argument relies on building the chain of homomorphisms shown below. 
$$\textstyle \aut \fcn{k}{S} \xrightarrow{\xi^{-1}}\aut \efcn{k}{S} \xrightarrow{\rho} \thomeo(S) \xrightarrow{\eta} \aut \fcn{k}{S}$$

\noindent We now give a brief overview of our proof and explain the constructions for the maps $\rho$ and $\xi$.  The map $\eta$ comes directly from the natural action of $\thomeo(S)$ on $\fcn{k}{S}$ and needs no further discussion.  When relevant, we also point out the key differences that arise when the $\fc{S}$ argument is applied to $C^k$ curves.

\subsubsection*{Step 1: $\mathbf{\befcn{k}{S} \boldsymbol{\xrightarrow{\boldsymbol{\rho}}} \bthomeo(S)}$} We start our proof in Section~\ref{sectionec1tohomeo1} with the homomorphism $\rho$ that reconstructs a homeomorphism of the surface from an automorphism of $\efcn{k}{S}$. In a similar fashion to the $\fc{S}$ case,  we take sequences of inessential curves that converge to a point.  These converging sequences are then used to determine exactly where each point on $S$ must be sent, defining the desired homeomorphism.

\medskip

\noindent \textit{Key difference: Connecting-the-dots.  }  Long--Margalit--Pham--Verberne--Yao~\cite{LMPVY} identify the combinatorial structure of these converging sequences of curves using other curves drawn through particular sequences of converging points.  With only a slight adjustment, a similar argument applies to the $C^1$-curve graph.  

Unfortunately,  this strategy does not work for the general $C^k$ case.  There exist sequences of converging points such that for any $k \geq 2$,  there is no $C^k$ curve that contains infinitely many of these points.  We discuss this obstruction in further detail in Section~\ref{sectionec1tohomeo1}.  To overcome this obstacle,  we give a nesting restriction to converging sequences and utilize the Jordan Curve Theorem to guarantee the existence of the desired sequence of points.  This new combinatorial characterization allows the remainder of the $\fc{S}$ argument to be applied to $C^k$ curves and extended to many other families of curves.

\subsubsection*{Step 2: $\mathbf{\baut \bfcn{k}{S} \boldsymbol{\xrightarrow{\boldsymbol{\xi}^{-1}}} \baut \befcn{k}{S}}$} Next, we work towards building the isomorphism $\xi$,  making the connection between the automorphisms of $\fcn{k}{S_g}$ and $\efcn{k}{S_g}$, when $g \geq 2$.  

In Section~\ref{sectioncurvepairs}, we identify the graph structures that are unique to a variety of topological properties of curves.  We use these structures to build up a combinatorial perspective for particular pairs of essential $C^k$ curves that determine a unique inessential $C^k$ curve.  

\medskip

\noindent \textit{Key difference: Topological curve constructions.  } 
The work of Long--Margalit--Pham--Verberne--Yao~\cite{LMPVY} accomplishes this step using a pair of homotopic curves, called a bigon pair.  This construction relies on topological curves that must have corners and so cannot be differentiable.  This argument is not easily adapted for higher regularity curves. 

\begin{figure}[h]
\centering
\includegraphics[height=.8in]{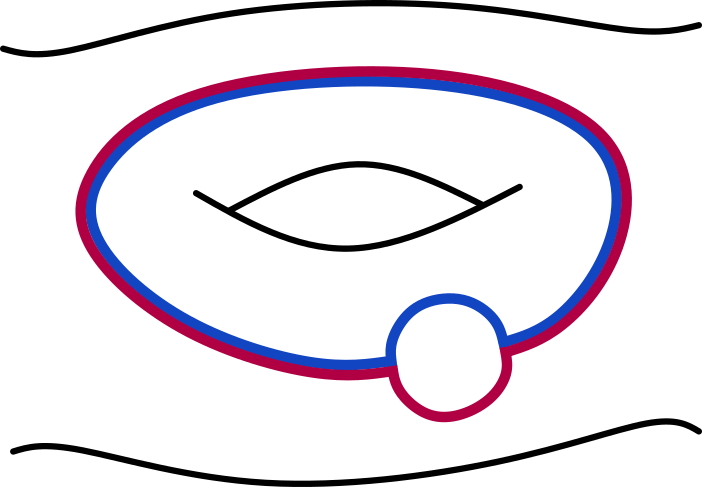} \hspace{.75in}
\includegraphics[height=.8in]{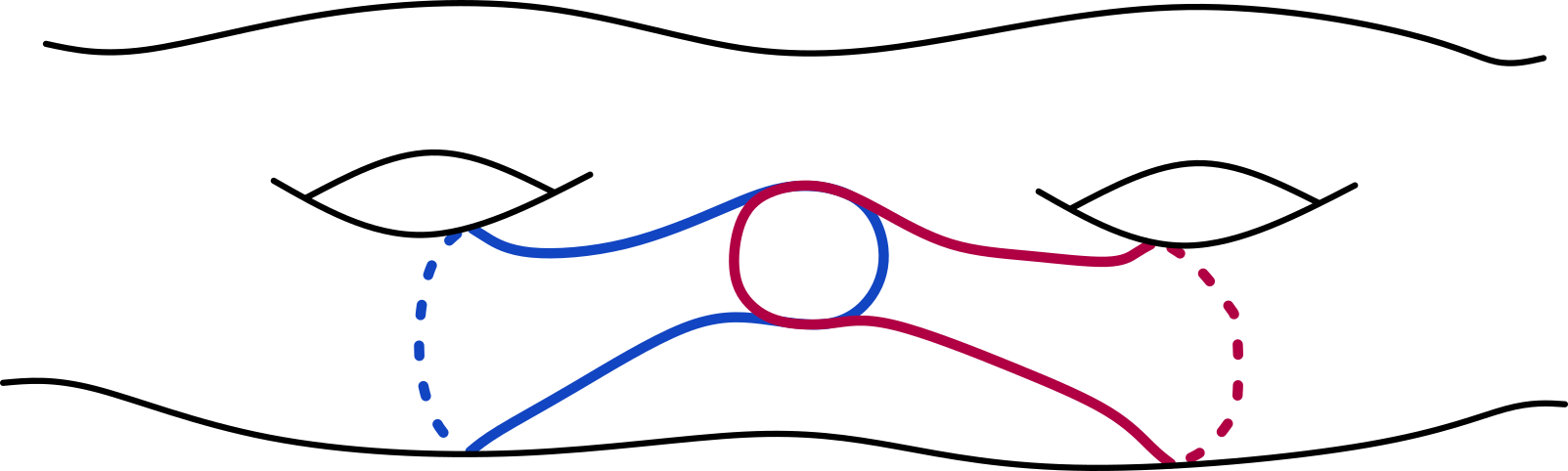}
\caption{Bigon pair (left) and $k$-smooth pair (right)}
\end{figure}

For $C^k$ curves, we have found a new construction in $\fcn{k}{S_g}$, called simple $k$-smooth pairs,  that uniquely determine inessential curves.  The graph structures related to our construction are more complicated than the structures for bigon pairs.

\medskip

\noindent In Section~\ref{sectionc1toec1},  we then utilize the simple $k$-smooth pairs to construct an isomorphism from $\aut \efcn{k}{S_g}$ to $\aut \fcn{k}{S_g}$ when $g \geq 2$.  The primary component involves building an equivalence relation on a connected arc graph to show that the inverse of the natural restriction map is well-defined.

\subsubsection*{Step 3: Completing the proof} 
In Section~\ref{sectionfinal}, we compose the above maps to get the identity, showing that the final map $\eta$ must by surjective.  A straighforward injectivity argument completes the isomorphism. 

\medskip

\subsubsection*{Example elements } 
Finally in Section~\ref{examples}, we give more details for the examples described earlier and a subtle non-example.  These specific maps explore the possible connections between the groups $\thomeo(S)$ for different values of $k$.  In particular, they lead us to ask the following:

\medskip 
\noindent \textbf{Question: } \textit{Are there any containment relationships between  $\homeo^k(S)$ and $\homeo^m(S)$ when $k \neq m$?}

\medskip
\noindent While the family of $\diff^k(S)$ is properly nested, it is not clear if a similar relationship exists for $\homeo^k(S)$.  In particular,  for larger $k$, the graph is restricted to curves that have more structure, but the restriction also implies that we have less curves to potentially control the behavior of elements in $\homeo^k(S)$.

\subsection{Acknowledgments}

We would like to thank our advisor Dan Margalit for his generous support and encouragement.  This work is considerably better than it would have been without his continual comments and critical eye.

We also thank Igor Belegradek,  Aaron Calderon,  Sierra Knavel, Thomas Koberda, Kathryn Mann, Daniel Minahan,  Alex Nolte,  Roberta Shapiro, and Yvon Verberne for helpful conversations and general encouragement for this project. We thank Zachary Himes for comments on an earlier draft.
The author was partially supported by the National Science Foundation under Grants No.  DMS-1745583 and DMS-2417920.

\section{Constructing $ \mathbf{\boldsymbol{\rho}: \baut \befcn{k}{S} \boldsymbol{\rightarrow} \bthomeo(S)}$}\label{sectionec1tohomeo1}

The goal of this section is to prove the following result,  building a map from the automorphisms of the extended $C^k$-curve graph to $\thomeo(S)$.  This is Step 1 from the proof outline given in the introduction.  

\begin{prop}\label{ec1tohomeo1}
Let $S$ be a smooth surface.  Then there is a natural map $$\rho : \aut \efcn{k}{S} \rightarrow \thomeo(S)$$ such that for any $\alpha \in \aut \efcn{k}{S}$,  $(\widehat{\eta} \circ \rho)(\alpha) = \alpha$, where $\widehat{\eta}$ is the natural map from $\homeo^k(S)$ to $\aut \efcn{k}{S}$. 
\end{prop}

To build this homeomorphism,  we use sequences of inessential curves that converge to points on the surface.  A key component of their structure in the graph involves other curves that intersect infinitely many of the converging curves and the point that they converge to.

\medskip

\noindent \textit{Previous approach and why it does not work here.  } As mentioned in Section~\ref{sectionintro},  to accomplish this step for the fine curve graph, 
 Long--Margalit--Pham--Verberne--Yao~\cite[Lemma 4.1]{LMPVY} utilize a connect-the-dots-lemma that gives the existence of a curve through infinitely many points.  In particular, they use the classification of surfaces by Richards~\cite{Richards} to straighten the points all to the $x$-axis. The $x$-axis can then be chosen as the desired curve.  
 
Unfortunately,  Richards' result is purely topological and has no smoothness conditions.  Due to a classic result by Rosenthal~\cite[Theorem~2]{Rosenthal} a similar (but not identical) process can be achieved for $C^1$ curves, but is not possible with $C^k$ curves,  when $k \geq 2$. 
In particular, for a set of points along the curve $y=|x|^{3/2}$ that converge to 0,  no $C^2$ curve will hit infinitely many of them as the curvature approaches $\infty$ as $x \rightarrow 0$.

\medskip

\noindent \textit{Our new approach.  }To generalize this strategy to all $C^k$ curves, we give a nesting restriction to converging curves and utilize the Jordan Curve Theorem.  This ensures that there exists a curve that intersects infinitely many of the curves in the convergent sequence.  This change has complicated both the combinatorial description of these sequences and the argument to show that they are well-defined. 

We then show that the images of converging curves must also converge to a point.  This allows us to define a specific image point for every point in the surface.  To complete the homeomorphism, we show that this map is well-defined,  bijective,  and continuous.  We now give a few relevant structures and lemmas. 

\medskip

\noindent \textit{Graph structures.  }
For a collection of vertices $\{v_1, v_2, \ldots, v_n\}$ of a graph $\Gamma$, the \emph{link} of these vertices, denoted $link(v_1, v_2, \ldots, v_n)$,  is the subgraph of $\Gamma \setminus \{v_1, \ldots, v_n\}$ induced by the set of vertices $u_j$ that are adjacent to all of the $v_i$. 

A graph $\Gamma$ is a \emph{$m$-join} if there is a partition of the vertices of $\Gamma$ into $m$ sets $V_1, \ldots, V_m$, called parts,  such that each vertex in $V_i$ is adjacent to every vertex in $V_j$ when $i \neq j$. 

Note that any $m$-join induces an $j$-join for any $j < m$ by combining the sets of the partition.  So any $m$-join with $m \geq 2$ induces a 2-join. For simplicity, we use the term \emph{join} to refer to any 2-join. We use these definitions to get the following:

\begin{lemma}\label{sepcurveinextgraph}
Let $S$ be an oriented surface.  Let $a \in \efcn{k}{S}$. Then $a$ is separating if and only if $link(a)$ is a join.
\end{lemma}

\begin{proof}
If $a$ is a separating curve, then any curve disjoint from it will be in only one component of $S \setminus a$. But then it will also be disjoint from every curve that is in the other component of $S \setminus a$. Note that neither side will be empty, since $S$ is oriented, then $a$ can be homotoped off of itself to make curves $a'$ and $a''$ on both sides of the separation. Thus the link of $a$ will be a join. 

On the other hand,  consider the case when $a$ is nonseparating. Suppose, for a contradiction, that the link of $a$ is a join.  Let $b$ and $c$ be curves in different parts of this join.  Our goal is to find another curve $d$ in the link of $a$ that intersects both $b$ and $c$.   Then this curve $d$ could not be placed in either part of the join since it does not have an edge to every vertex in the other part.  Therefore the link of $a$ cannot be a join, contradicting the initial assumption.

We will now find the curve $d$.  Since $a$ is nonseparating, then it is possible to draw an arc $e$ from $b$ to $c$ without intersecting $a$.  Since $a$ is closed,  then there is an open neighborhood of $e$ that is disjoint from $a$. So $b$ can be smoothly isotoped to a curve $d \in \efcn{k}{S}$ that stays within the neighborhood of $e$ until it intersects $c$.  
\end{proof}

\medskip

\noindent \textit{Nested sequences.  }
A sequence of curves $(c_i)$ on a surface $S$ is called \emph{nested} if each $c_i$ is separating and there exists a component $C_i$ of $S \setminus c_i$ such that $c_j \subset C_i$, for all $j > i$. 

This definition immediately implies that all the curves of a nested sequence must be disjoint from one another.  In addition, we have the following consequence:

\begin{lemma}\label{nestedcurvespreserved}
Let $S$ be a surface and $\alpha \in \aut \efcn{k}{S}$. If $(c_i) \in \efcn{k}{S}$ is a nested sequence of curves, then the sequence $(\alpha(c_i))$ is also nested. 
\end{lemma}

\begin{proof}
First, by the definition of nested, the $(c_i)$ are all separating. So by Lemma \ref{sepcurveinextgraph},  the link of $c_i$ is a join. Also, the set $\{c_j \; | \; j >i\}$ will be contained in one only part of the join.  Since $\alpha$ is an automorphism of the graph, than $\alpha(c_i)$ will also be a join with the set $\{\alpha(c_j) \; | \; j >i\}$ contained in only one part of the join.  Thus $\alpha(c_i)$ is also separating.  Moreover, each part of the join corresponds exactly with the curves contained in one of the two components of $S \setminus \alpha(c_i)$. Thus $(\alpha(c_i))$ is also a nested sequence on $S$.
\end{proof}

\medskip 

\noindent \textit{Convergent sequences.  }
A sequence of curves $(c_i)$ will be said to \emph{converge to a set $P$} if any regular neighborhood of $P$ contains all but finitely many of the $c_i$.  We say that a curve  \emph{intersects the tail} of a sequence $(c_i)$ if it intersects infinitely many of the $c_i$. 

We can now give a unique combinatorial description of convergent nested sequences.

\begin{lemma}\label{convseqineckgraph}
A sequence of nested curves $(c_i) \in \efcn{k}{S}$ converges to a point $x \in S$ if and only if both of the following hold:
\begin{enumerate}[noitemsep,topsep=0pt, label=$(\roman*)$]
\item there exists a curve $a \in \efcn{k}{S}$ that intersects the tail of $(c_i)$
\item any two curves $a, b \in \efcn{k}{S}$ that intersect the tail of $(c_i)$ must intersect each other
\end{enumerate}
\end{lemma}

\begin{proof}
In the forward direction,  let $a \in \efcn{k}{S}$ be a curve that intersects the point $x$.   Then this curve will also leave some small neighborhood of $x$. But by the definition of convergence,  all but finitely many of the $c_i$ are contained in this neighborhood.  Since the $(c_i)$ are nested, then for each $c_i$,  $x$ must be contained in the component of $S \setminus c_i$ that also contains all the $c_j$ where $j >i$ and for large enough $i$, these components will also be contained in the small neighborhood of $x$.  So by the Jordan Curve Theorem,  $a$ must intersect each of the corresponding $c_i$ whose components are contained in the neighborhood of $x$.  Thus $a$ will intersect the tail of $(c_i)$. 

Since $x$ is the limit point for any sequence of points $x_i \in c_i$, then any curve that intersects infinitely many of the $c_i$ must also contain the point $x$.  So any curves that intersect the tail of $(c_i)$ will intersect each other at the point $x$. 

In the reverse direction,  suppose that the $(c_i)$ do not converge to a single point. Either the $(c_i)$  do not converge to any points, or they converge to more than one point.  If they do not converge to any  points, then there are only finitely many of the $c_i$ in any compact subsurface of $S$.  But any curve in $\efcn{k}{S}$ is contained in a compact subsurface of $S$, so it cannot intersect the tail of the $(c_i)$.  On the other hand, if the $(c_i)$ converge to more than one point, say $x$ and $y \in S$, then we can construct disjoint curves $a, b \in \efcn{k}{S}$ such that $a$ contains the point $x$ and $b$ contains the point $y$.  By constructing these curves carefully, we can ensure that both intersect the tail of the $(c_i)$. 

First, since $S$ is Hausdorff,  there exist disjoint open neighborhoods of $x$ and $y$.  Let $N$ denote the neighborhood of $x$.  Since the $(c_i)$ converge to $x$, then all but finitely many of the $c_i$ will intersect $N$.  Considering one of the $c_i$ that intersect $N$,  it will separate $N$ into at least two components. Let $C_x$ denote the component that contains the point $x$.  Let $x'$ be a point in $N$ that is in the component of $S \setminus c_i$ that does not include $x$.  Since $N$ is an open set,  there exists a curve $a \in \efcn{k}{S}$ through $x$ and $x'$ that is contained within $N$. 

\begin{figure}[h]
\centering
\begin{tikzpicture}
\small
    \node[anchor=south west, inner sep = 0] at (0,0){\includegraphics[width=2in]{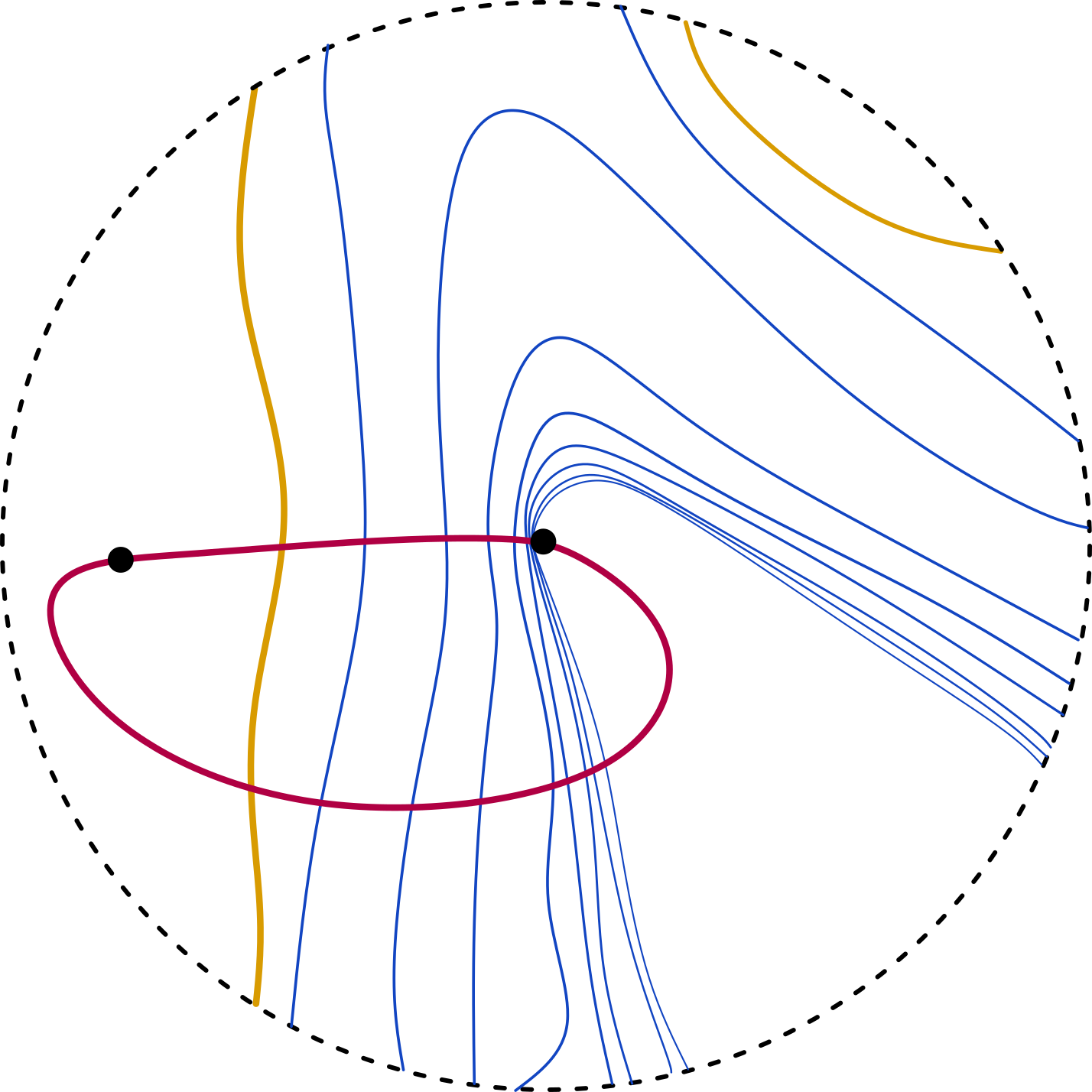}};
    \node at (.92,4) {$c_i$};
    \node at (.5,2.7) {$x'$};
    \node at (2.72,2.7) {$x$};
    \node at (3.2,1.65) {$a$};
\end{tikzpicture}
\caption{Finding a curve that intersects the tail of $(c_i)$ near the point $x$}
\end{figure}
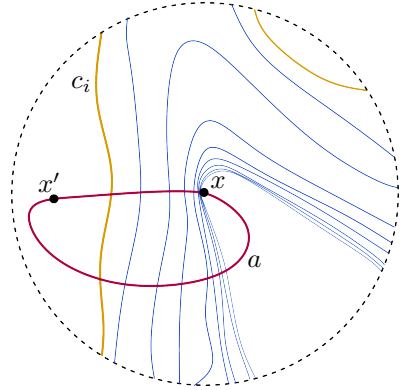

 Moreover,   $x$ is always in the component with infinitely many $c_j$ with $j>i$.  So any curve from $x'$ to $x$ must intersect infinitely many $c_j$. 
Thus $a$ intersects the tail of the $(c_i)$. 

But this same construction can also be done within the neighborhood of $y$ to find a curve $b \in \efcn{k}{S}$ that also intersects the tail of $(c_i)$ and is disjoint from $a$. Thus condition (ii) fails to hold. 
\end{proof}

The characterization given in Lemma \ref{convseqineckgraph} utilizes only the relations that are present in the graph,  so any automorphism of $\efcn{k}{S}$ must preserve convergent sequences.

\begin{lemma} \label{autec1tohomeo}
Every $\alpha \in \aut \efcn{k}{S}$ induces a homeomorphism $f_{\alpha}$ such that the action of $f_{\alpha}$ on $\efcn{k}{S}$ is given by $\alpha$.
\end{lemma}

\begin{proof}
Let $x \in S$ and let $(c_i)$ be a nested sequence in $\efcn{k}{S}$ that converges to $x$. By Lemma \ref{convseqineckgraph}, $(\alpha(c_i))$ will be a convergent nested sequence that converges to some point $y \in S$.  Define the function $f_{\alpha}: S \rightarrow S$ by $f_{\alpha}(x)=y$, for every point $x \in S$. 

This map is well-defined. To see this consider two nested sequences $(c_i)$ and $(d_j)$ that converge to the same point $x$. Notice that for each $c_i$, the component of $S \setminus c_i$ that contains $x$ is itself a neighborhood for $x$. Thus by the definition of convergence, all but finitely many of the $d_j$ must also be contained in this component.  So we can choose subsequences $(c_{i_n})$ and $(d_{j_n})$ such that each $d_{j_n}$ is contained in the same component of $S \setminus c_{i_n}$ as the point $x$ and each $c_{i_{n+1}}$ is contained in the same component of $S \setminus d_{j_n}$ as the point $x$.  Then the sequence of curves $$e_n = \left\{ \begin{array}{ll} c_{i_{n/2}} & n \mbox{ even} \\ d_{j_{(n+1)/2}} & n \mbox{ odd} \end{array} \right.$$ will also be a nested sequence that converges to the point $x$.  Since $(\alpha(e_n))$ converges to a single point, then both $(\alpha(c_{i_n}))$ and $(\alpha(d_{j_n}))$ converge to the same point.  But $(\alpha(c_{i_n}))$ and $(\alpha(c_i))$, and also $(\alpha(d_{j_n}))$ and $(\alpha(d_j))$,  must then all converge to the same point. Thus $f_{\alpha}$ is a well-defined map.

Since every point $y \in S$ will have some nested convergent sequence $(d_j)$ and $\alpha$ is an automorphism of the graph, then $(\alpha^{-1}(d_j))$ must also be a nested convergent sequence that converges to some point $x \in S$.  But then $f_{\alpha}(x)$ will be the point that $(\alpha(\alpha^{-1}(d_j)))= (d_j)$ converges to, which is the point $y$.  Thus $f_{\alpha}$ is surjective. 

This map is also injective.  Let $(c_i)$ and $(d_j)$ be nested sequences that converge to the points $x_1$ and $x_2$ such that $(\alpha(c_i))$ and $(\alpha(d_j))$ converge to the same point $y$.  Similar to above, we can construct a nested sequence $e_n$ that interleaves curves from $(\alpha(c_i))$ and $(\alpha(d_j))$. So $(\alpha^{-1}(e_n))$ is a convergent sequence that converges to a single point.  But applying $\alpha^{-1}$ to the chosen subsequences of $(\alpha(c_i))$ and $(\alpha(d_j))$ will again give us that the corresponding subsequences of $(c_i)$ and $(d_j)$ must converge to the same point and thus $x_1=x_2$.

Finally, we need to show that $f_{\alpha}$ and $f_{\alpha}^{-1}$ are both continuous. Since $f_{\alpha^{-1}} \circ f_{\alpha}$ will be the identity on all curves, then $f_{\alpha}^{-1} = f_{\alpha^{-1}}$. So we only need to show that $f_{\alpha}$ is continuous.  Let $x_n$  be a sequence of points in $S$ that converges to $x$.   We start by constructing nested sequences $(c_j^n)$ of curves in $\efcn{k}{S}$ that converge to each $x_n$ such that $c^n_n$ does not separate $x$ from the set $\{x_m \; | \; m \geq n\}$ and is contained in the same component of $S \setminus c^{n-1}_{n-1}$ as $x$.  These $c^n_n$ exist since $c^{n-1}_{n-1} \cup \{x\} \cup \{x_m \; | \; m \geq n\}$ is a closed set and so the components of the complement will all be open sets. Moreover, the component that contains $c^n_n$ is a disk with a countable number of points removed and the only limit point for these removed points is $x$, which is not on the boundary of the disk.  So it is possible to find a $C^k$ curve that separates $c^{n-1}_{n-1}$ from $\{x\} \cup \{x_m \; | \; m \geq n\}$.  Then by construction, the diagonal sequence $(c_n^n)$ are nested and converge to $x$.  So the sequence $(\alpha(c_n^n))$ must also converge to a point $y$ which is the same point that the sequence of points $y_n$ of $(\alpha(c_j^n))$ converge to.  Thus $f_{\alpha}$ is continuous and so $f_{\alpha}$ is a homeomorphism of $S$. 

For each point $x$ on a curve $c \in \efcn{k}{S}$, there is a nested sequence $(c_i)$ that converges to $x$. So by Lemma \ref{convseqineckgraph},  $c$ will intersect the tail of the $c_i$. Thus $\alpha(c)$ must intersect the tail of the $\alpha(c_i)$. Since $(\alpha(c_i))$ converge to the point $f_{\alpha}(x)$, then $\alpha(c)$ must contain the point $f_{\alpha}(x)$.  Since this holds for every point in $c$, then $f_{\alpha}(c) = \alpha(c)$.  Thus $f_{\alpha}$ induces the same action on $\efcn{k}{S}$ as $\alpha$. 
\end{proof}

\noindent \textit{The convergent sequences map.  }
Define the map $$\rho : \aut \efcn{k}{S} \rightarrow \homeo(S)$$ given by $\rho[\alpha] = f_{\alpha}$ from Lemma \ref{autec1tohomeo}.

\medskip

Finally, we prove the main proposition for this section, showing that all the homeomorphisms naturally induced from the extended graph are elements of $\thomeo(S)$. 

\begin{proof}[Proof of Proposition \ref{ec1tohomeo1}]
By Lemma \ref{autec1tohomeo},  the homeomorphism $\rho[\alpha]$ gives an action on $\efcn{k}{S}$ given by the automorphism $\alpha$. So both $\rho[\alpha]$ and $\rho[\alpha]^{-1} = \rho[\alpha^{-1}]$ map $C^k$ curves to $C^k$ curves. Thus $\rho[\alpha]$ is an element of $\thomeo(S)$.
\end{proof}

\section{Distinguishing pairs of curves}\label{sectioncurvepairs}

In this section, we complete the first half of Step 2 from the proof outline given in the introduction.  We are working towards defining a map between the automorphisms of the $C^k$-curve graph and the automorphisms of the extended $C^k$-curve graph.  In particular,  we find a pair of essential curves that uniquely determine an inessential curve.  

While along the same vein as the work of Long--Margalit--Pham--Verberne--Yao~\cite{LMPVY}, their construction of bigon pairs does not work when restricted to $C^k$ curves since it requires corners to exist for either curve in the bigon pair or the inessential curve associated to the pair.  So a new construction of curves, called simple $k$-smooth pairs, was needed to make the connection between these two graphs.  By the end of this section, we will prove the following proposition, showing that these pairs have a unique structure in the $C^k$-curve graphs.

\begin{prop}\label{c1preservessimplesmoothpair}
Let $S_g$ be a surface with $g \geq 2$.  Then every automorphism of $ \fcn{k}{S_g}$ preserves the set of simple $k$-smooth pairs.
\end{prop}

This section is broken into several subsections, each giving different curve structures that are preserved by the automorphisms of the $C^k$-curve graph. The preserved structures for each section are: 

 \begin{enumerate}[label=\S \ref{sectioncurvepairs}.\arabic*:, leftmargin=2\parindent, topsep=3pt]
\item separating, homotopic and hulls of curves
\item annuli, nonseparating pants, and one holed tori 
\item torus pairs 
\item $k$-smooth pairs
\item simple $k$-smooth pairs
\end{enumerate} 

\subsection{Separating, homotopic, and hulls of curves}\label{subsectionsephomhull}

The goal of this section is to prove Lemma \ref{autpreservesephomandhull}, which shows that sets of separating, homotopic, and hulls of curves are preserved by automorphisms of the $C^k$-curve graph.  In order to do this, we recall a few relevant graph structures defined in Section~\ref{sectionec1tohomeo1}.

\medskip

Recall that the \emph{link of the vertices $\{v_1,  v_2, \ldots, v_n\}$} of a graph $\Gamma$, denoted $link(v_1, v_2, \ldots, v_n)$,  is the subgraph of $\Gamma \setminus \{v_1, \ldots, v_n\}$ induced by the set of vertices $u_j$ that are adjacent to all of the $v_i$. 

A graph $\Gamma$ is a \emph{join} if there is a partition of the vertices of $\Gamma$ into two parts,  such that each vertex in one part is adjacent to every vertex in the other part. 

\begin{lemma}\label{sepcurveingraph}
Let $S$ be an oriented surface.  Let $a \in \fcn{k}{S}$. Then $a$ is separating if and only if $link(a)$ is a join.
\end{lemma}

\begin{lemma}\label{jointlyseparatingingraph}
Let $\{\gamma_1, \ldots, \gamma_m\}$ be disjoint curves in $\fcn{k}{S}$.  Then $\{\gamma_1, \ldots, \gamma_m\}$ jointly separate $S$ if and only if $link(\gamma_1, \ldots, \gamma_m)$ is a join.
\end{lemma}

The proofs for both of the above lemmas are omitted since they are essentially the same as the proof of Lemma \ref{sepcurveinextgraph}. 

\begin{lemma}\label{disjointisotoptograph}
Let $g \geq2$ and $a, b$ be disjoint curves in $\fcn{k}{S_g}$. Then $a$ and $b$ are homotopic if and only if $link(a,b)$ is a join where one part contains either only separating or only nonseparating curves.
\end{lemma}

\begin{proof}
Suppose that $a$ and $b$ are disjoint homotopic curves.  Up to a change of coordinates \cite{Primer} they will look like one of the examples in Figure~\ref{fig-sepvnsep}.  So they will jointly separate the surface, with one component being an annulus.  Note that all the essential curves in this annulus will be homotopic to $a$ and $b$.  Thus they will either all be separating or all nonseparating, based on the whether $a$ and $b$ are separating or nonseparating.
\begin{figure}[h]
\centering
\includegraphics[width=3in]{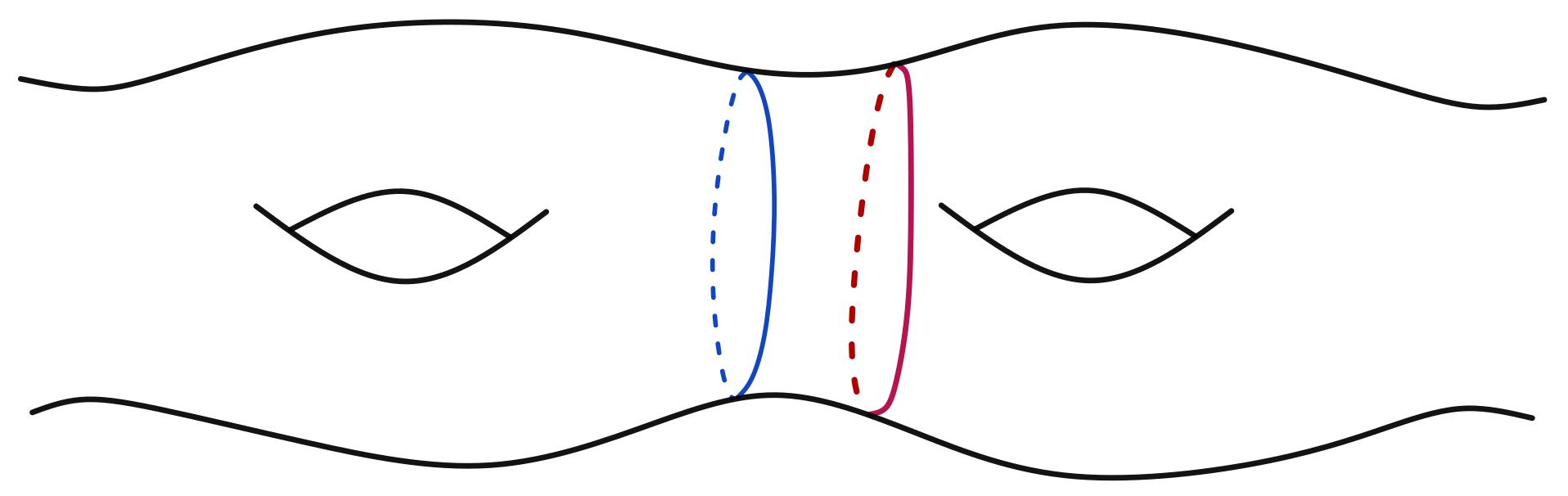} \hspace{.5in}
\includegraphics[width=1.5in]{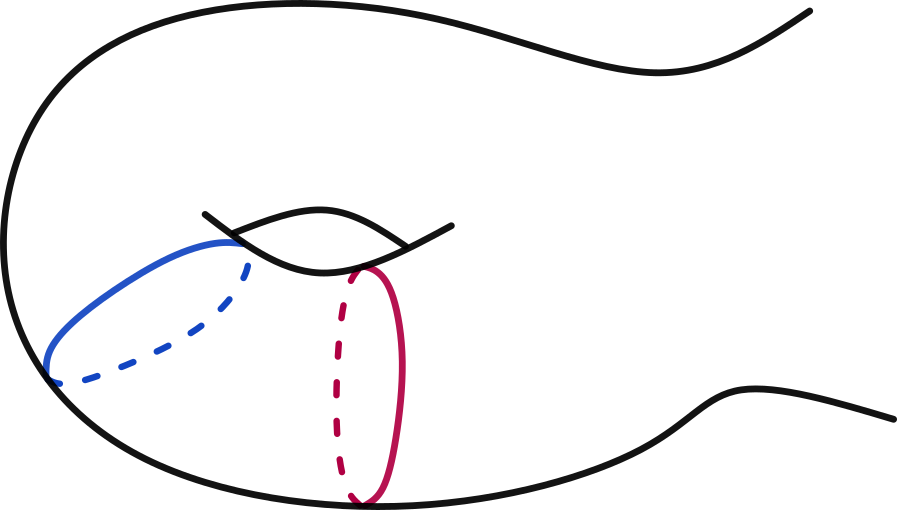}
\caption{Pairs of separating (left) and non-separating (right) homotopic curves}
\label{fig-sepvnsep}
\end{figure}

Suppose on the other hand that $a$ and $b$ are disjoint non-homotopic curves.  Suppose that $a$ is separating and $b$ is nonseparating.  By Lemma \ref{sepcurveingraph}, the link of $a$ is a join and $b$ will be in one of the components of $S_g \setminus a$.  By pushing $a$ off of itself on both sides, then there is a separating curve in both sides of this join that does not intersect $b$.  Also, a push off of $b$ will be a nonseparating curve on one side of $S_g \setminus a$ that is also disjoint from $b$.  The component of $S_g \setminus a$ that does not contain $b$ will have enough topology to contain another nonseparating curve, which must be disjoint from $a$ and $b$. Thus both parts of the join structure of $link(a,b)$ will contain both separating and nonseparating curves.

If both $a$ and $b$ are separating but not homotopic, then $link(a,b)$ will be a 3-join with each part containing enough topology to contain a nonseparating curve.  Each part will also contain separating curves by pushing off of either $a$ or $b$. 

Finally, if both $a$ and $b$ are nonseparating, but not homotopic, then they are either jointly nonseparating or jointly separating. In the first case, the $link(a,b)$ will not be a join by Lemma \ref{jointlyseparatingingraph}. In the second case, as shown below, each component of $S_g \setminus \{a,b\}$ will have two boundary components and at least one genus.  So there is enough topology to find both separating and nonseparating curves in each part of the join structure of $link(a,b)$.
\begin{figure}[h]
\centering
\includegraphics[width=3.5in]{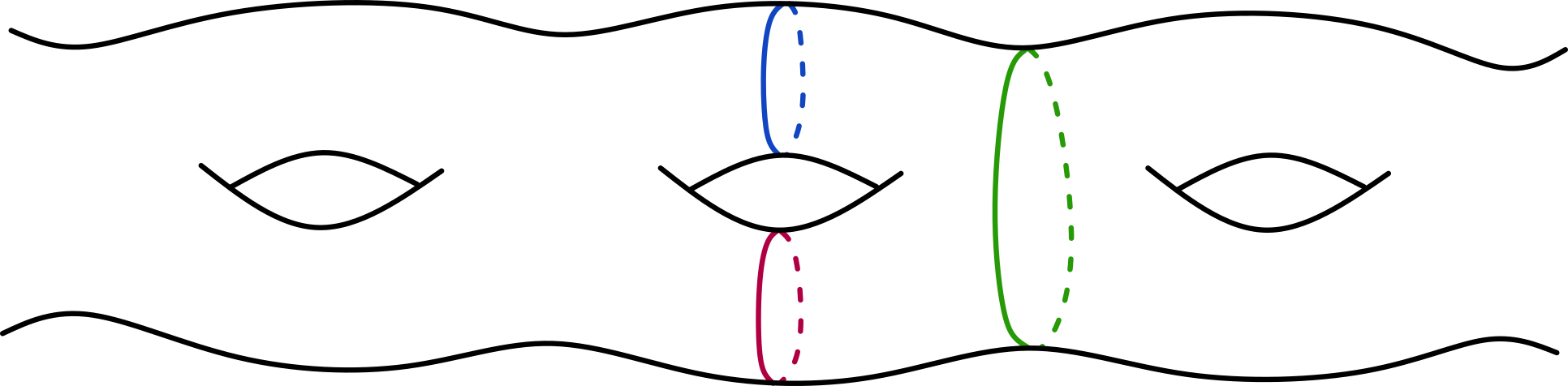}
\caption{The jointly separating case}
\end{figure}
\end{proof}

\begin{lemma}\label{homotopicingraph}
Let $a, b \in \fcn{k}{S_g}$ for $g \geq 2$. Then $a$ and $b$ are homotopic if and only if there exists a finite path of disjoint homotopic curves in $ \fcn{k}{S_g}$.
\end{lemma}

\begin{proof}
For the forward direction, let $a$ and $b$ be homotopic curves in $\fcn{k}{S_g}$.  So there is a smooth isotopy $H$ between $a$ and $b$ that takes $t\in [0,1]$ to a curve $c_t \in \fcn{k}{S_g}$.   For each $t \in [0,1]$, there is an open interval $(t_0, t_1)$ such that for any $s \in (t_0,t_1)$ the curves $c_t$ and $c_s$ are contained in the interior of some annulus.  Since $[0,1]$ is compact, then finitely many of these intervals cover the interval.  So there is a finite sequence $\{c_0 = a, c_1,  \ldots, c_{n-1}, c_n = b\}$ such that each pair $c_i$ and $c_{i+1}$ are contained in the interior of an annulus. 

Moreover, since $c_i$ and $c_{i+1}$ are closed subsets in $S_g$, then $S_g \setminus\{c_i, c_{i+1}\}$ is open.  So the boundary curves of this annulus can be made $C^k$ and not intersect $c_i$ or $c_{i+1}$.  Choose $d_i \in \fcn{k}{S_g}$ to be one of the boundary curves of the annulus containing $c_i$ and $c_{i+1}$ in its interior. Then $d_i$ is disjoint and homotopic to both $c_i$ and $c_{i+1}$.  Thus the sequence $\{c_0, d_0, c_1, d_1, \ldots, c_{n-1}, d_{n-1}, c_n\}$ is a path in $\fcn{k}{S_g}$ between $a$ and $b$ of disjoint homotopic curves.

For the reverse direction,  If there is a finite path of disjoint homotopic curves in $\fcn{k}{S_g}$, then there is a homotopy between each adjacent pair of curves in the path.  So all these homotopies can be reparameterized and done sequentially to get a single homotopy from $a$ to $b$. Thus $a$ and $b$ must be homotopic. 
\end{proof}

\medskip

\noindent \textit{Hulls of curves.  } The \emph{hull} of a collection of curves $\{\gamma_1, \ldots, \gamma_m\}$ in a closed surface $S$ is subset of $S$ consisting of the union of the $\gamma_i$ with all the connected components of $S{\setminus}\{\gamma_1, \ldots, \gamma_m\}$ which are topological disks.

\begin{lemma}\label{curvesinhull}
Let $\{\gamma_1, \cdots, \gamma_m\}$ be a collection of curves in $\fcn{k}{S}$.  Then a curve $c \in \fcn{k}{S}$ is contained in the hull of $\{\gamma_1, \ldots, \gamma_m\}$ if and only if every curve $b$ in $\fcn{k}{S}$ that is disjoint from every $\gamma_i$  is also disjoint from $c$. 
\end{lemma}

\begin{proof}
Let $c \in \fcn{k}{S}$ be in the hull of $\{\gamma_1, \ldots, \gamma_m\}$ and let $b$ be a curve in $\fcn{k}{S}$ that is disjoint from every $\gamma_i$.  Then $b$ must be contained in a component of the subsurface $S \setminus \{\cup \gamma_i\}$.  In particular, since $b$ is essential, it must be in a non-disk component. But then $b$ does not intersect the hull, and so it cannot intersect $c$.  

On the other hand, suppose that a curve $c \in \fcn{k}{S}$ is not in the hull of $\{\gamma_1, \ldots, \gamma_m\}$. Then there is some point $x \in c \cap (S\setminus\{\cup \gamma_i\})$ that is not in a disk component of $S\setminus\{\cup \gamma_i\}$.  So $x$ must be in an open non-disk component that contains a curve $b \in \fcn{k}{S}$ that contains $x$.  So $b$ is disjoint from the $\gamma_i$'s, but intersects $c$ at $x$.
\end{proof}

\begin{lemma}\label{autpreservesephomandhull}
Let $g \geq 2$ and $\alpha \in \aut\fcn{k}{S_g}$.  Then $\alpha$ preserves the set of separating curves, sets of jointly separating curves,  sets of homotopic curves and hulls of curves.
\end{lemma}

\begin{proof}
By Lemma \ref{sepcurveingraph},  the link of any separating curve $a \in \fcn{k}{S_g}$ is a join. Since $\alpha$ is a graph automorphism, then $link(\alpha(a))$ is also a join.  So $\alpha(a)$ is a separating curve.  The sets of jointly separating curves are preserved by Lemma \ref{jointlyseparatingingraph}.

Similarly by Lemmas \ref{disjointisotoptograph} and \ref{homotopicingraph}, any homotopic curves $a, b \in \fcn{k}{S_g}$ has a path of disjoint curves between them where the link of each pair is a join with one part containing only separating or nonseparating curves.  Thus $\alpha(a)$ and $\alpha(b)$ has a path of disjoint curves with this same structure.  So $\alpha(a)$ and $\alpha(b)$ are homotopic. 

Finally,  if $c \in \fcn{k}{S_g}$ is in the hull of $\{\gamma_1, \cdots, \gamma_m\} \subset \fcn{k}{S_g}$, then by Lemma \ref{curvesinhull}, any curves $b$ disjoint from the $\gamma_i$ are also disjoint from $c$.  But for any curve $d$ disjoint from $\alpha(\gamma_i)$, the curve $\alpha^{-1}(d)$ is disjoint from $\gamma_i$ and thus also disjoint from $c$.  So $d = \alpha(\alpha^{-1}(d))$ is disjoint from $\alpha(c)$.  So $\alpha(c)$ is contained in the hull of $\{\alpha(\gamma_1), \cdots, \alpha(\gamma_m)\}$. 
\end{proof}

Note that since any automorphism preserves the set of all essential curves, then this Lemma also implies that the sets of nonseparating curves and non-homotopic curves are also preserved by any automorphism of $\fcn{k}{S_g}$. 

\subsection{Recognizing subsurfaces with $\mathbf{\bfcn{k}{S}}$}\label{subsectionsubsurfaces}

The goal of this section is to prove Lemma \ref{autpreservessubsurfaces},  which shows that sets of curves bounding annuli, nonseparating pairs of pants, and one-holed tori are preserved.  We start with identifying the graph structures for each of these three different situations. 

\begin{lemma}\label{containedinannulus}
Let $A$ be an annulus in $S_g$, $g \geq2$,  with boundaries $\gamma_1$ and $\gamma_2 \in \fcn{k}{S_g}$. Then $c \in \fcn{k}{S_g}$ is contained in $A$ if and only if for any $e \in \fcn{k}{S_g}$ not homotopic to $\gamma_1$, if $e$ intersects $c$, then $e$ must intersect $\gamma_1$ or $\gamma_2$.
\end{lemma}

\begin{proof}
Let $c \in \fcn{k}{S_g}$ be a curve contained in $A$ and let $e \in \fcn{k}{S_g}$ be a curve not homotopic to $\gamma_1$. Since any essential curves in an annulus must be homotopic to the boundary curves, then $e$ cannot be contained in the interior of $A$.  But since $c$ is contained in $A$, then if $e$ intersects $c$,  it must enter $A$ through one of the boundary curves.  So $e$ must intersect either $\gamma_1$ or $\gamma_2$. 

On the other hand, suppose that $c$ is not contained in $A$. Then there is some point $x \in c$ that is not contained in $A$.  If $A$ is nonseparating,  then $S_g \setminus A$ is a connected subsurface with genus $g-1$ and two boundary components.  Since $g \geq 2$,  then this subsurface will contain essential curves.  Similarly, if $A$ is separating, then both components of $S_g \setminus A$ will contain at least one genus and so both also contain essential curves. So is either case,  there exists $e \in \fcn{k}{S_g}$ that is not homotopic to $\gamma_1$ and do not intersect $\gamma_1$ or $\gamma_2$, but intersects $c$ at the point $x$. 
\end{proof}

\begin{lemma}\label{pantsingraph}
Let $S_g$ with $g \geq 2$. Let $\gamma_1, \gamma_2, \gamma_3 \in \fcn{k}{S_g}$ be disjoint non-homotopic nonseparating curves. Then the $\gamma_i$ bound a pair of pants if and only if $\{\gamma_1, \gamma_2, \gamma_3\}$ jointly separate $S_g$ with a component that only contains nonseparating curves. 
\end{lemma}

\begin{proof}
Suppose that the $\gamma_i$ bound a pair of pants.  So they jointly separate the surface into at least two components. Moreover, any essential curve in a pair of pants will be homotopic to one of the boundary components. Since the $\gamma_i$ are all nonseparating, then all the other essential curves contained in the pair of pants must also be nonseparating. 

On the other hand, suppose that there are disjoint non-homotopic nonseparating $\gamma_1, \gamma_2, \gamma_3 \in \fcn{k}{S_g}$ such that they jointly separate $S_g$ with a component that only containes nonseparating curves.  Call the subsurface corresponding to the side containing only nonseparating curves $P$.  Since each of the $\gamma_i$ is nonseparating, then each component of $P$ must have either 2 or 3 of the $\gamma_i$ as boundary curves.  Moreover,  no components of $P$ can have any genus since otherwise there would be a curve separating the genus from the boundaries in that component.  But any component with only 2 boundaries would then be an annulus, which contradicts the assumption that the $\gamma_i$ are non-homotopic.  Thus $P$ must be a pair of pants.  
\end{proof}

Let $S_{g}^b$ denote the surface of genus $g$ with $b$ disjoint open disks removed such that each boundary is a $C^k$ curve. 

\begin{lemma}\label{oneholedtorusingraph}
Let $S_g$ with $g \geq 2$ and let $\gamma \in \fcn{k}{S_g}$ be a separating curve.  Then $\gamma$ bounds an $S_1^1$ subsurface if and only if one component of $S_g \setminus \gamma$ only contains curves that are either homotopic to $\gamma$ or nonseparating. 
\end{lemma}

\begin{proof}
Let $\gamma$ be a separating curve that bounds an $S_1^1$ subsurface.  Since a torus only contains nonseparating essential curves, then the only separating curves on a $S_1^1$ subsurface that are essential in $S_g$ must be homotopic to the boundary curve. $\gamma$. 

Now suppose that $\gamma$ is a separating curve that does not bound a $S_1^1$ subsurface. Then there will be at least two genus in each of the two components of $S_g \setminus \gamma$.  So each side contains a separating curve disjoint from $\gamma$ that separates one genus from the rest. Thus this separating curve is not homotopic to $\gamma$. 
\end{proof}

Now, we can combine the previous three results to get the final lemma for this section. 

\begin{lemma} \label{autpreservessubsurfaces}
Let $g \geq 2$ and $\alpha \in \aut\fcn{k}{S_g}$.  Then $\alpha$ preserves sets of curves that bound annuli, nonseparating pairs of pants,  and one-holed tori in $S_g$.
\end{lemma}

\begin{proof}
Note that in Lemmas \ref{containedinannulus},  \ref{pantsingraph}, and \ref{oneholedtorusingraph}, the boundary curves are characterized by some combination of whether curves are separating, homotopic,  or intersecting.  Any $\alpha \in \aut \fcn{k}{S_g}$ preserves whether curves intersect and by using Lemma \ref{autpreservesephomandhull},  we know that sets of separating and homotopic curves are also preserved by $\alpha$.  Thus the image of these boundary curves will have the exact same characteristics, and so they must also bound the same type of subsurface. 
\end{proof}

\subsection{Torus pairs}\label{subsectiontorus}

Two curves $a$ and $b$ are called a \emph{torus pair} if they intersect one another in a single topologically transverse intersection. This intersection is called \emph{degenerate} if it consists of a single point and \emph{nondegenerate} otherwise. In this section, we prove Lemma \ref{autpreservestoruspairs}, showing that the automorphisms of the $C^k$-curve graph preserve the sets of degenerate and nondegenerate torus pairs.

Note that any torus pair will necessarily consist of non-homotopic, nonseparating curves. The name is suggestive since any torus pair fills an $S_1^1$ subsurface.

\begin{figure}[h]
\centering
\includegraphics[width=2in]{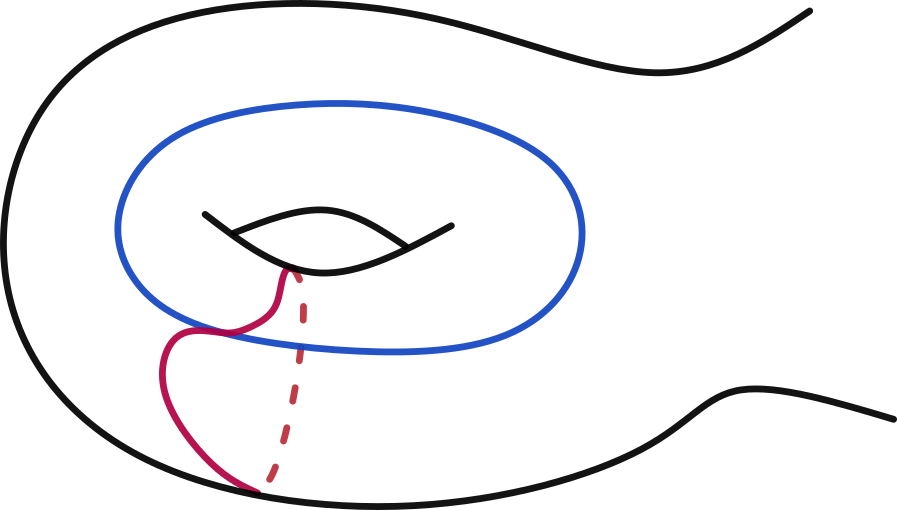}
\caption{A torus pair}
\end{figure}

While there is a similar result in the proof of Long--Margalit--Pham--Verberne--Yao \cite[Lemma 2.5]{LMPVY},  their graph structure does not exist in the $C^k$-curve graph.  Specifically,  they utilize the existence of the symmetric difference as a third curve to distinguish between nondegenerate and degenerate torus pairs.  But the symmetric difference cannot be a $C^k$ curve (not even $C^1$), and so it is not a curve in our graph.  Instead, we can find a pair of $C^k$ curves in the neighborhood of the symmetric difference to distinguish between nondegnerate and degenerate torus pairs.  

We start this section with the following two lemmas that are necessary to characterize torus pairs in Lemma \ref{crossintersectioningraph}.

\begin{lemma}\label{disksinpants}
Let $P$ be a pair of pants with boundary curves $\gamma_1, \gamma_2,$ and $\gamma_3$.  Let $A = \{\alpha_i\}$ be a disjoint collection of arcs with endpoints in $\gamma_1$ and $\gamma_2$ that have interiors disjoint from the $\gamma_i$.  Then there is a topological disk in $P$ bounded by some subcollection of subarcs of $\gamma_1, \gamma_2$ and the $\alpha_i$'s unless $A$ is empty or consists of only one arc, which either goes between $\gamma_1$ and $\gamma_2$ or separates two of the $\gamma_i$'s.
\end{lemma}

\begin{proof}
The structure of this proof involves considering several possible collections of arcs and finding the topological disk in all cases except the two cases given in the statement of the lemma.

Now suppose $A$ only contains arcs which have endpoints in the same boundary component. Since $P$ is an oriented surface with no genus, any of these arcs, say $\alpha$ with endpoints in $\gamma_2$,  will separate $P$ into two subsurfaces.  If one of these subsurfaces contains both $\gamma_1$ and $\gamma_3$, then there is a subsurface bounded by $\alpha$ and a subarc of $\gamma_2$ that contains no other boundary components. Since $P$ has no genus, then this subsurface must be a disk. 

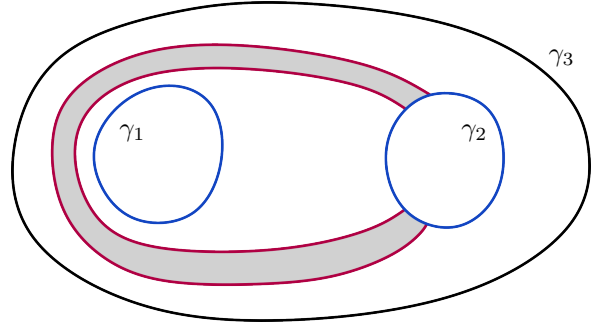
\begin{figure}[h]
\centering
\begin{tikzpicture}
\small
    \node[anchor=south west, inner sep = 0] at (0,0){\includegraphics[width=3in]{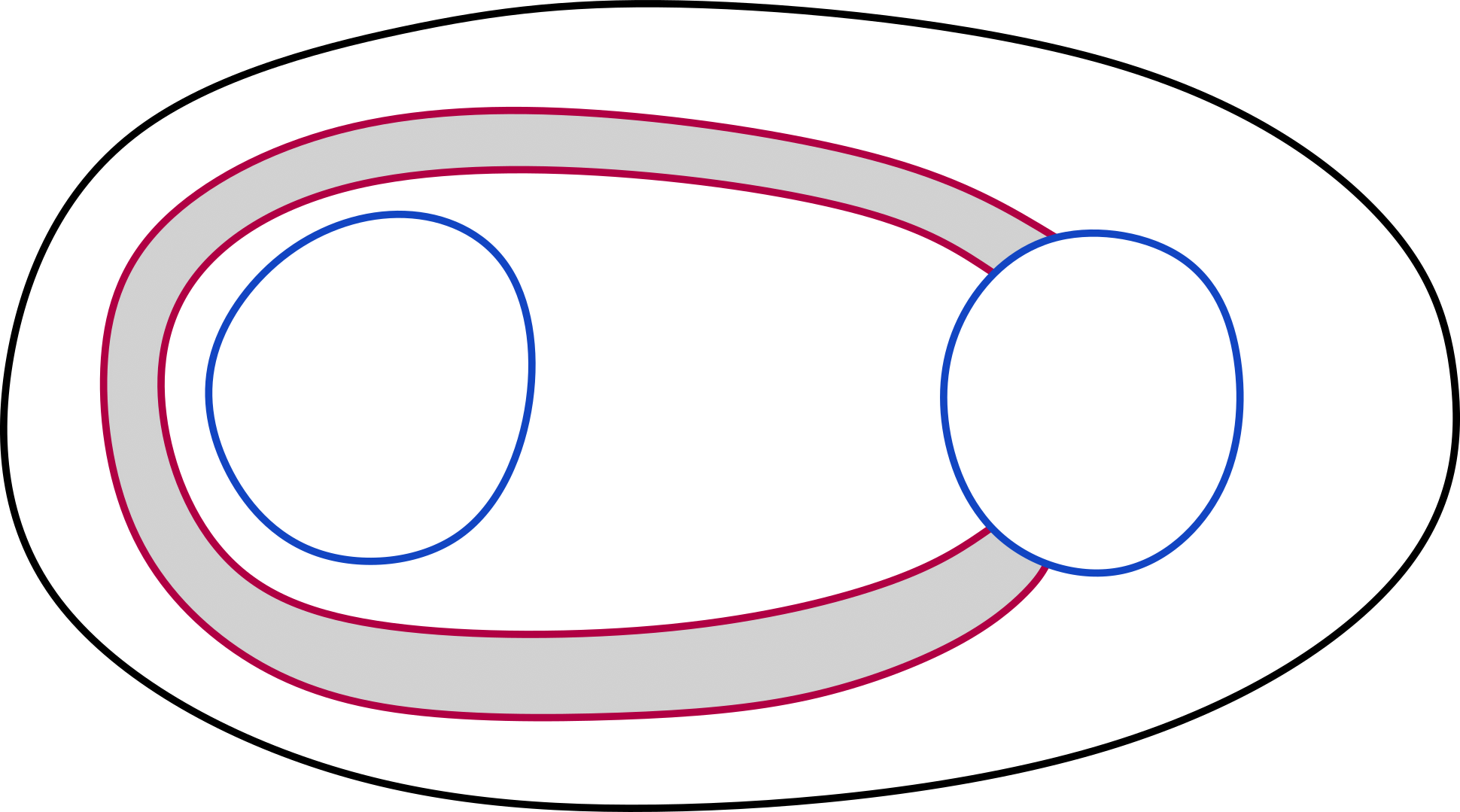}};
    \node at (1.6,2.5) {$\gamma_1$};
    \node at (6.1,2.5) {$\gamma_2$};
    \node at (7.25,3.5) {$\gamma_3$};
\end{tikzpicture}
\caption{The disk created when arcs have endpoints in the same boundary}
\end{figure}

On the other hand,  suppose that every $\alpha \in A$ separates $\gamma_1$ and $\gamma_3$. Cutting along $\alpha$ results in two annuli,  each with one boundary from $\alpha$ and a subarc of $\gamma_2$ and the other boundary from one of the $\gamma_i$.  Thus if $A$ contains another arc $\beta$, then it must cut one of these annuli into two subsurfaces where only one can contain the boundary $\gamma_i$. So the other subsurface will be a disk. 

\begin{figure}[h]
\centering
\begin{tikzpicture}
\small
    \node[anchor=south west, inner sep = 0] at (0,0){\includegraphics[width=3in]{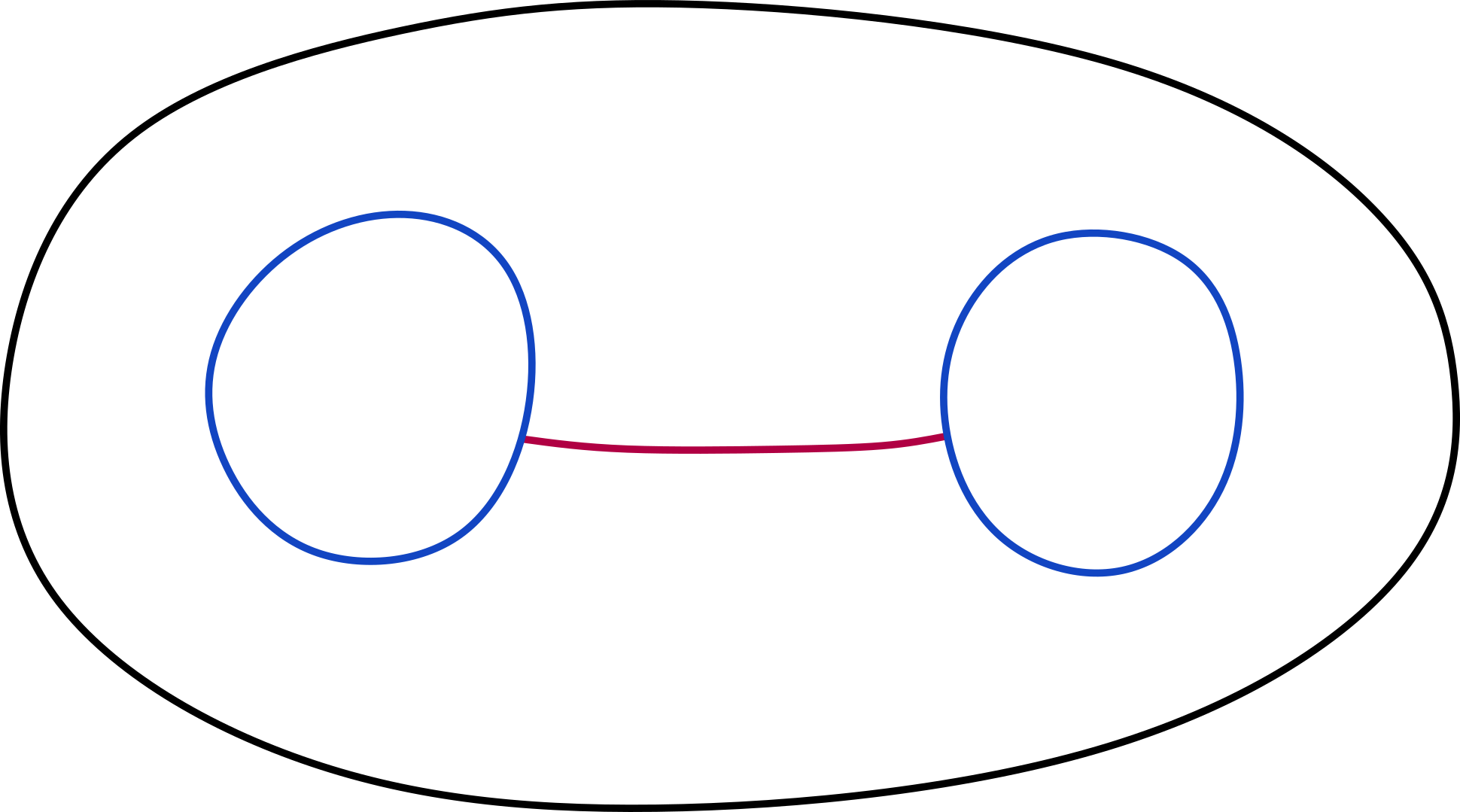} };
    \node at (1.5,2) {$\gamma_1$};
    \node at (6.1,2) {$\gamma_2$};
    \node at (7.15,3.5) {$\gamma_3$};
    \node at (4.1, 2.1) {$\alpha$};
    \node[anchor=south west, inner sep = 0] at (10, -.5) {\includegraphics[width=2in]{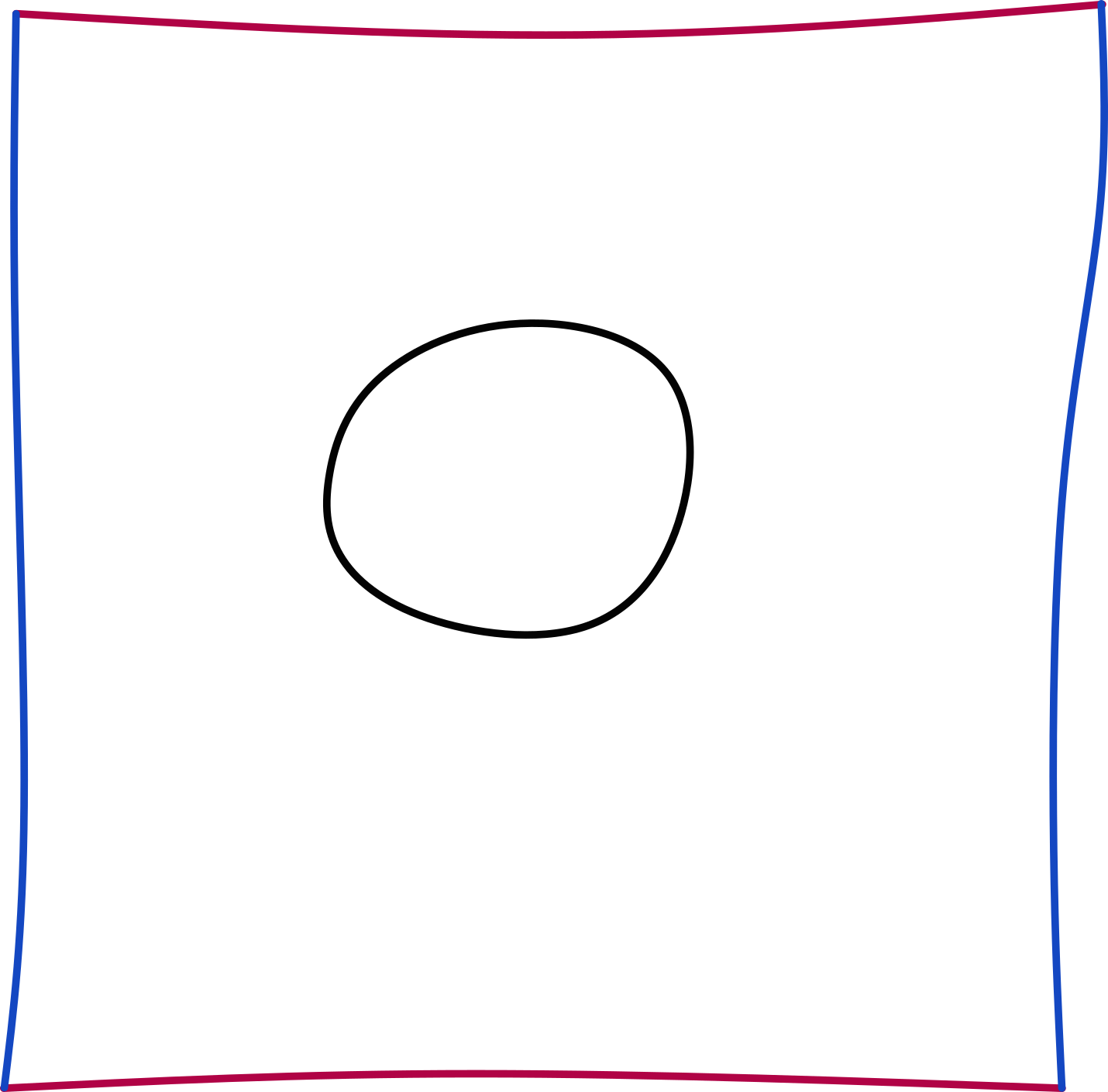}};
    \node at (9.8,3) {$\gamma_1$};
    \node at (15.1,1.5) {$\gamma_2$};
    \node at (12,2.5) {$\gamma_3$};
    \node at (12, -.2) {$\alpha$};
    \node at (13, 4.6) {$\alpha$};
\end{tikzpicture}
\caption{The first step of the two arc case, before (left) and after (right) cutting along $\alpha$}
\end{figure}
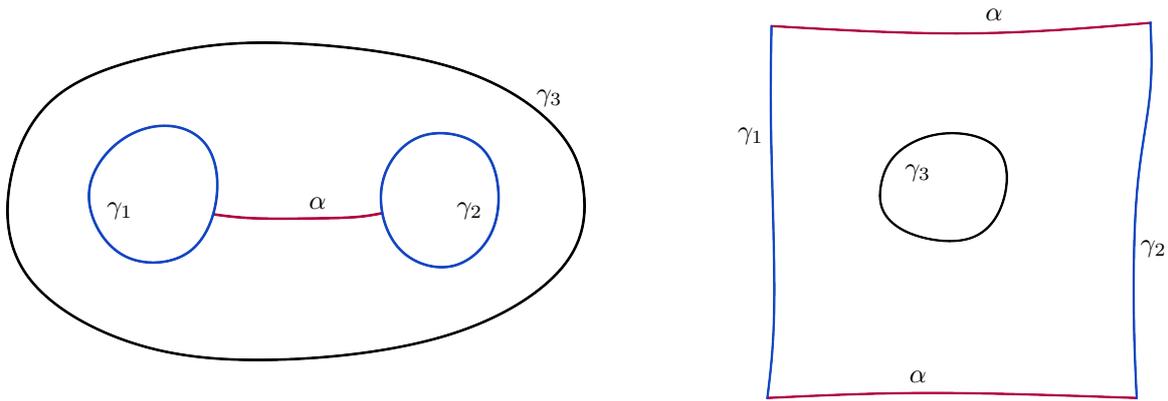

Finally, consider a collection $A$ with at least two arcs, one of which goes between $\gamma_1$ and $\gamma_2$.  By cutting along $\alpha$, the resulting surface will be a square with the boundary curve $\gamma_3$ inside. But then any other arc must separate this square into two subsurfaces, one of which will not contain $\gamma_3$, and so must be a disk. 
\end{proof}

\begin{lemma}\label{2^ncurvesinunionsameorient}
If two $C^k$ curves $\alpha$ and $\beta$ have $n\geq1$ points or intervals where there exists $C^k$ parameterizations of $\alpha$ and $\beta$ with nonzero tangent vectors that agree up to the $k$th derivative and all the intersections occur in the same order and with consistent orientations, then there are $2^n$ distinct $C^k$ curves in $\alpha \cup \beta$.
\end{lemma}

\begin{proof}
Pick one of the intersections or intervals and name it $x_1$.  Following an orientation of $\alpha$, continue naming the intersections and overlapping intervals $x_2, \ldots, x_n$.  Starting at $x_1$, there are two possible choices that a $C^k$ curve could make, either to follow $\alpha$ or to follow $\beta$, to get to $x_2$.  Both of these choices will result in a $C^k$ curve since there are parameterizations of $\alpha$ and $\beta$ that agree up to the $k$th derivative at $x_1$.  Similarly at $x_2$, the curve could make two choices again, each being distinct. This process can be continued until $x_n$, where the final choice will bring the curve back to $x_1$ with the correct orientation.  So there are 2 choices at each of the $n$ intersections or overlapping intervals, with each choice giving a different curve. Thus there at least $2^n$ distinct curves in $\alpha \cup \beta$. 

Note that as soon as a curve in $\alpha \cup \beta$ intersects one of the $x_i$, it must intersect all of them.  Thus there are no further curves in the union that intersect some of the $x_i$, but not $x_1$.  Moreover, there are no curves in the union that do not intersect any of the $x_i$,  since otherwise it is contained in an arc of either $\alpha$ or $\beta$. Thus there are exactly $2^n$ distinct curves in $\alpha \cup \beta$. 
\end{proof}

We can now state a combinatorial characterization for torus pairs.

\begin{lemma}\label{crossintersectioningraph}
Let $a$ and $b$ be non-homotopic nonseparating curves in $\fcn{k}{S_g}$, $g \geq 2$. Then $(a,b)$ is a torus pair if and only if the hull of $\{a, b\}$ contains no other vertices of $\fcn{k}{S_g}$ and there is a separating curve $\gamma \in \fcn{k}{S_g}$, disjoint from $a$ and $b$,  such that one side of $\gamma$ is an $S_1^1$ subsurface that contains $a$ and $b$. 
\end{lemma}

\begin{proof}
If $a$ and $b$ intersect only at one topologically transverse intersection, then $a$ and $b$ cannot bound any bigons.  Thus the hull of $\{a, b\}$ is exactly $a \cup b$. Even if the intersection agrees up to the $k$th derivative,  by Lemma \ref{2^ncurvesinunionsameorient}  there are only 2 curves in $a \cup b$, namely $a$ and $b$.  If the intersection does not agree up to the $k$th derivative,  then there is no way for a $C^k$ curve in $a \cup b$ to switch from traveling on an arc in $a$ to an arc in $b$.   So the only $C^k$ curves in the hull will be the original curves $a$ and $b$.  The boundary of a regular neighborhood of $a \cup b$ can be smoothed to make a separating curve $\gamma \in \fcn{k}{S_g}$, which will separate an $S_1^1$ subsurface containing $a$ and $b$ from the rest of $S_g$.

Now suppose that $a$ and $b$ are non-homotopic, nonseparating curves on an $S_1^1$ subsurface $T$ bounded by a $C^k$ curve $\gamma$ and the hull of $\{a, b\}$ contains only the vertices $a$ and $b$.  Any two non-homotopic nonseparating curves on a torus must cross each other, so $a$ and $b$ must have at least one topologically transverse intersection.  Since the hull only contains two vertices, then $a$ and $b$ cannot mutually bound any disks  in $S_g$.  Since $a$ is nonseparating, then $\overline{T \setminus \{a\}}$ is a pair of pants with arcs of $b$ going between the two boundary curves that correspond to $a$.  So Lemma \ref{disksinpants} will imply that there must be at most one subarc of $b$ with non-empty interior in $T\setminus \{a\}$.  This subarc either goes from one copy of $a$ to the other copy, or has endpoints in the same copy of $a$ and separates the other copy and $\gamma$.  But in the latter case, $b$ would not have any topologically transverse intersections with $a$.  In fact, it would be homotopic to either $a$ or $\gamma$. But this contradicts our initial assumptions. Thus $b\setminus a$ will consist of a single arc from one side of $a$ to the other side. So $a$ and $b$ are a torus pair.  
\end{proof}

We now give the graph structure that is unique to nondegenerate torus pairs.

\begin{lemma}\label{nondegenerateintervalingraph}
Let $g \geq 2$ and let $(a,b)$ be a torus pair in $\fcn{k}{S_g}$ whose intersection arc is denoted by $c$. Then $(a,b)$ is nondegenerate ($c$ is not a single point) if and only if there exist disjoint homotopic curves $d_1$ and $d_2$ in $\fcn{k}{S_g}$ such that $(a,d_1), (a, d_2), (b,d_1), $ and $(b,d_2)$ are all torus pairs and for any curve $e \in \fcn{k}{S_g}$ disjoint and non-homotopic to $d_1$ and $d_2$, then $e$ intersects $a$ if and only if $e$ intersects $b$.
\end{lemma}

\begin{proof}
For the forward direction, consider a regular neighborhood of the symmetric difference $a \Delta b$ that does not intersect itself. Then it is possible to find two disjoint homotopic $C^1$ curves $d_1$ and $d_2$ in this neighborhood such that the annulus defined by $d_1$ and $d_2$ still contains $a \Delta b$.   Since both $d_1$ and $d_2$ are closed in $S_g$ and do not intersect, then there is an open component of $c$ outside of the annulus bounded by $d_1$ and $d_2$.   Let $e$ be an essential curve that is disjoint from both $d_1$ and $d_2$,  and not homotopic to either $d_1$ and $d_2$.  Thus $e$ cannot intersect anything in the interior of this annulus.  And so the only point that $e$ could intersect either $a$ or $b$ would be along the arc $c$ where it must intersect both $a$ and $b$ at the same time.

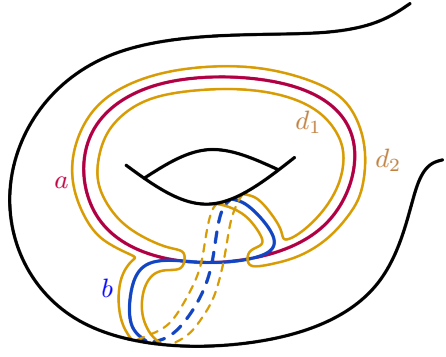
\begin{figure}[h]
\centering
\begin{tikzpicture}
\small
    \node[anchor=south west, inner sep = 0] at (0,0){\includegraphics[width=2.25in]{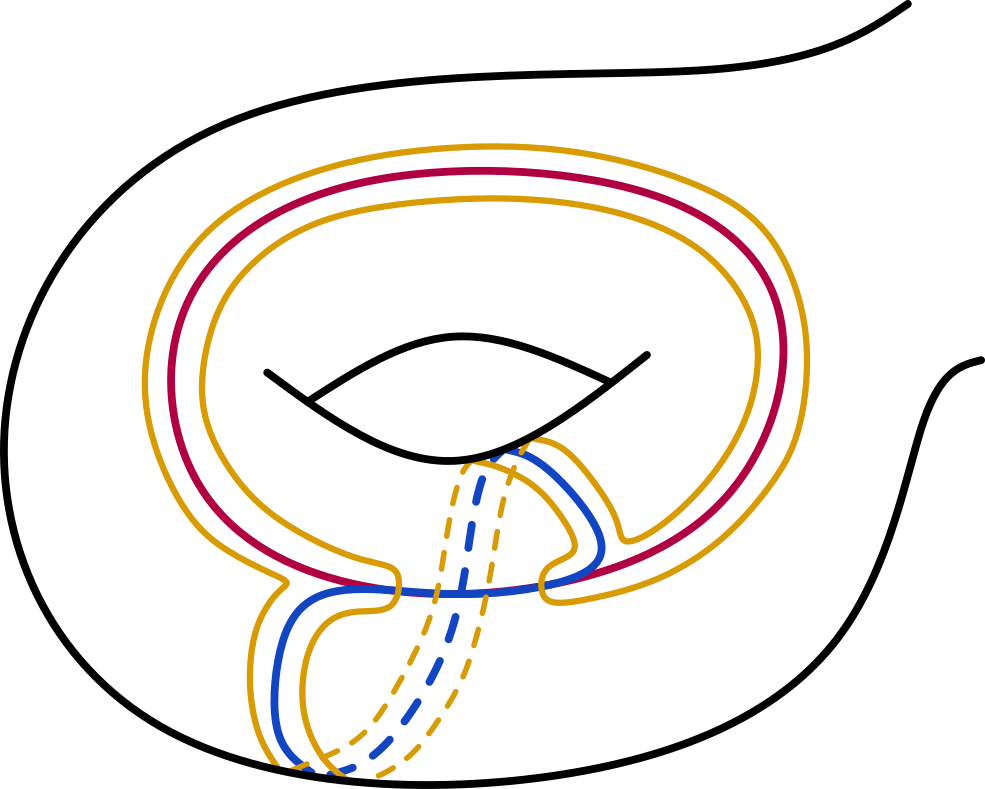} };
    \node at (.7,2.2) {\color{purple}$a$};
    \node at (1.3,.8) {\color{blue}$b$};
    \node at (3.95,3) {\color{brown} $d_1$};
    \node at (5, 2.5) {\color{brown} $d_2$};
\end{tikzpicture}
\caption{A nondegenerate torus pair and curves bounding the symmetric difference $a \Delta b$. }
\end{figure}

To prove the reverse direction, suppose that the arc $c$ is a single point and let $d_1$ and $d_2$ be  disjoint homotopic curves in $\fcn{k}{S_g}$ that each intersect $a$ and $b$ in a single topologically transverse interval.  First, note that $d_1$ and $d_2$ must be nonseparating curves.  Let $A$ denote the annulus bounded by the $d_i$. Thus $S_g \setminus A$ will be a single connected surface with two boundaries and at least one genus.  Moreover,  since $(a,d_1),(a,d_2),(b,d_1)$, and $(b, d_2)$ are all torus pairs and the $d_i$ are disjoint, then the components of $a$ and $b$ in $S_g \setminus A$ will each be a single arc with nonempty interior going from one of the boundaries to the other, intersecting in at most one point.  But then $\overline{S_g \setminus (A \cup a)}$ will still be one connected surface with one boundary and at least one genus. The interior of this subsurface will contain either one or two subarcs of $b$ with nonempty interior.  So there exist $e \in \fcn{k}{S_g}$ that live on this subsurface, disjoint from $a, d_1$, and $d_2$ that will intersect $b$. 
\end{proof}

\begin{lemma}\label{autpreservestoruspairs}
Let $g \geq 2$ and $\alpha \in \aut\fcn{k} {S_g}$.  Then $\alpha$ preserves the sets of degenerate and nondegenerate torus pairs.
\end{lemma}

\begin{proof}
Lemma \ref{crossintersectioningraph} gives a characterization of any torus pair $(a,b)$ using disjointness, the hull,  homotopies,  separating curves,  and the subsurface $S_1^1$.  For any $\alpha \in \aut \fcn{k}{S_g}$,  the disjointness will be preserved; the hull, homotopies, and separating curves will all be preserved by Lemma \ref{autpreservesephomandhull}; and the subsurface $S_1^1$ will be preserved by Lemma \ref{autpreservessubsurfaces}.  Thus $(\alpha(a), \alpha(b))$ will also be a torus pair. So $\alpha$ will preserve the set of torus pairs. 

By Lemma \ref{nondegenerateintervalingraph},  any nondegenerate torus pair can be characterized similarly using torus pairs, homotopies, and disjointness.  Thus the set of nondegenerate torus pairs will also be preserved by $\alpha$.  

Since the set of degenerate torus pairs is exactly the set of torus pairs without the nondegenerate ones, and both of these sets are preserved by $\alpha$,  then the set of degenerate torus pairs is preserved by $\alpha$. 
\end{proof}

\subsection{$\mathbf{k}$-Smooth pairs}\label{subsectionsmoothpair}

In this section, we define $k$-smooth pairs and show that they are preserved by automorphisms of the $C^k$-curve graph.

Let $P$ be a nonseparating pair of pants in $S$. Two non-homotopic curves $a,  b \in \fcn{k}{S}$ form a \emph{$k$-smooth pair in $P$} if $a \cup b \subset P$ and there exists a curve $d\in \fcn{k}{S}$, with the following properties:
\begin{enumerate}[label=(\roman*)]
\item $d$ is not homotopic to either $a$ or $b$
\item $d= a_1 \cup b_1 \cup c_1 \cup c_2$ such that $a_1 \subset a$,  $b_1 \subset b$, and $c_1 \cup c_2 \subset  a \cap b$ are all arcs
\item the interior of $a_1$ does not intersect $b$ and the interior of $b_1$ does not intersect $a$
\item $a \cup b$ is contained in a pair of pants $P_d \subset P$ which has $d$ as one of its boundary components.
\end{enumerate}

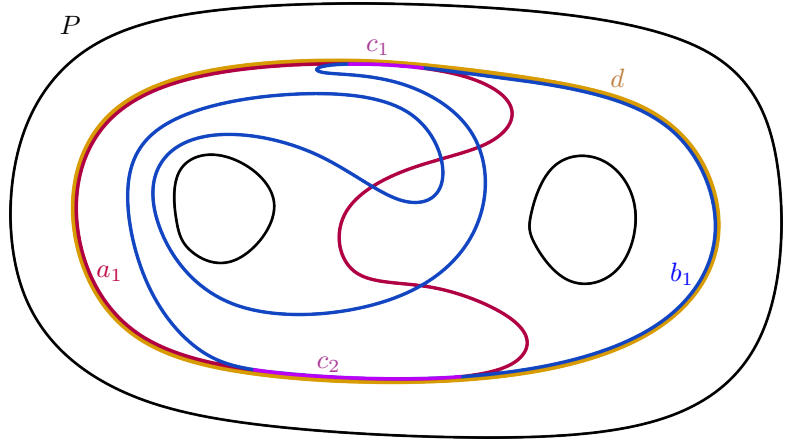
\begin{figure}[h]
\centering
\begin{tikzpicture}
\small
    \node[anchor=south west, inner sep = 0] at (0,0){\includegraphics[width=4in]{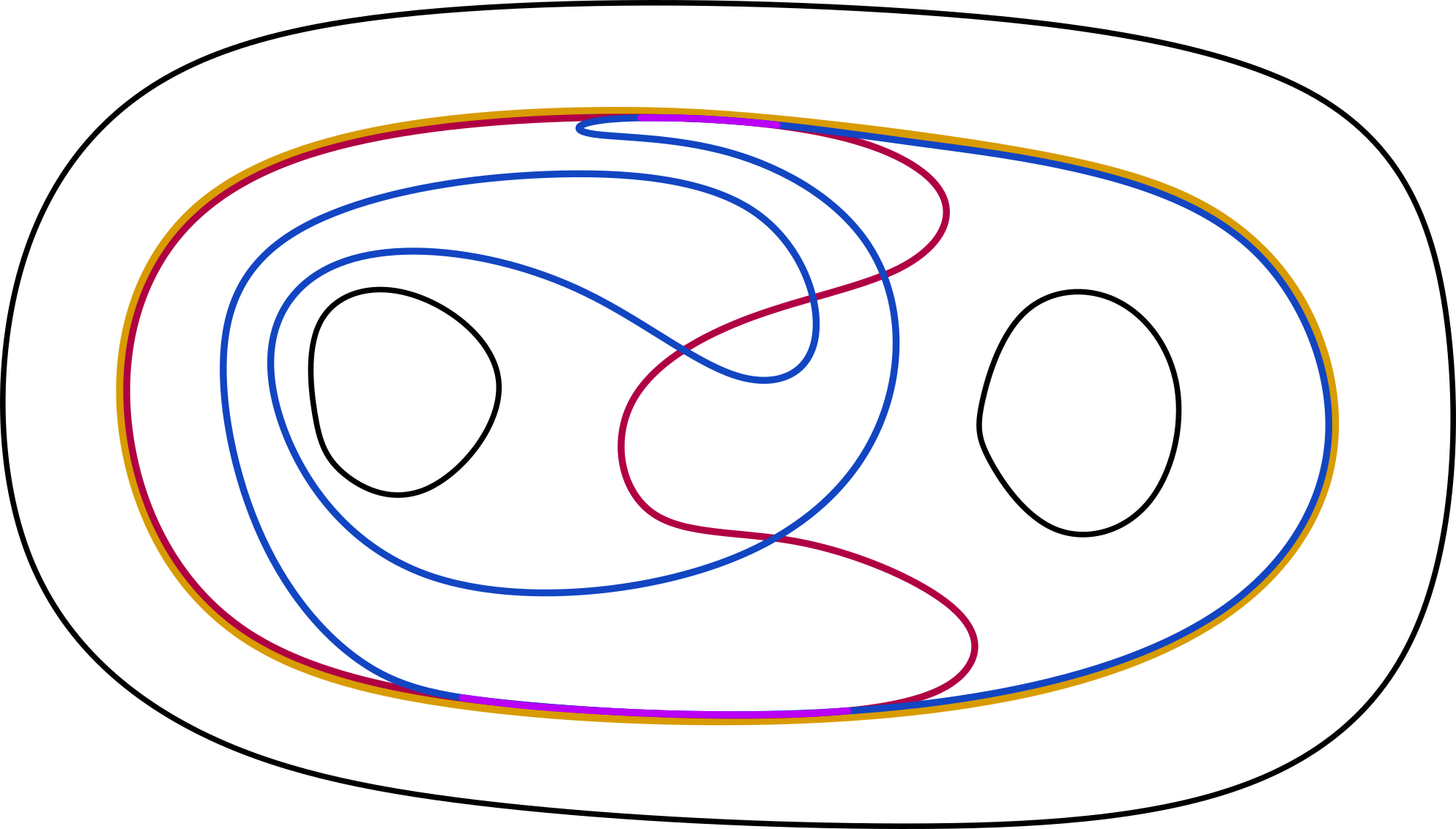} };
    \node at (1.32,2.2) {\color{purple}$a_1$};
    \node at (8.85,2.2) {\color{blue}$b_1$};
    \node at (4.85,5.2) {\color{Mulberry}$c_1$};
    \node at (4.2, 1) {\color{Mulberry}$c_2$};
    \node at (.8, 5.5) {$P$};
    \node at (8, 4.8) {\color{brown} $d$};
\end{tikzpicture}
\caption{A $k$-smooth pair}
\end{figure}

We say that $d$ is \emph{determined} by $(a, b)$.
The following lemma gives the unique structure that identifies $k$-smooth pairs in our graph.

\begin{lemma}\label{smoothpairingraph}
Let $g \geq 2$, $a, b \in \fcn{k}{S_g}$ be non-homotopic curves in a nonseparating pair of pants $P$ with boundaries $\gamma_1, \gamma_2$ and $\gamma_3 \in \fcn{k}{S_g}$.  Suppose that $d \in \fcn{k}{S_g}$ is in the hull of $\{a,b\}$ such that $d$ is not homotopic to $a$ or $b$ and $a \cup b$ is contained in the pair of pants $P_d$ with boundaries $\gamma_1, \gamma_2, $ and $d$.  Then $(a,b)$ form a $k$-smooth pair that determines $d$ if and only if any $e \in \fcn{k}{S_g}$ that intersects $\gamma_3$ and $d$, but not $\gamma_1$ or $\gamma_2$ satisfies one of the following:
\begin{enumerate}[noitemsep,topsep=0pt, label=$(\roman*)$]
\item $e$ intersects $a$ but not $b$ $\Rightarrow$ there is an $a' \in \fcn{k}{S_g}$ contained in $P_d$,  homotopic to $a$,  such that $a'$ intersects $e$, but not $b$
\item $e$ intersects $b$ but not $a$ $\Rightarrow$ there is an $b'\in \fcn{k}{S_g}$ contained in $P_d$,  homotopic to $b$,  such that $b'$ intersects $e$, but not $a$
\item $e$ intersects both $a$ and $b$
\end{enumerate}
\end{lemma}

Before proving the Lemma~\ref{smoothpairingraph},  we will state and prove two technical lemmas related to the existence of $C^k$ curves and arcs in subsurfaces. 

\begin{lemma}\label{c1curvestoboundaries}
Let $L$ be a connected subsurface of $S$ with at least 3 boundaries $\gamma_i \in \fcn{k}{S}$.  Let $a \in \fcn{k}{S}$ be in the interior of $L$ and homotopic to $\gamma_1$.  If $b \in \fcn{k}{S}$ is in the interior $L$ and is not homotopic to $\gamma_1$, then for each $i \neq 1$, there is an $e_i \in \fcn{k}{S}$ such that $e_i$ intersects $b$ and $\gamma_i$, but does not intersect $a$. 
\end{lemma}

\begin{proof}
First, note that $a$ and $\gamma_1$ bound an open annulus in $L$. Call this annulus $A$.  Then $L\setminus A$ is a connected subsurface with the same number of boundaries as $L$.  Any curves that are contained in $\bar{A}$ will either be homotopic to $\gamma_1$ or inessential. Thus if $b \in \fcn{k}{S}$ contained in $L$ and not homotopic to $\gamma_1$,  then it cannot be contained in $\bar{A}$, and so there is some point $x \in b$ such that $x \in L \setminus \bar{A}$.  

Since $L \setminus \bar{A}$ is connected, then for any $i \neq 1$, there is an arc $c_i$ from $\gamma_i$ to $x$ that does not intersect $a$.  Also,  since $\bar{A}$ is a closed set in $L$, there is an open neighborhood of $c_i$ that does not intersect $\bar{A}$.  So there is enough space to smoothly isotope $\gamma_i$ along $c_i$ to make $e_i \in \fcn{k}{S}$ that intersects both $b$ and $\gamma_i$ without intersecting $a$.  
\end{proof}

\begin{lemma}\label{smoothingarcsatboundary}
Let $M$ and $L$ be two subsurfaces of $S$ with disjoint interiors $\mathring{M}$ and $\mathring{L}$.  Let $m$ be a boundary component of $M$ and $l$ be a boundary component of $L$ that intersect at a point $x$ and are both $C^k$ in a neighborhood of $x$.  Then any arc with interior in $\mathring{M} \cup \mathring{L} \cup \{x\}$ can be isotoped in $M \cup L$ to a $C^k$ arc with the same endpoints.
\end{lemma}
\begin{proof}
If the interior of an arc $e$ is contained completely within the interior of $M$, then there is an open neighborhood of the interior of $e$ contained with $M$. So any portions of $e$ that are not $C^k$ can be smoothed out within this open neighborhood to make an $C^k$ arc isotopic relative boundary to $e$. 

Now suppose that the interior of $e$ contains the point $x$.  If the $e \setminus\{x\}$ is contained in $M$, then consider the boundary $m$ near the point $x$.  Since it is $C^k$, then in a local chart near $x$,  there is a $C^k$ diffeomorphism that takes $k$ to a straight line with the interior of $M$ on one side. So $e$ can be smoothly isotoped in the chart near $x$ to match the curve $k$ near $x$ and then smoothly veer back into the interior of $M$. The resulting arc will be $C^k$ while staying completely within $M$.  

So any arc whose interior is contained in $\mathring{M} \cup \{x\}$ can be isotoped to a $C^k$ arc in $M$ with the same endpoints.  By symmetry, the same will hold for any arc whose interior is contained in $\mathring{L} \cup \{x\}$.

Now consider an arc $e$ that passes between $M$ and $L$ at $x$.  Since both boundaries $m$ and $l$ are $C^k$ near $x$ and the interiors of $M$ and $L$ are disjoint, then $m$ and $l$ must meet at a one-sided tangent intersection. But then they will have the same tangent direction at $x$. So any other tangent direction at $x$ will be transverse to both $m$ and $l$.  So $e$ can be isotoped within $M$ and $L$ so that the tangent lines are the same as the arc approaches $x$ from both sides.  Moreover, these arcs can have any other higher derivatives and still be contained in $M \cup L$, so we can isotoped these arcs so that the resulting arc will be $C^k$ at $x$. 
\end{proof}

We are now ready to complete the proof of Lemma \ref{smoothpairingraph}.

\begin{proof}
We start with the forward direction.  Let $(a,b)$ form a $k$-smooth pair that determines $d$. Then $d = a_1 \cup b_1 \cup c_1 \cup c_2$ such that $a_1$ is an arc of $a$ whose interior does not intersect $b$,  $b_1$ is an arc of $b$ whose interior does not intersect $a$ and $c_1, c_2$ are (possibly degenerate) arcs in $a \cap b$.  Thus $b_2 = b \setminus b_1 \cup c_1 \cup c_2$ is an arc with the same endpoints as $a_1$, but cannot intersect the interior of $a_1$.  Moreover, this arc must be contained in $P_d$.  So it is either homotopic to $a_1$, homotopic to $b_1 \cup c_1 \cup c_2$, or separates $\gamma_1$ and $\gamma_2$ in $P_d$. Since $b$ is not homotopic to $d$, then $b_2$ cannot be homotopic to $a_1$. Since $b$ is essential in $P$, then $b_2$ cannot be homotopic to $b_1 \cup c_1 \cup c_2$.  Thus $b_2$ separates $\gamma_1$ and $\gamma_2$ in $P_d$. So the surface bounded by $a_1 \cup b_2$ and $\gamma_1$ must be an annulus.  Denote this subsurface by $A_1$. 

\begin{figure}[h]
\centering
\begin{tikzpicture}
\small
    \node[anchor=south west, inner sep = 0] at (0,0){\includegraphics[width=4in]{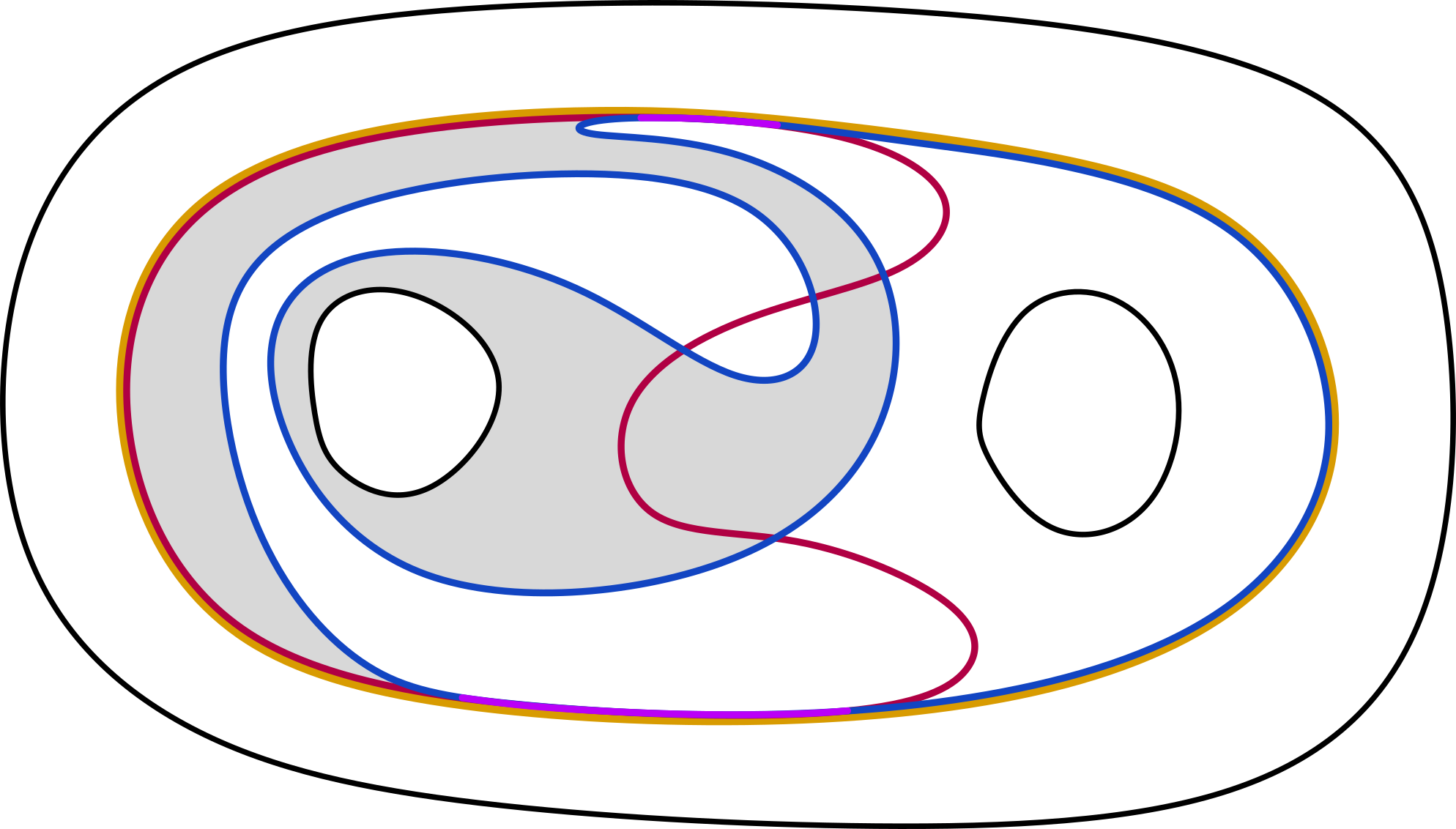} };
    \node at (.8,2.2) {\color{purple}$a_1$};
    \node at (6.3,2.5) {\color{blue}$b_2$};
    \node at (3.7,2.3) {$A_1$};
    \node at (.8, 5.5) {$P$};
    \node at (8, 4.8) { \color{brown} $d$};
\end{tikzpicture}
\caption{The annulus $A_1$ between $a_1$ and $b_2$}
\end{figure}
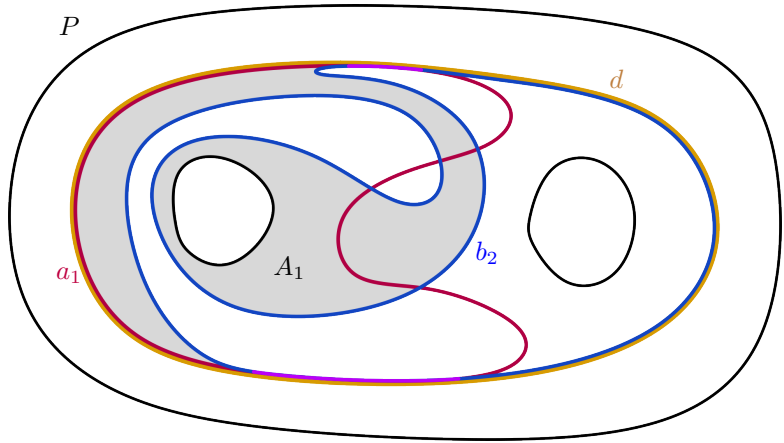

 Note that $A_1 \subset P_d$ and no part of $b$ is in the interior of $A_1$ since $b$ cannot self-intersect, nor intersect the interior of $a_1$. Since $P_d$ is a pair of pants containing $a$ but not $\gamma_3$ with boundaries $d$,  $\gamma_1$,  and $\gamma_2$,  then any $e$ that intersects $a$ and $\gamma_3$,  but not $\gamma_1$ or $\gamma_2$ must then intersect $d$.  If $e$ does not intersect $b$,  then it must intersect $d$ somewhere in the interior of $a_1$, say at the point $x$.  Since $b$ is a closed set in $P_d$,  there exists an open neighborhood in $P_d$ of $x$ that does not contain $b$.  Moreover, this neighborhood will also be contained in $A_1$.  Since $a_1$ is a $C^k$ arc, then $a_1 \cup b_2$ can be isotoped by a push off to be an essential $C^k$ curve $a'$ contained in $A_1$ and intersecting $a_1 \cup b_2$ only at the point $x$.  Since $a'$ is an essential curve in an annulus with a boundary $\gamma_1$, then $a'$ is homotopic to $\gamma_1$ and thus homotopic to $a$.  

By symmetry, the same reasoning can be used to show that any curve that intersects $b$ but not $a$ will also have the desired property.  

\medskip

\noindent The proof of the reverse direction will be broken down into the following steps: 
\begin{enumerate}[label=\arabic*)]
\item  The curve $d$ is a union of subarcs of $a$, $b$, and $a \cap b$. 
\item For each subarc of $a$, there exists a curve from $\gamma_3$  to $\gamma_1$ that only intersects $d$ along this arc. 
\item These curves bound regions that cannot contain $b$, and so there can only be one subarc of $a$ in $d$. 
\item There is also only one subarc for $b$, and thus $(a,b)$ must be a $k$-smooth pair that determine $d$.
\end{enumerate}

\noindent \textit{Step 1:  } Our first goal is to show that $d$ is the union of subarcs from $a$, $b$, and $a \cap b$.  Notice that if $d$ is in the hull of $\{a,b\}$ and $a \cup b$ is in the pair of pants bounded by $\gamma_1, \gamma_2$, and $d$, then $d$ will be contained in $a \cup b$.  Thus $d$ will be comprised of collections of arcs $\{a_i\}$, $\{b_j\}$ and $\{c_l\}$,  where the $a_i$ are subarcs of $a$ whose nonempty interiors do not intersect $b$, the $b_j$ are subarcs of $b$ whose nonempty interiors do not intersect $a$ and $c_l$ are subarcs of $a \cap b$, possibly degenerate. 

\medskip

\noindent \textit{Step 2:  } For each $a_i$, we construct curves $f_i$ from $\gamma_1$ to $\gamma_3$ that intersects $d$ only along $a_i$.  We start by constructing a curve $e_i \in \fcn{k}{S}$ in the annulus bounded by $\gamma_3$ and $d$ such that $e_i$ intersects $d$ only at an interior point $x \in a_i$.  In addition, $e_i$ can also be made to intersect $\gamma_3$.  Thus by condition $(i)$,  there exists an $a_i' \in \fcn{k}{S_g}$ contained in $P_d$ and homotopic to $a$ that intersects $e_i$, but does not intersect $b$.  Since $e_i$ only intersects $P_d$ at the point $x$, then $a_i'$ must also contain the point $x$.  

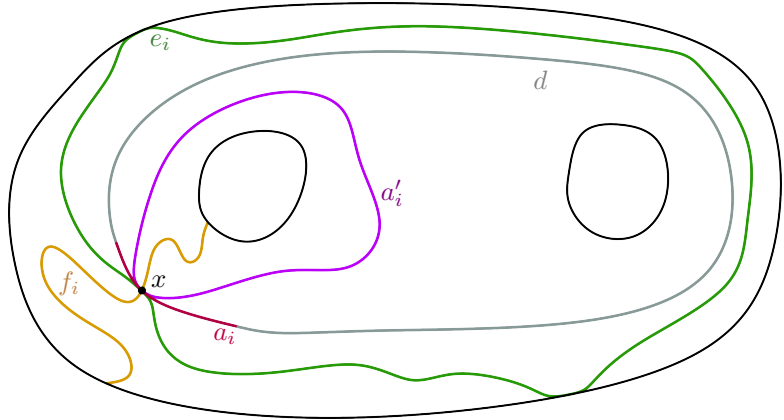
\begin{figure}[h]
\centering
\begin{tikzpicture}
\small
    \node[anchor=south west, inner sep = 0] at (0,0){\includegraphics[width=4in]{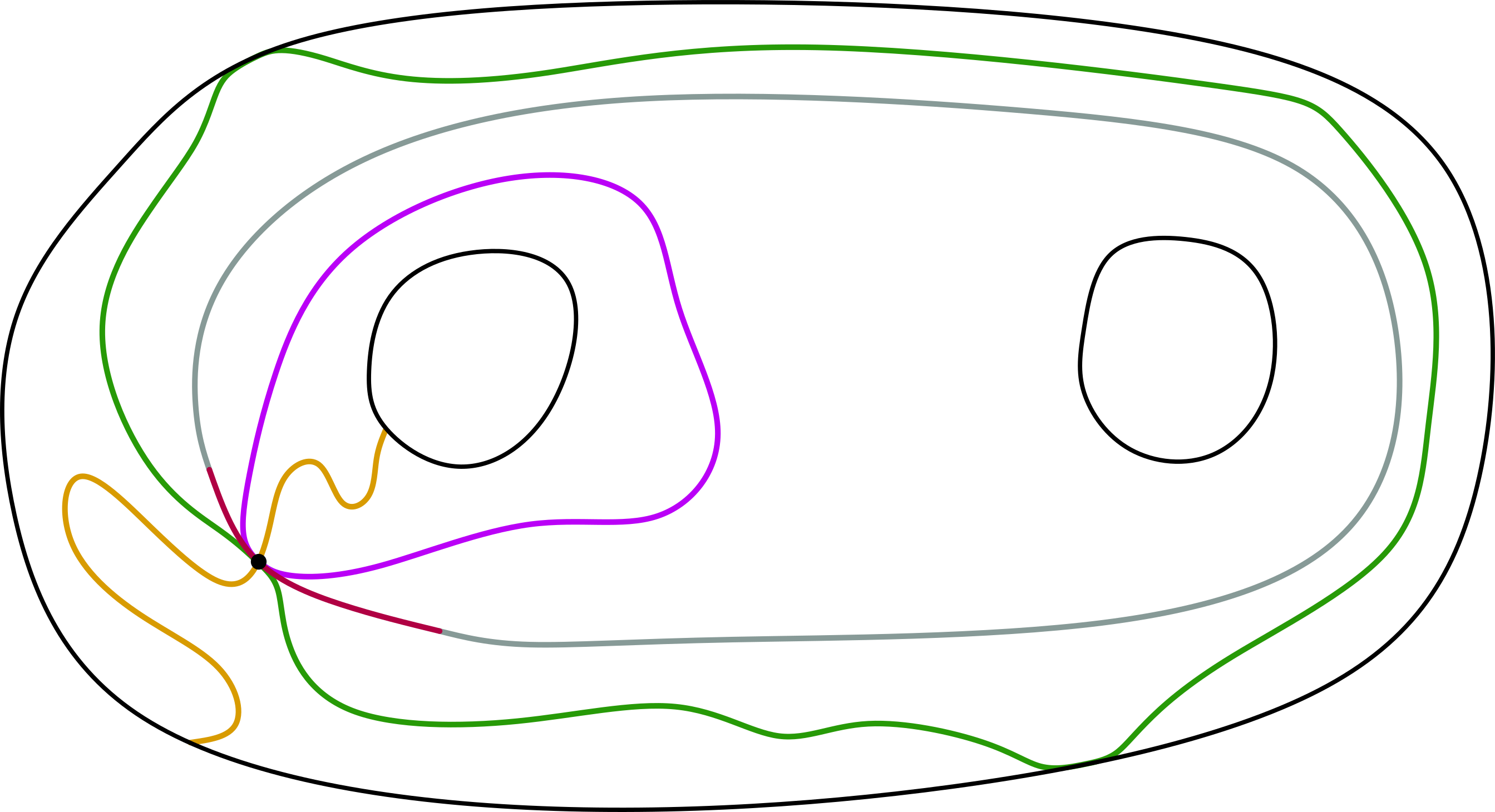} };
    \node at (2,5) {\color{OliveGreen} $e_i$};
    \node at (7,4.5) {\color{gray} $d$};
    \node at (5.05,3) {\color{violet} $a'_i$};
    \node at (.8, 1.8) {\color{brown} $f_i$};
    \node at (1.98, 1.83) {$x$};
    \node at (2.85, 1.1) {\color{purple} $a_i$};
\end{tikzpicture}
\caption{The construction of the curves $e_i$ and $f_i$.}
\end{figure}

Since neither $e_i$ and $a_i'$ cross $d$ at the point $x$, then they must have a one-sided tangent intersection at $x$.  Moreover, $e_i$ and $\gamma_3$ form a pinched annulus and $a_i'$ and $\gamma_1$ form an annulus.  So there exists $C^k$ arcs from the point $x$ to both $\gamma_1$ and $\gamma_3$.  By Lemma \ref{smoothingarcsatboundary}, this can be smoothed to a single $C^k$ arc with the same endpoints on $\gamma_1$ and $\gamma_3$ that is still disjoint from $b$.  Since $S_g \setminus P$ is connected, then this arc can be extended to a curve $f_i \in \fcn{k}{S_g}$ that forms degenerate torus pairs with $\gamma_1, \gamma_3,$ and $d$, but does not intersect~$b$.

\medskip

\noindent \textit{Step 3:  } We now show that there can only be one $a_i$ arc.  Suppose that there are at least two of the $a_i$. Then there will be curves $f_1, f_2 \in \fcn{k}{S_g}$ that are torus pairs with $\gamma_1, \gamma_3$ and $d$, but do not intersect $b$.

\begin{figure}[h]
\centering
\begin{tikzpicture}
\small
    \node[anchor=south west, inner sep = 0] at (0,0){\includegraphics[width=4in]{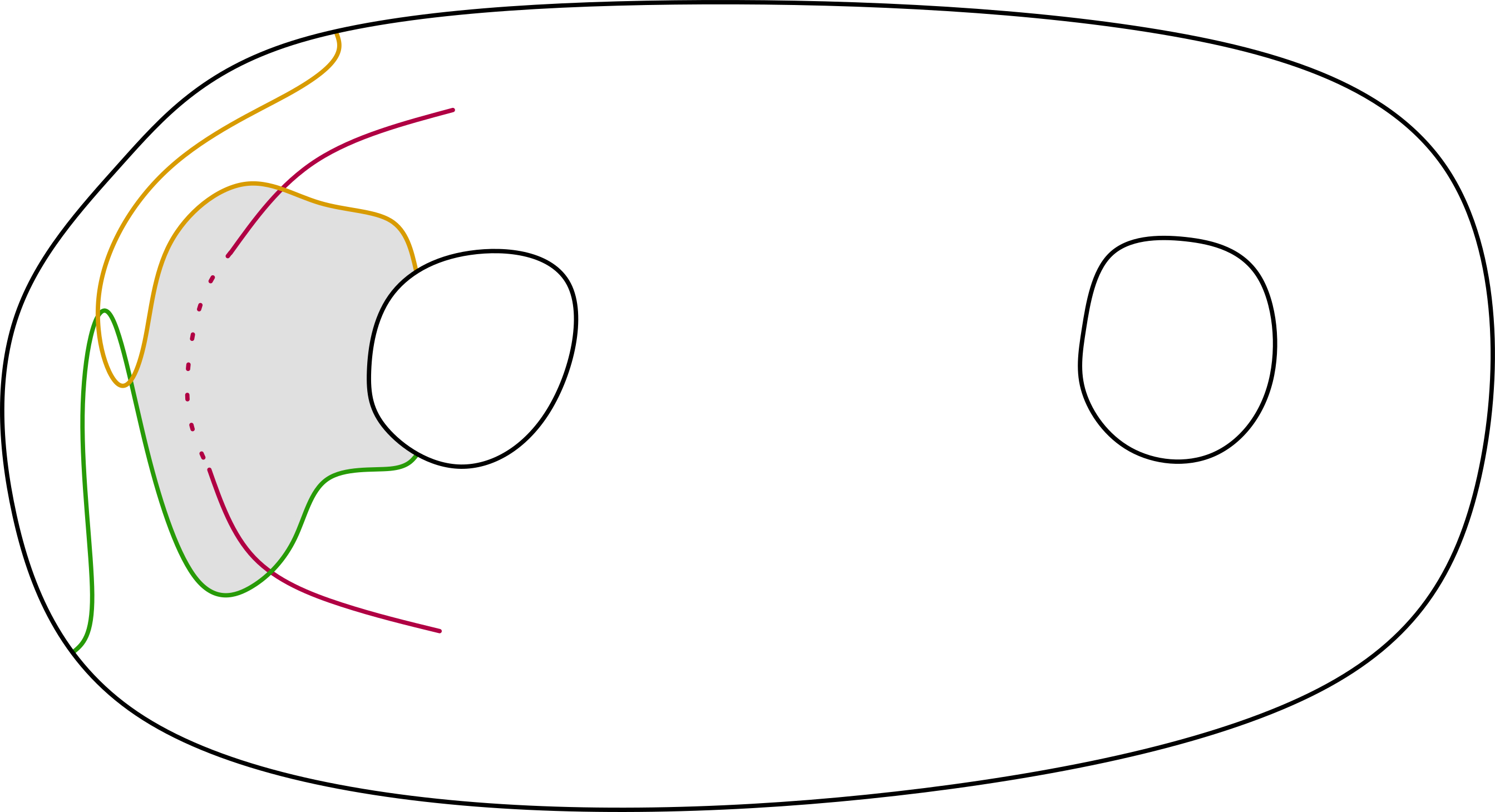} };
    \node at (2.5,5.1) {\color{brown} $f_1$};
    \node at (.9,1.5) {\color{OliveGreen} $f_2$};
    \node at (3,4.5) {\color{purple} $a_1$};
    \node at (2.7, 1.5) {\color{purple} $a_2$};
\end{tikzpicture}
\caption{The disc created by $f_1$, $f_2$, and $\gamma_1$.}
\end{figure}
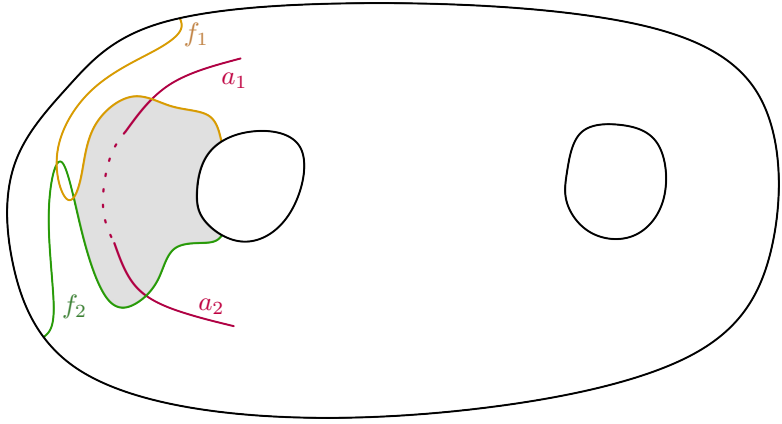

Similar to Lemma \ref{disksinpants}, there is some combination of arcs from $f_1$, $f_2$, $\gamma_1$ and $\gamma_3$ that bound a disk that contains one endpoint from each of $a_1$ and $a_2$.  In particular, $f_1$ and $f_2$ can intersect, but this will cause $P \setminus (f_1 \cup f_2)$ to have more components where still only one component can be an annulus.  

Since $d$ only intersects $f_1$ and $f_2$ once and does not intersect $\gamma_1$ or $\gamma_3$,  then the curve $d$ cannot leave this disk except where it intersects $f_1$ and $f_2$.  Nor can any part of $b$ be in this disk. Thus $d$ must connect $a_1$ and $a_2$ by an arc of $a$.  But then $d$ must connect every $a_i$ by another arc of $a$.  So all of the $a_i$ can be represented by a single subarc $a_1$ of $a$. 

\medskip

\noindent \textit{Step 4:  } By symmetry, this argument can also be used to show that all the $b_j$ can be  represented by a single subarc $b_1$ of $b$.  Since there are only two subarcs and thus four endpoints of $a_1$ and $b_1$ that $d$ must connect, then there are only two $c_l$, each connecting one endpoint of $a_1$ to one endpoint of $b_1$. Thus $(a,b)$ is a $k$-smooth pair that determines $d$. 
\end{proof}

\begin{lemma}\label{autpreservessmoothpairs}
Let $g \geq 2$ and $\alpha \in \aut\fcn{k}{S_g}$.  Then $\alpha$ preserves the set of $k$-smooth pairs.
\end{lemma}

\begin{proof}
By Lemma \ref{smoothpairingraph}, the set of $k$-smooth pairs $(a,b)$ is characterized using homotopies, hulls, and intersections in the nonseparating pairs of pants $P_d$ and $P$.  For any $\alpha \in \aut \fcn{k}{S_g}$,  intersections will be preserved and hulls and homotopies are also preserved by Lemma \ref{autpreservesephomandhull}.  Moreover,  containment in nonseparating pairs of pants is preserved by Lemma \ref{autpreservessubsurfaces}.  Thus $(\alpha(a),\alpha(b))$ will still have the same characterization as $(a,b)$ in the graph and so it is also a $k$-smooth pair that determines $\alpha(d)$. 
\end{proof}

\subsection{Simple $\mathbf{k}$-smooth pairs}\label{subsectionsimplesmooth}

In this section, we will define what it means for a $k$-smooth pair to be simple and show that simple $k$-smooth pairs are preserved by automorphisms.  This notion of simpleness is concerned with the intersections between the two curves of the $k$-smooth pair.

\medskip

A $k$-smooth pair $(a,b) \in \fcn{k}{S_g}$ in a nonseparating pair of pants $P$ that determines $d = a_1 \cup b_1 \cup c_1 \cup c_2$ is called \emph{simple} if $a$ and $b$ only intersect along $c_1$ and $c_2$.

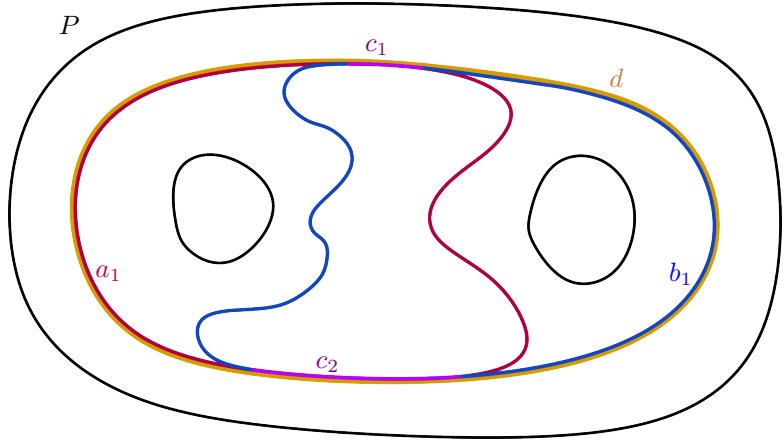
\begin{figure}[h]
\centering
\begin{tikzpicture}
\small
    \node[anchor=south west, inner sep = 0] at (0,0){\includegraphics[width=4in]{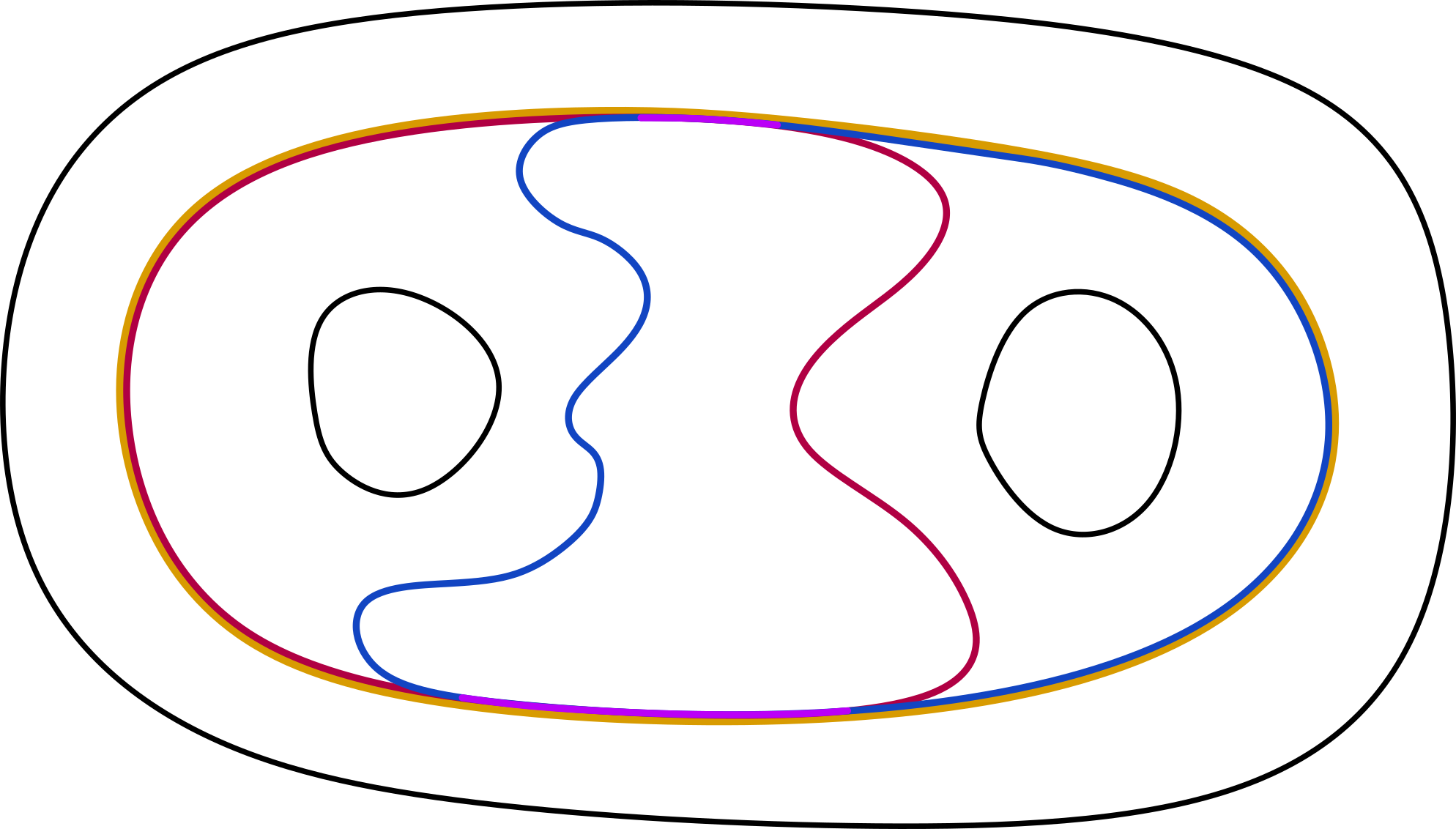}};
    \node at (1.32,2.2) {\color{purple}$a_1$};
    \node at (8.85,2.2) {\color{blue}$b_1$};
    \node at (4.85,5.2) {\color{violet}$c_1$};
    \node at (4.2, 1) {\color{violet}$c_2$};
    \node at (.8, 5.5) {$P$};
    \node at (8, 4.8) {\color{brown} $d$};
\end{tikzpicture}
\caption{A simple $k$-smooth pair}
\end{figure}

\begin{lemma}\label{smoothbigonscross}
Let $P$ be a nonseparating pair of pants in $S_g$ and $a,b \in \fcn{k}{S_g}$ be a $k$-smooth pair contained in $P$ that determine the curve $d = a_1 \cup b_1 \cup c_1 \cup c_2$.  Then $a$ and $b$ must cross each other at both $c_1$ and $c_2$.
\end{lemma}
\begin{proof}
Set $d = a_1 \cup b_1 \cup c_1 \cup c_2$, with these arcs given the definition of the $k$-smooth pair $(a,b)$.  Denote the complementary arcs of $a$ and $b$ by $a_2 = a \setminus (a_1 \cup c_1 \cup c_2)$ and $b_2 = b \setminus (b_1 \cup c_1 \cup c_2)$.
Since $d$ is a $C^k$ curve, then $d$ can be seen as a straight line in a local chart after applying a $C^k$ diffeomorphism.  Similarly,  since $a$ and $b$ are also $C^k$, then both $a_1$ and $b_2$ must have the same tangent line as $d$ at their shared endpoint that meets $c_1$.  Moreover, since $a_1$ and $b_1$ will approach $c_1$ from opposite directions along $d$, then $b_2$ must leave $c_1$ in exactly the same direction as $a_1$.  But if $a$ does not cross $b$ or $d$, then this would force $a_2$ to leave $c_1$ in the same direction and between $a_1$ and $b_2$. But then $a$ would have a cusp at $c_1$, which is not possible since $a$ is a $C^k$ curve. Thus $a$ and $b$ must cross at $c_1$.  By symmetry, $a$ and $b$ will also cross at $c_2$.  \end{proof}

\begin{figure}[h]
\centering
\begin{tikzpicture}
\small
    \node[anchor=south west, inner sep = 0] at (0,0){\includegraphics[width=6in]{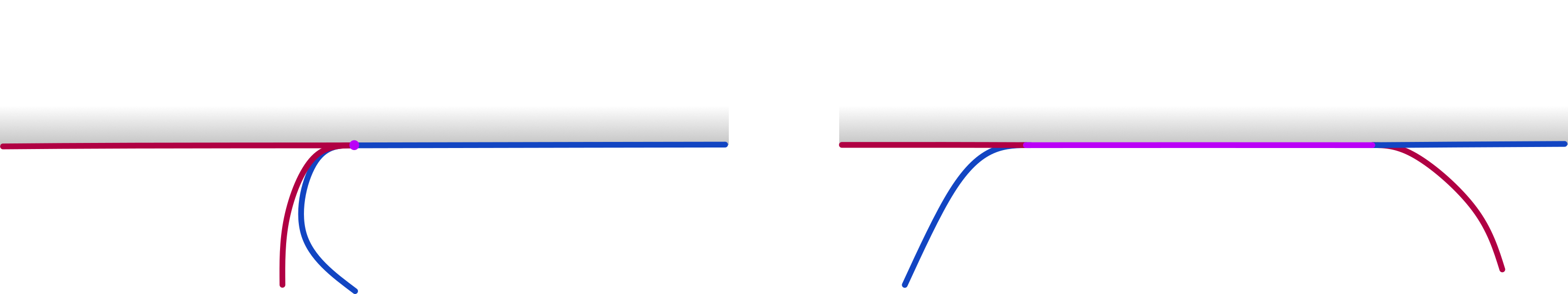}};
    \node at (1,1.2) {\color{purple}$a_1$};
        \node at (8.5,1.2) {\color{purple}$a_1$};
            \node at (2.5,.5) {\color{purple}$a_2$};
                \node at (14.25,.4) {\color{purple}$a_2$};
    \node at (5.5,1.2) {\color{blue}$b_1$};
    \node at (15,1.2) {\color{blue}$b_1$};
    \node at (3.5,.3) {\color{blue}$b_2$};
    \node at (9.3,.5) {\color{blue}$b_2$};
    \node at (3.5,1.2) {\color{violet}$c_1$};
    \node at (11.5, 1.2) {\color{violet}$c_1$};
    \node at (3.5, 2.2) {\color{gray} Outside of $P'$};
    \node at (11.5, 2.2) {\color{gray} Outside of $P'$};
\end{tikzpicture}
\caption{The noncrossing (left) and crossing (right) cases for the arc $c_1$.  }
\end{figure}
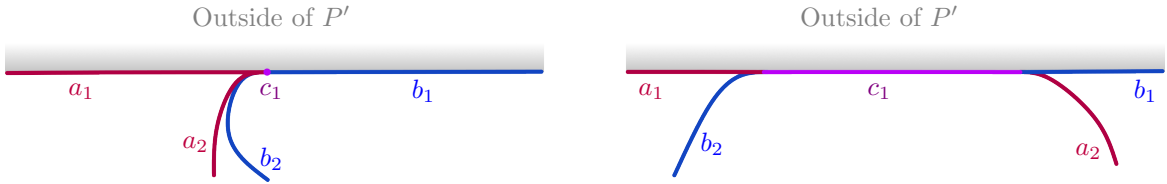

\begin{lemma}\label{noncrosssmoothpairingraph}
Let $P$ be a nonseparating pair of pants in $S_g$ with boundary curves $\gamma_1, \gamma_2, $ and $\gamma_3 \in \fcn{k}{S_g}$. Let $a,b \in \fcn{k}{S_g}$ be a $k$-smooth pair in $P$ with $a$ homotopic to $\gamma_1$ and $b$ homotopic to $\gamma_2$ that determine the curve $d = a_1 \cup b_1 \cup c_1 \cup c_2$.  Then the following are equivalent:
\begin{enumerate}[noitemsep,topsep=0pt, label=$(\alph{*})$]
\item $a$ and $b$ only cross at $c_1$ and $c_2$
\item If $e \in \fcn{k}{S_g}$ is contained in $P$, homotopic to $a$ and intersects $a$,  then $e$ also intersects $b$ or $d$.
\item If $e \in \fcn{k}{S_g}$ is contained in $P$, homotopic to $b$ and intersects $b$,  then $e$ also intersects $a$ or $d$.
 \end{enumerate}
\end{lemma}

\begin{figure}[h]
\centering
\begin{tikzpicture}
\small
    \node[anchor=south west, inner sep = 0] at (0,0){\includegraphics[width=4in]{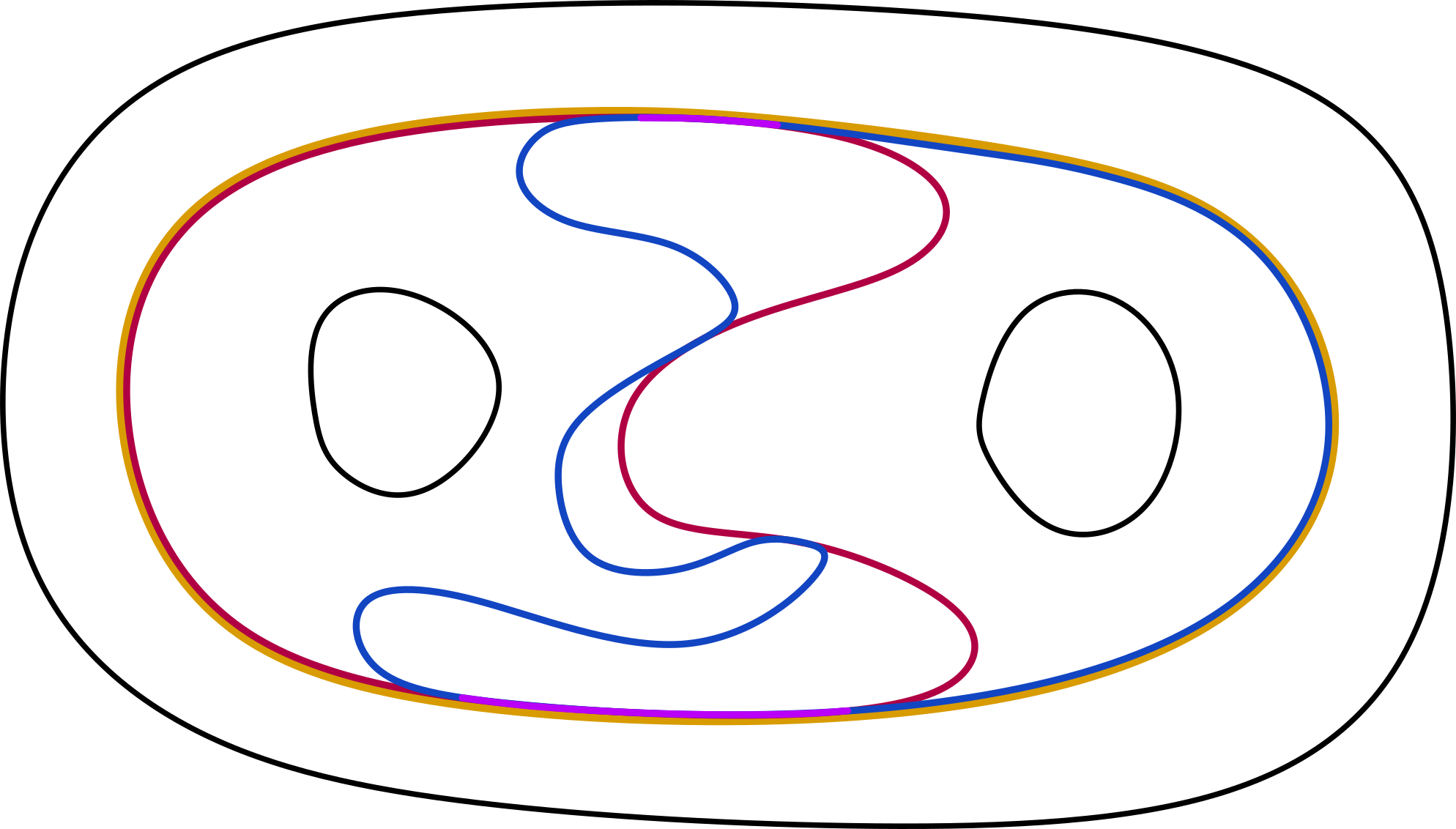}};
    \node at (6.1,3.6) {\color{purple}$a_2$};
    \node at (2.5,1.9) {\color{blue}$b_2$};
    \node at (4.85,5.2) {\color{violet}$c_1$};
    \node at (4.2, .6) {\color{violet}$c_2$};
    \node at (.8, 5.5) {$P$};
    \node at (8, 4.8) {\color{brown} $d$};
\end{tikzpicture}
\caption{A $k$-smooth pair that only cross at $c_1$ and $c_2$}
\end{figure}
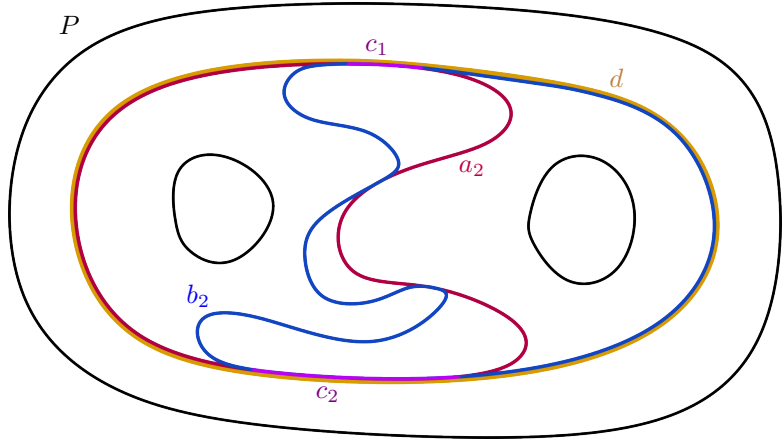

\begin{proof}
By symmetry of the statements, if (a) and (b) are equivalent, then so are (a) and (c). 

\medskip

\noindent $(a) \Rightarrow (b)$.  \hspace{.01in}
Suppose that there is a curve $e \in \fcn{k}{S_g}$ contained in $P$, homotopic to $a$,  that intersects $a$, but not $b$ or $d$.  Since $e$ is not homotopic to $d$ nor intersects $d$,  then $e$ must be contained in the pair of pants bounded by $\gamma_1, \gamma_2$, and $d$.  Even more,  $e$ will be contained in the open annulus $A_1$ bounded by $\gamma_1, a_1$, and $b_2 = b \setminus (b_1 \cup c_1 \cup c_2)$.  Since $a_1 \subset d$, then $e$ cannot intersect $a$ along the arc $a_1$.  Also, $e$ does not intersect $a$ along $c_1$ or $c_2$ since $e$ does not intersect $b$. Thus $e$ must intersect $a$ along $a_2 = a \setminus (a_1 \cup c_1 \cup c_2)$.  Since $a$ cannot intersect itself, then some portion of $a_2$ must enter $A_1$ by crossing $b_2$. Thus $a$ and $b$ must cross at points that are not in $c_1$ and $c_2$. 

\medskip

\noindent $(b) \Rightarrow (a)$. \hspace{.01in}
Suppose that $a$ crosses $b$ at a point that is not in $c_1$ or $c_2$.  Since $a$ cannot cross $d$, then $a$ must cross $b$ along the arc $b_2$. Thus there will be a point $x$ of $a$ that is in the open annulus $A_1$ bounded by $\gamma_1, a_1$, and $b_2$.  Since $A_1$ is open, then there exists a curve $e \in \fcn{k}{S_g}$ that is contained in $A_1$ and contains the point $x$. Thus $e$ is contained in $P$, homotopic to $a$ and intersects $a$ at $x$, but does not intersect $b$ or $d$. 
\end{proof}

\begin{lemma}\label{notmanyannulicomponents}
Let $P$ be a nonseparating pair of pants in $S_g$ with boundary curves $\gamma_1, \gamma_2,$ and $\gamma_3$.  Let $a,b \in \fcn{k}{S_g}$ be a $k$-smooth pair contained in $P$ that determine the curve $d$ homotopic to $\gamma_3$.  Let $P_d \subset P$ denote the pair of pants with boundaries $\gamma_1$, $\gamma_2$, and $d$. Then the only annular components of $P_d \setminus \{a \cup b \}$ have either $\gamma_1$ or $\gamma_2$ as a boundary.  All other components are disks. 
\end{lemma}

\begin{proof}
If any annular component of $P_d \setminus \{a \cup b \}$ does not contain $\gamma_1$ or $\gamma_2$, then both boundary components will consist of arcs of $a$ and $b$ that would be separated in $P_d$ by this annular component.  Since $a$ and $b$ are each a single connected curve, then subarcs of either of them cannot be separated in $P_d$.  So all annular components of $P_d \setminus \{a \cup b \}$ must have either $\gamma_1$ or $\gamma_2$ as boundaries.  

Since $P_d$ is a pair of pants, then any subsurfaces will either be pairs of pants, annuli, or disks. Since $a$ separates $\gamma_1$ and $\gamma_2$ and intersects $d$, then no subsurface of $P_d \setminus \{a \cup b \}$ can be a pair of pants. Thus all nonannular components of $P_d \setminus \{a \cup b \}$ are disks.
\end{proof}

\medskip

\noindent\textit{Curves of type 1 and type 2. }
Let $P$ be a nonseparating pair of pants on $S_g$ and $(a,b) \in \fcn{k}{S_g}$ be a $k$-smooth pair that determine $d$.  Suppose that $\gamma_1, \gamma_2$ are the boundaries of $P$ homotopic to $a$ and $b$. Let $e \in \fcn{k}{S_g}$ be a curve disjoint from $d$ that forms degenerate torus pairs with $\gamma_1, \gamma_2, a$ and $b$. Then $e$ is called \emph{type 1} if it intersects $a$ and $b$ at the same point.  Otherwise, it is called \emph{type 2}. 

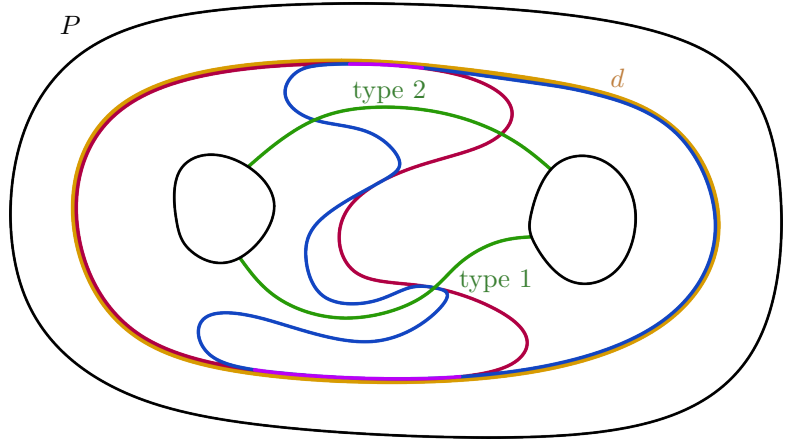
\begin{figure}[h]
\centering
\begin{tikzpicture}
\small
    \node[anchor=south west, inner sep = 0] at (0,0){\includegraphics[width=4in]{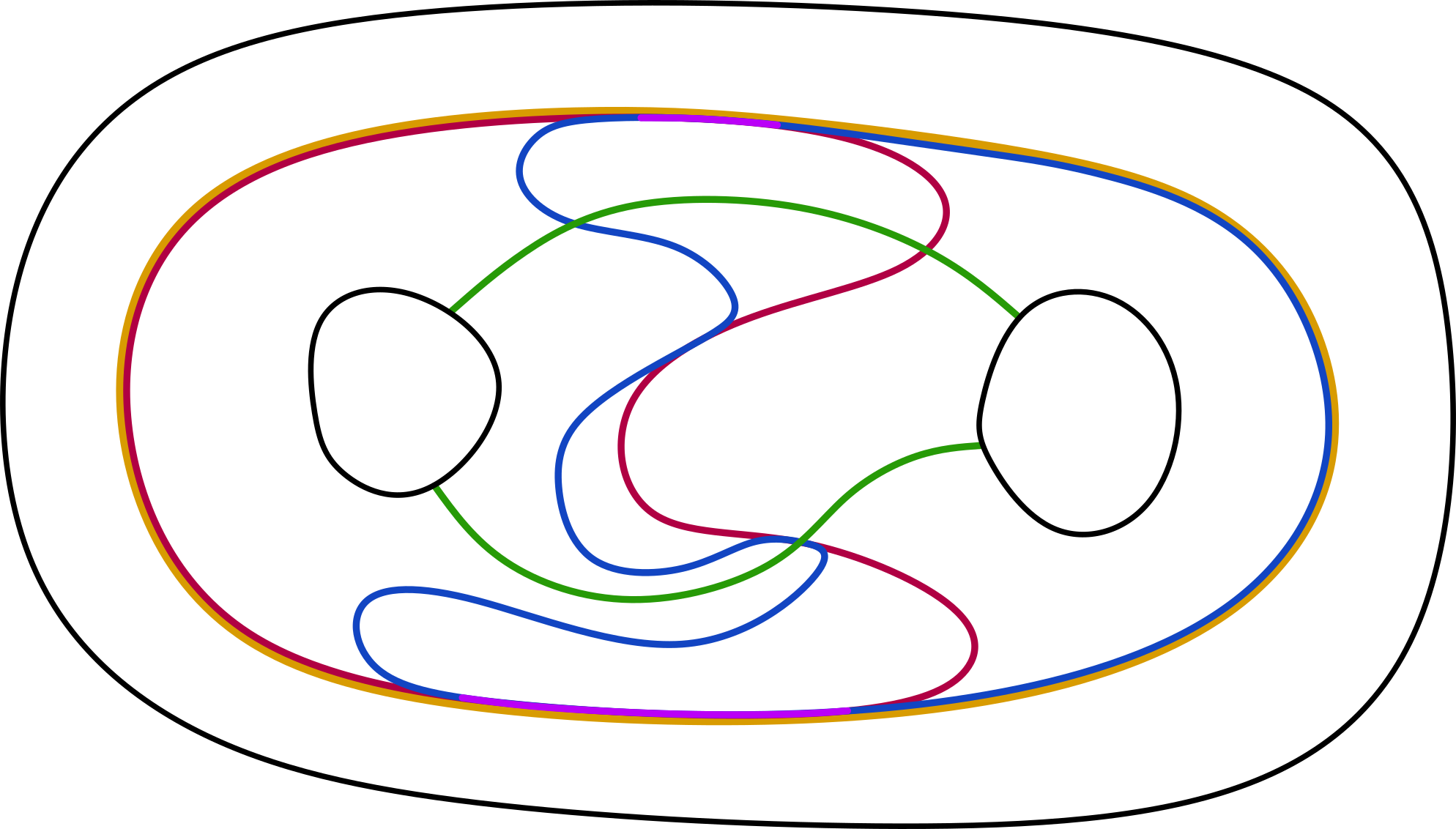}};
    \node at (6.4,2.1) {\color{OliveGreen} type 1};
    \node at (5,4.6) {\color{OliveGreen} type 2};
    \node at (.8, 5.5) {$P$};
    \node at (8, 4.8) {\color{brown} $d$};
\end{tikzpicture}
\caption{Type 1 and type 2 curves}
\end{figure}

\begin{lemma}\label{type1ingraph}
Let $P$ be a nonseparating pair of pants on $S_g$ with boundary curves $\gamma_1, \gamma_2$,  and $\gamma_3 \in \fcn{k}{S_g}$.  Let $(a,b) \in \fcn{k}{S_g}$ be a $k$-smooth pair that determine a curve $d$ homotopic to $\gamma_3$.  Suppose that $e \in \fcn{k}{S_g}$ is disjoint from $d$ and forms degenerate torus pairs with $\gamma_1, \gamma_2, a$ and $b$.   Then $e$ is type 1 if and only if there are no curves homotopic to $e$ in the hull of $\{a, b, e\}$ except $e$ itself.
\end{lemma}

\begin{proof} Let $P_d \subset P$ denote the pair of pants with boundary curves $\gamma_1$, $\gamma_2$, and $d$. 

We start with the forward direction.  Suppose $e$ is a curve of type 1.  Then $e$ will only intersect each $a$ and $b$ once at the exact same point. So $e \setminus (a \cup b)$ must be a single interval.  Since $e$ is essential and $e \cup (a \cup b)$ is a single point,  then the closure of this interval is $e$ itself and cannot bound a disk.  So the hull of $\{a,b,e\}$ will consist exactly of the curves in the hull of $\{a, b\}$ and the curve $e$ itself.  Also, any curve in the hull of $\{a,b\}$ must live on $P_d$ and is thus homotopic to $a$, $b$ or $d$.  But $e$ forms torus pairs with $a$ and $b$. Since $S_g$ is orientable,  then $e$ cannot be homotopic to $a$ or $b$ and every curve homotopic to $e$ must also intersect $a$ and $b$. So $e$ cannot be isotopic to $d$.  Thus any curve in the hull of $\{a, b\}$ is not homotopic to $e$. So $e$ is the only curve in the hull of $\{a,b,e\}$ that is homotopic to $e$. 

We now complete the reverse implication.
Suppose that $e$ is type 2. Then $e$ must intersect $a$ and $b$ at two distinct points.  So $e \setminus a \cup b$ consists of two nonempty open subarcs $e_1$ and $e_2$. Since $e$ also forms torus pairs with $\gamma_1$ and $\gamma_2$, then one of $e_1$ or $e_2$ must include the entire arc of $e \setminus P$. So the other interval must be completely contained within $P_d$.  Without loss of generality, let $e_1$ be the interval contained in $P_d$.   By Lemma \ref{notmanyannulicomponents},  since $e_1$ is not in a component of $P_d \setminus \{a \cup b \}$ that contains $\gamma_1$ or $\gamma_2$, then this component must be a disk.  Since this is an open interval in an open disk, then $e$ can be smoothly isotoped within this disk to make a family of new $e' \in \fcn{k}{S_g}$ homotopic to $e$ and in the hull of $\{a, b, e\}$.
\end{proof}

\begin{lemma}\label{simpleonlytype2}
Let $(a,b) \in \fcn{k}{S_g}$ be a $k$-smooth pair in a nonseparating pair of pants $P$ that determines a curve $d = a_1 \cup b_1 \cup c_1 \cup c_2$, where $a$ and $b$ only cross at $c_1$ and $c_2$.  Then $(a,b)$ is simple if and only if every $e$ in $\fcn{k}{S_g}$ disjoint from $d$ which forms degenerate torus pairs with $\gamma_1, \gamma_2, a$ and $b$ must be type 2. 
\end{lemma}

\begin{proof}

If any curve $e \in \fcn{k}{S_g}$ disjoint from $d$ is type 1, then the point where $e$ intersects $a$ and $b$ is also a point not in $c_1$ or $c_2$ where $a$ and $b$ intersect. Thus $a$ and $b$ will not be simple. 

Suppose that $(a,b)$ is not simple.  Then there is some point $x$ not in $c_1$ or $c_2$ where $a$ and $b$ intersect.  By assumption, $a$ and $b$ cannot cross at $x$.  Thus they will be tangent at $x$. Moreover, this point $x$ will be on the arc $b_2$ in the boundary of the annulus $A_1$ described in the proof of Lemma \ref{noncrosssmoothpairingraph}. By symmetry, it will also be on the arc $a_2$ in the boundary of the corresponding annulus $B_1$.  So there are arcs $e_1$ and $e_2$ from $\gamma_1$ and $\gamma_2$ to $x$ that do not intersect $a$ or $b$ except at $x$.  By Lemma \ref{smoothingarcsatboundary},  these arcs can be smoothed and extended in $S_g \setminus P$ to a curve $e \in \fcn{k}{S_g}$ that is type 1. 
\end{proof}

Finally, we prove that any automorphism of $\fcn{k}{S_g}$ will preserve the set of simple $k$-smooth pairs.

\begin{proof}[Proof of Proposition \ref{c1preservessimplesmoothpair}]
Let $\alpha \in \aut \fcn{k}{S_g}$ and $(a,b)$ be a simple $k$-smooth pair that determines $d$.  By Lemma~\ref{autpreservessmoothpairs}, $(\alpha(a), \alpha(b))$ is also a $k$-smooth pair that determines $\alpha(d)$. 

Since $(a,b)$ is simple, then $a$ and $b$ only intersect along $c_1$ and $c_2$, then they must only cross at $c_1$ and $c_2$ since any other crossing would be cause $a$ and $b$ to have more intersections.   By Lemma~\ref{noncrosssmoothpairingraph},  only crossing at $c_1$ and $c_2$ can be characterized using containment in a nonseparating pair of pants, homotopy and intersections.  But Lemmas~\ref{autpreservesephomandhull} and \ref{autpreservessubsurfaces} show that these are all preserved by $\alpha$.  So $\alpha(a)$ and $\alpha(b)$ will also only cross at their overlapping intervals in $\alpha(d)$.  

From Lemma~\ref{type1ingraph},  any curve $e \in \fcn{k}{S_g}$ that is type 1 in relation to $(a,b)$ can be characterized with degenerate torus pairs, homotopies,  hulls, and intersections.  These are preserved by $\alpha$ using Lemmas~\ref{autpreservesephomandhull} and \ref{autpreservestoruspairs}. Thus $\alpha(e)$ must be type 1 in relation to $(\alpha(a), \alpha(b))$.  But by Lemma~\ref{simpleonlytype2}, $(a,b)$ will only have type 2 curves, and thus $(\alpha(a),\alpha(b))$ will also have only type 2 curves, and thus must be a simple $k$-smooth pair.
\end{proof}

\section{Constructing $\mathbf{ \boldsymbol{\xi}^{-1}: \baut \bfcn{k}{S} \boldsymbol{\rightarrow} \baut \befcn{k}{S}}$}\label{sectionc1toec1}

The goal of this section is to use the simple $k$-smooth pair construction defined in Section~\ref{sectioncurvepairs} to prove the following proposition giving an isomorphism between the automorphisms of the $C^k$-curve graph and the automorphisms of the extended $C^k$-curve graph.  This completes Step~2 of the proof outline from the introduction.

\begin{prop}\label{ec1toc1iso}
Let $S_g$ be a compact surface with $g \geq 2$. Then the natural restriction 
$$\xi: \aut \efcn{k}{S_g} \rightarrow \aut \fcn{k}{S_g}$$ is an isomorphism.
\end{prop}

In order to prove this proposition,  we build a connected arc graph in Section~\ref{subsectionpairedarcgraph} on a subsurface with edges that correspond to the simple $k$-smooth pairs.  In Section~\ref{subsectionkernel},  we leverage the connectedness of this arc graph and its line graph to build an equivalence relation on simple $k$-smooth pairs.  We then show that inessential curves determine by the simple $k$-smooth pairs are preserved by $\aut \fcn{k}{S_g}$ and well-defined.
Next, in Section~\ref{subsectioninessential} we show that there are unique structures in $\efcn{k}{S_g}$ for both inessential and essential curves and any combination of edges between them.  This implies that any automorphism of the extended graph will automatically give an automorphism of the subgraph $\fcn{k}{S_g}$.  Finally,  in Section~\ref{subsectionpropauts} we utilize our simple $k$-smooth pair construction to build an inverse map that recovers the image of each inessential curve based only on the images of the essential curves. 

\subsection{Paired $k$-smooth arc graphs are connected}\label{subsectionpairedarcgraph}

In this section, we will be considering arc graphs that correspond to simple $k$-smooth pairs where a complementary disk has been removed. The main result is Lemma \ref{finenonseppairedarcconnected}, showing that the fine version this arc graph is connected.  To get to this result, we first show that the corresponding non-fine arc graph is connected.  We then complete the argument by showing that each isotopy class is also connected. Recall that we denote by $S_{g}^b$ the surface of genus $g$ with $b$ disjoint open disks removed such that each boundary is a $C^k$ curve. 

\medskip

\noindent \textit{A non-fine arc graph.  }
We now define the \emph{paired $k$-smooth nonseparating arc graph}, denoted $\natp{S_{g}^1}$. The vertices of this graph consist of isotopy classes of arcs in $S_{g}^1$,  relative to the boundary and the tangent spaces on the boundary,  with the following properties:
\begin{enumerate}[label=(\alph*)]
\item the arcs are essential and nonseparating
\item the endpoints of each arc are distinct
\item the interior of the arcs are disjoint from any boundaries
\item there is an extension of each arc by a subarc of the boundary such that this extension is a $C^k$ curve on $S_{g}^1$
\end{enumerate}
There are edges between vertices when the corresponding isotopy classes have representative arcs with the following properties:
\begin{enumerate}[label=(\roman*)]
\item the interiors of the arcs are disjoint
\item the arcs are jointly nonseparating
\item each extension given in property (e) contains both endpoints of the other arc
\end{enumerate}
The arcs remaining from a simple $k$-smooth pair once the disk in their complement has been removed have the desired properties given above.

\begin{lemma}\label{nonseppairedarcgraphconnected}
Let $S_{g}^1$ for $g \geq 2$.  Then $\natp{S_{g}^1}$ is connected.
\end{lemma}

\begin{proof}

 We let $\sim$ denote the equivalence relation on $\natp{S_{g}^1}$ that identifies the two isotopy classes of arcs with the same endpoints. Then the mapping class group $\mcg(S_{g}^1)$ acts on the quotient graph $\natp{S_{g}^1}/{\sim}$.  The proof happens in two steps.  In the first step, we show that this quotient graph is connected.  We then show that the identified isotopy classes are connected in $\natp{S_{g}^1}$, resulting in the arc graph being connected.
 
\begin{figure}[h]
\centering
\includegraphics[angle=90, width=1in]{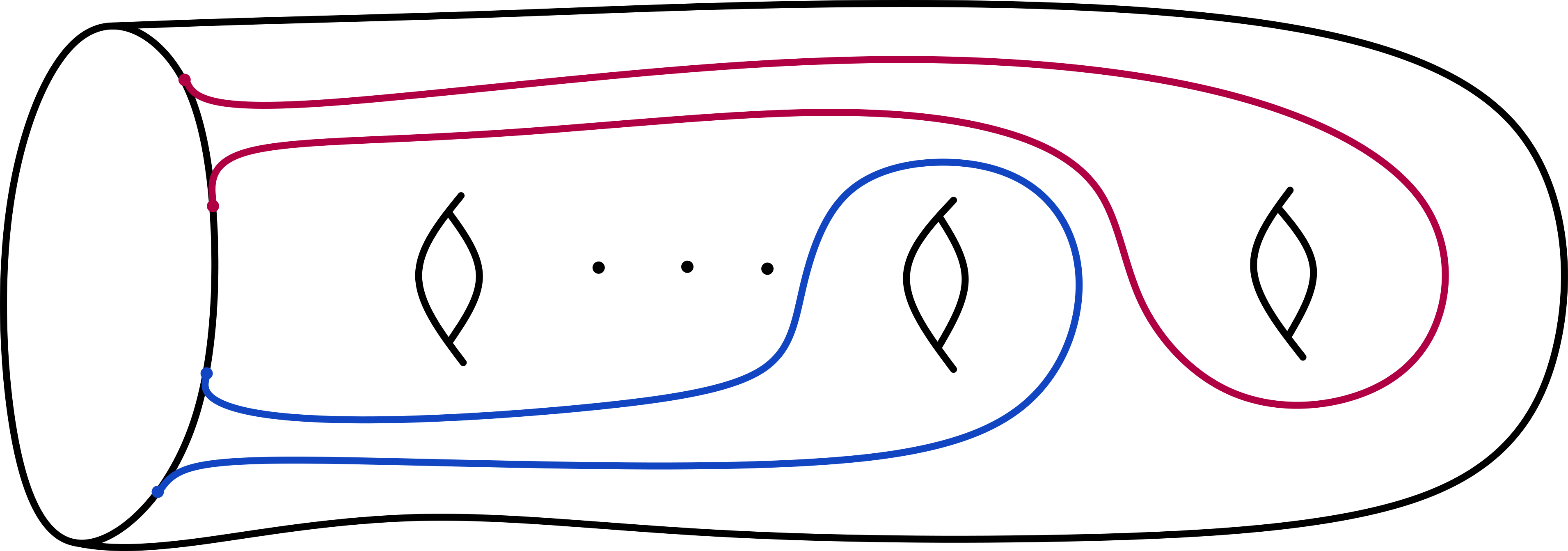} \hspace{.5in}
\includegraphics[angle=90, width=1in]{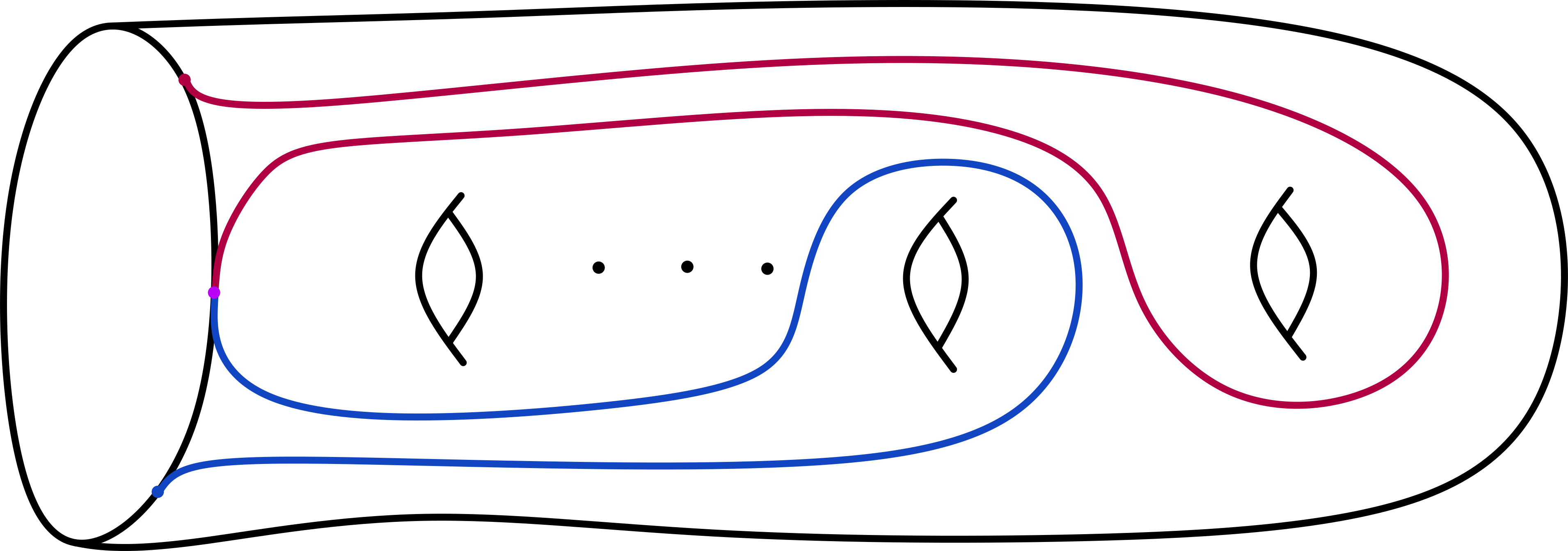} \hspace{.5in}
\includegraphics[angle=90, width=1in]{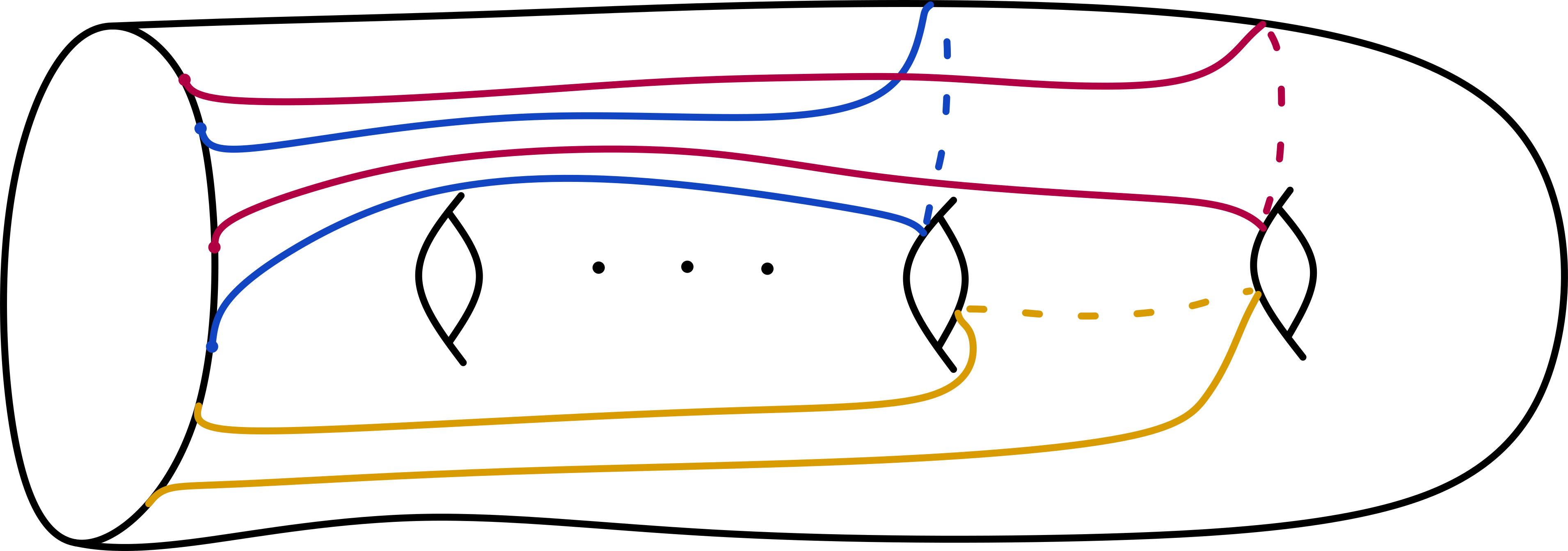}
\caption{The three possible configurations for endpoints from two different classes of arcs in the quotient graph.}
\end{figure}

\medskip

\noindent \textit{Step 1:  } To show that the quotient graph is connected, we use the Putman trick.  
 \cite[Lemma 2.1]{Putman} This method utilizes the action of a group on a simplicial complex to split connectivity into two parts.  First, we show that the group action forces the orbit of any vertex to the connected component of any other vertex.  We can then show that for any generator of the group, there is a path from a fixed vertex to the generator's image of that vertex.

 Since $\mcg(S_{g}^1)$ acts on the quotient graph, then using the change of coordinates principle \cite[Section 1.3]{Primer},   $\mcg(S_{g}^1)$ acts transitively on the set of vertices with the same endpoints. 

There are three possible configurations of pairs of endpoints in the boundary for arcs in different classes.  The endpoints can be paired,  overlapped,  or linked.  In all three situations, there exists either an edge or a path of length 2 that connects arcs in these configurations.  Thus all arcs are connected in $\natp{S_{g}^1}/\sim$ to an arc with any given pair of endpoints. 

Now consider a fixed pair of points in the boundary of $S_{g}^1$.  Let $\{h_1, h_2, \ldots, h_{2g+1}\}$ be the Humphries generators of $S_{g}^1$ and let $v$ be the arc given in Figure~\ref{fig-trickbefore} whose endpoints are the points that were fixed in the boundary.   

\begin{figure}[h!]
\centering
\begin{tikzpicture}
\small
    \node[anchor=south west, inner sep = 0] at (0,0){\includegraphics[width=4.5in]{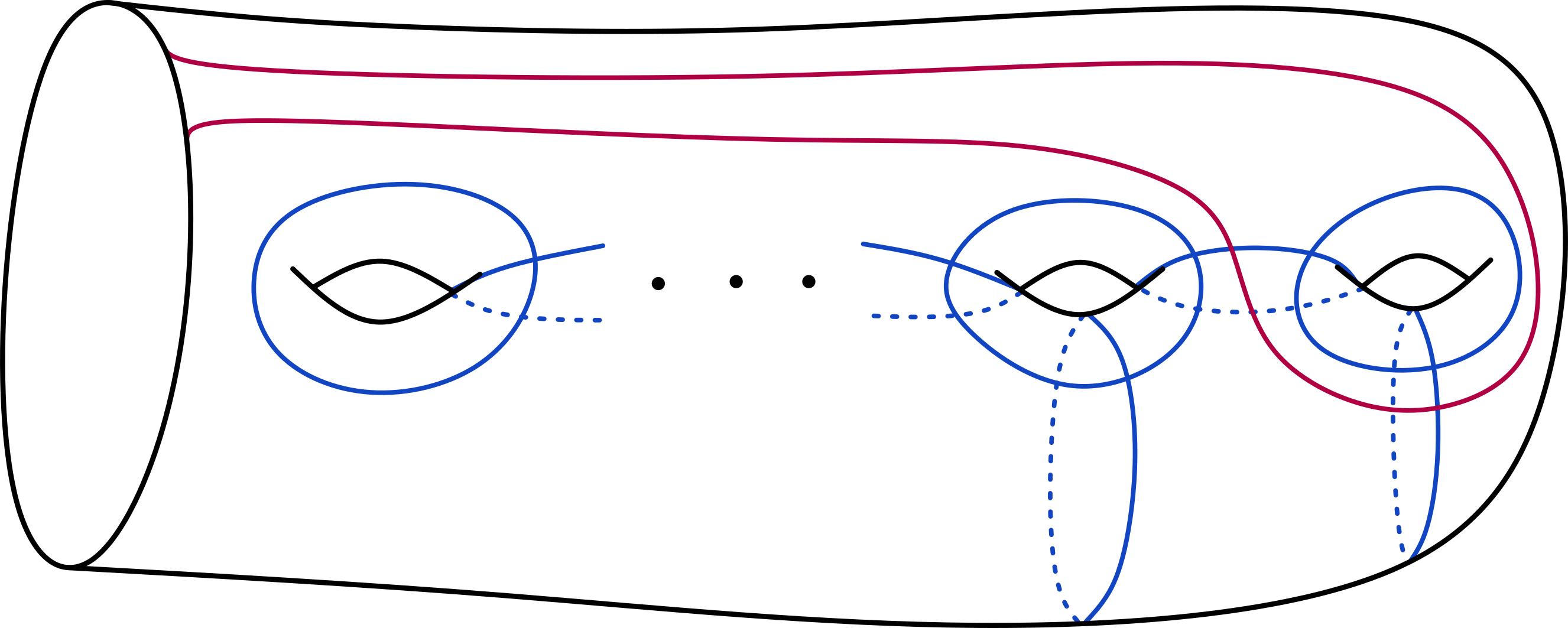}};
    \node at (9.6,3.85) {\color{purple} $v$};
    \node at (8.5,.5) {\color{blue} $h_2$};
    \node at (10.73, 1.13) {\color{blue} $h_1$};
    \node at (9.35, 3) { \color{blue} $h_4$};
    \node at (10.7, 3) { \color{blue} $h_3$};
    \node at (7.4,1.75) { \color{blue} $h_5$};
    \node at (6.65, 3) { \color{blue} $h_6$};
    \node at (4.4, 3) { \color{blue} $h_{2g}$};
    \node at (3, 1.5) { \color{blue} $h_{2g+1}$};
\end{tikzpicture}
\caption{The Humphries generators $h_i$ of $S^1_g$ with the arc $v$. }
\label{fig-trickbefore}
\end{figure}
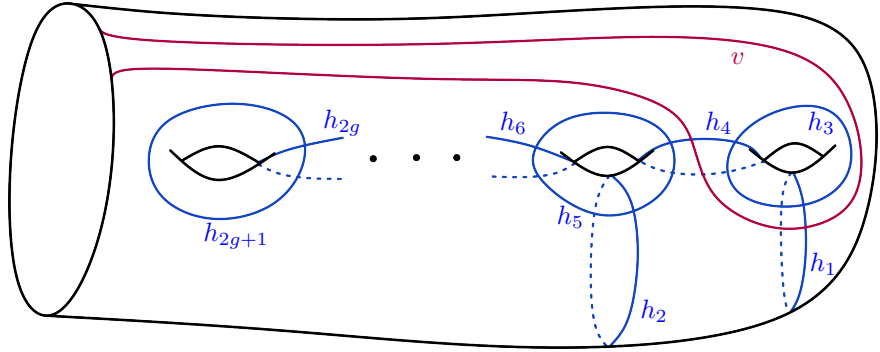

So $h_iv = v$ when $i \neq 1,4$ and are thus trivially connected to $v$ in the graph. When $i=1$ or $4$, then the arc $u$ in $\natp{S_{g}^1}$ given in Figure~\ref{fig-trickafter} is adjacent to $v$,  $h_1v$,  and $h_4v$.  So $v$ is connected by a path in $\natp{S_{g}^1}$ to both $h_1v$ and $h_4v$.

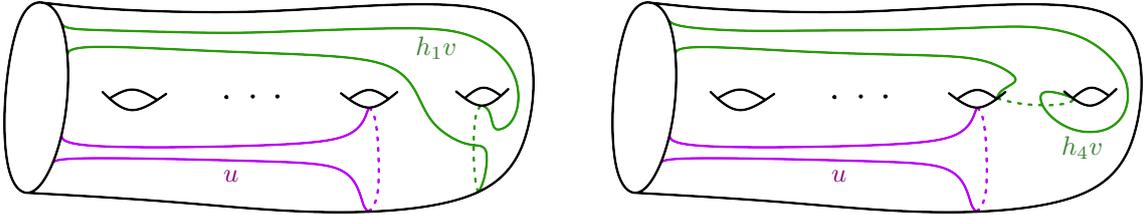
\begin{figure}[h!]
\centering
\begin{tikzpicture}
\small
    \node[anchor=south west, inner sep = 0] at (0,0){\includegraphics[width=2.75in]{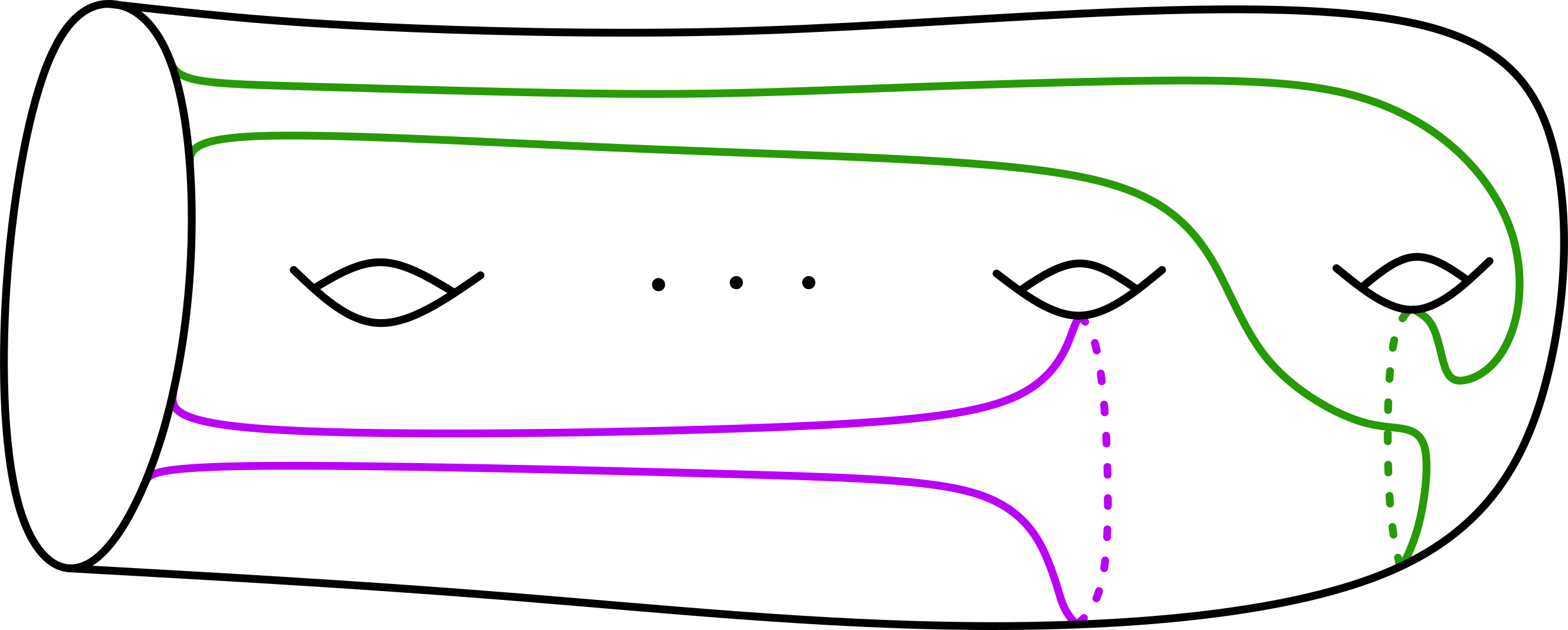}};
    \node[anchor=south west, inner sep = 0] at (8,0){\includegraphics[width=2.75in]{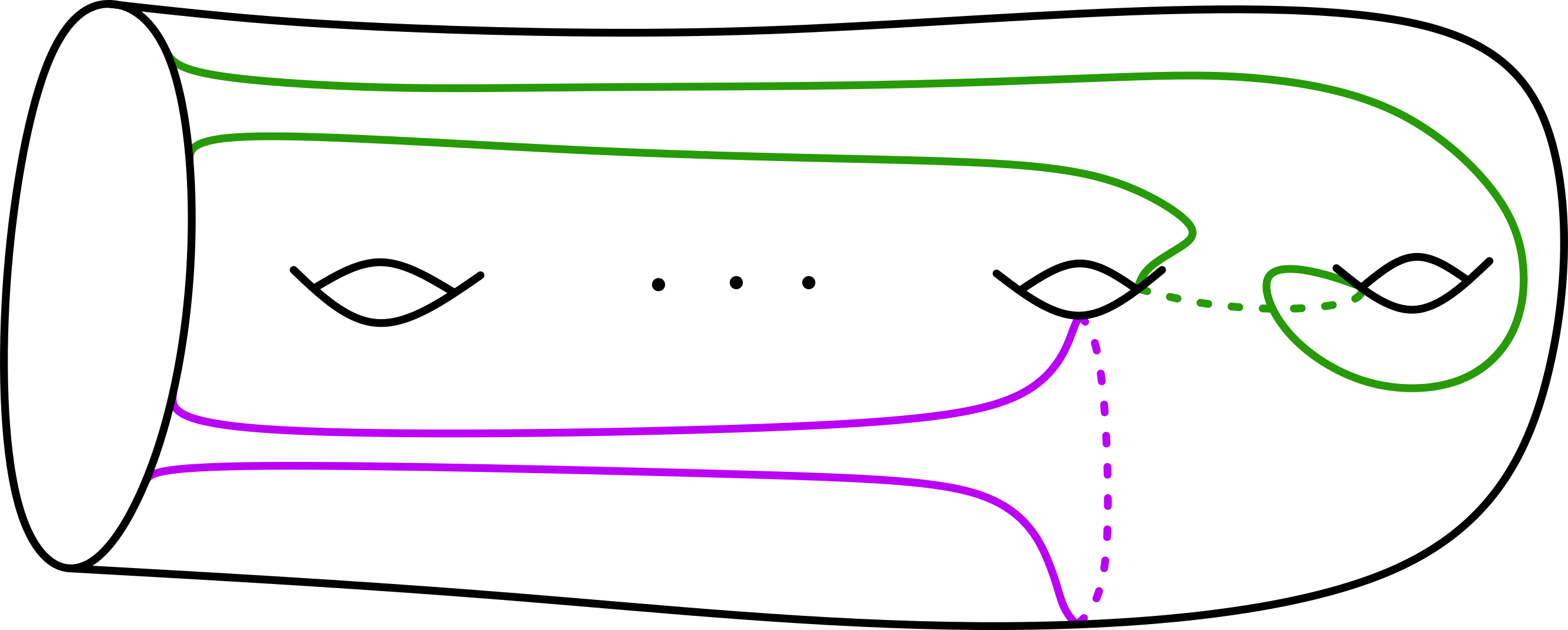}};
    \node at (3,.5) {\color{violet} $u$};
    \node at (11,.5) {\color{violet} $u$};
    \node at (5.7, 2.2) {\color{OliveGreen} $h_1v$};
    \node at (14.2, .88) { \color{OliveGreen} $h_4v$};
\end{tikzpicture}
\caption{The arcs $h_1v$ (left) and $h_4v$ (right) with arc $u$ disjoint to $v$. }
\label{fig-trickafter}
\end{figure}

So $v$ is connected by a path in $\natp{S_{g}^1}/{\sim}$ to any other arc with the same endpoints.  Since the connected components of any arc includes an arc with any given pair of endpoints, then by the Putman trick,  the graph $\natp{S_{g}^1}/{\sim}$ is connected.

\medskip

\noindent \textit{Step 2:  } Since the quotient graph is connected, then it is sufficient to show that $\natp{S_{g}^1}$ is connected between any identified elements in the same equivalence class.  Notice that each class consists of exactly two isotopy classes,  relative to the boundary and tangent spaces on the boundary.  Specifically,  there is one isotopy class for each subarc of the boundary between the endpoints. Let $a$ and $a'$ be disjoint representatives of the same quotient class.  As shown in Figure~\ref{fig-homconnect}, there exists a path in the graph $\natp{S_{g}^1}$ between $a$ and $a'$. 

\begin{figure}[h!]
\centering
\begin{tikzpicture}
\small
    \node[anchor=south west, inner sep = 0] at (0,0){\includegraphics[width=4in]{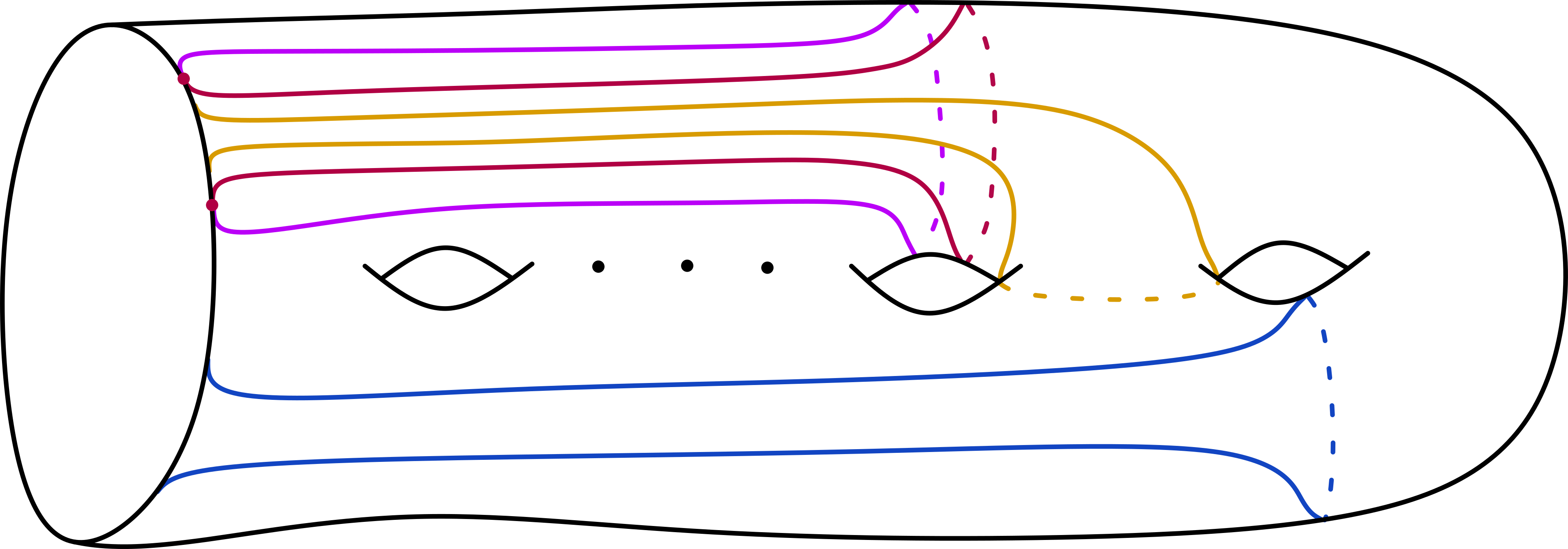}};
    \node at (4.2,2.05) {\color{violet} $a'$};
    \node at (6.65,3.3) {\color{purple} $a$};
    \node at (6.8, .48) {\color{blue} $b$};
    \node at (7.9, 2.3) { \color{brown} $c$};
\end{tikzpicture}
\caption{The path $\{a, b, c, a'\}$ in $\natp{S_{g}^1}$. }
\label{fig-homconnect}
\end{figure}

By the change of coordinates principle,  this path also exists for any other element of $\natp{S_{g}^1}$ that shares the same endpoints as $a$.  Since the choice of endpoints was arbitrary, then there is a path between any two isotopy classes from the same quotient class.  
\end{proof}

\medskip

\noindent \textit{A fine arc graph.  } We now consider the fine version of the previously defined arc graph. We define the \emph{fine paired $k$-smooth nonseparating arc graph}, denoted $\fnatp{S_{g}^1}$. The vertices of this graph consist of arcs in $S_{g}^1$ with the following properties:
\begin{enumerate}[label=(\alph*)]
\item the arcs are essential and nonseparating
\item the endpoints of each arc are distinct
\item the interior of the arcs are disjoint from any boundaries
\item there is an extension of each arc by a subarc of the boundary such that this extension is a $C^k$ curve on $S_{g}^1$
\end{enumerate}
There are edges between vertices when the corresponding arcs have the following properties:
\begin{enumerate}[label=(\roman*)]
\item the interiors of the arcs are disjoint
\item the arcs are jointly nonseparating
\item each extension given in property (e) contains both endpoints of the other arc
\end{enumerate}
These are the same conditions given for the non-fine arc graph, but the vertices now consist of arcs instead of isotopy classes of arcs. We will now show that this fine arc graph is also connected.

\begin{lemma}\label{finenonseppairedarcconnected}
Let $S_{g}^1$, with $g \geq 2$.  Then $\fnatp{S_{g}^1}$ is connected.
\end{lemma}

\begin{proof}
Let $a,b$ be arcs in $\fnatp{S_{g}^1}$.  By Lemma \ref{nonseppairedarcgraphconnected},  the isotopy classes $[a]$ and $[b]$ are connected by a path $\{ [c_0]=[a], [c_1], [c_2], \ldots,  [c_{n-1}], [c_n]=[b]\}$.  We prove this lemma in two steps, first showing that each isotopy class is connected and then use this fact to get a path from $a$ to $b$.

\medskip

\noindent \textit{Step 1:  } We claim that each isotopy class relative to the boundary is connected. To prove this, let $p,q$ be two arcs in the same isotopy class in $\fnatp{S_{g}^1}$.  Then there is an isotopy $H : I \times I \rightarrow S_{g}^1$ relative boundary and tangent spaces on the boundary such that each arc $c_t(s) = H(s, t)$ is an vertex of $\fnatp{S_{g}^1}$.  Note that each of the $c_t(s)$ has the same tangent direction at its endpoints as $p$ and $q$.  

It is possible to partition $I$ into small enough pieces $\{t_0 =0, t_1, \ldots, t_{n-1}, t_n=1\}$ such that consecutive $c_{t_{i}}$, $c_{t_{i+1}}$ live in a small nonseparating ribbon $R$ on $S_{g}^1$ with ends in $\partial S_{g}^1$.  Moreover, since $c_{t_{i}}$ and $c_{t_{i+1}}$ have the same endpoints, then $R$ can be made sufficiently narrow to ensure that there are two components of $\partial S_{g}^1 \setminus \overline{R}$ with nonempty interiors.  In particular, any $C^k$ extension of the $c_t$ will contain exactly one of these components.  Denote this component of $\partial S_{g}^1 \setminus \overline{R}$ by $d$.  

Since $S_{g}^1 \setminus R$ is a surface with two boundaries and $g-1$ genus and $g \geq 2$, then there is a nonseparating $C^k$ arc $e_i$ on $S_{g}^1 \setminus R$ with endpoints on $d$ that is disjoint to both $c_{t_i}$ and $c_{t_{i+1}}$.   Note that since $d$ has a nonempty interior, then there is a nonempty open subset of $S_{g}^1 \setminus R$ that contains the interior of $d$.  Thus $e_i$ can be chosen such that it is compatible with either orientation on $\partial S_{g}^1$.  In particular,  this allows us to pick an $e_i$ such that the $C_k$ extension intersects both endpoints of the $c_t$.  Since $e_i$ is nonseparating on $S_{g}^1 \setminus R$, then both pairs $(e_i, c_{t_i})$ and $(e_i, c_{t_{i+1}})$ are jointly nonseparating.  So $e_i$ is adjacent to $c_{t_i}$ and $c_{t_{i+1}}$ in $\fnatp{S_{g}^1}$.  Thus $$\{c_{t_0} = p, e_0, c_{t_1}, \ldots, c_{t_{n-1}}, e_{n-1}, c_{t_n} = q\}$$ is a path in $\fnatp{S_{g}^1}$ between $p$ and $q$.  Since the choice of $p$ and $q$ was arbitrary,  then any isotopy class is connected in $\fnatp{S_{g}^1}$.

\medskip

\noindent \textit{Step 2:  } Since $\{ [c_0]=[a], [c_1], [c_2], \ldots,  [c_{n-1}], [c_n]=[b]\}$ is the path between $[a]$ and $[b]$, then there are representatives $c_i'$ and $c_i''$ for each $i$ such that $c_{i}''$ is disjoint from $c_{i+1}'$.  By step 1,  there are paths $\{c_i' = d_0^i , d_1^i, \ldots d_{k_i}^i = c_i''\}$ within each isotopy class $[c_i]$ that connect $c_i'$ and $c_i''$. 
Thus 
\begin{flalign*}
&\{a = c_0'  = d_0^0, d_1^0, \ldots, d_{k_0}^0 = c_0'', c_1 ' = d_0^1, \ldots \\
& \qquad \qquad \ldots,   d_{k_{i-1}}^{i-1} = c_{i-1}'',  c_i' = d_0^i , d_1^i, \ldots d_{k_i}^i = c_i'', c_{i+1}' = d_0^{i+1}, \ldots \\
& \qquad \qquad \qquad \qquad  \qquad \qquad \qquad \qquad \ldots,  d_{k_{n-1}}^{n-1} = c_{n-1}'', c_n' = d_0^n, \ldots, d_{k_n}^n = c_n'' = b\}
\end{flalign*}
is a path from $a$ to $b$.  Since $a$ and $b$ were chosen arbitrarily,  then $\fnatp{S_{g}^1}$ is connected.
\end{proof}

\subsection{Kernels of simple $\mathbf{k}$-smooth pairs}\label{subsectionkernel}
In this section, we build up an equivalence relation on simple $k$-smooth pairs based on the inessential curves that they define. We then use this equivalence relation to show that any automorphism of the $C^k$-curve graph is well-defined on inessential curves.  

\medskip

\noindent \textit{Kernel curves.  } Given any simple $k$-smooth pair $(a, b)$ that determines the curve $d = a_1 \cup b_1 \cup c_1 \cup c_2$, there is a unique inessential curve $e = a_2 \cup b_2 \cup c_1 \cup c_2$. We call $e$ the \emph{kernel curve} of $(a,b)$. 

Note that the kernel curve of any simple $k$-smooth pair will be a vertex in $\efcn{k}{S_g}$.  Our goal is to show that this curve $e$ is well-defined under any automorphism of $\fcn{k}{S_g}$. To achieve this, we use an equivalence relation on the set of simple $k$-smooth pairs in $S_g$.

\medskip

We define the equivalence relation $\!\!{\kereq}$, called \textit{kernel equivalence},  on the set of simple $k$-smooth pairs in $S_g$ by the transitive closure of the following relation: $(a,b) \kereq (b,u)$ if there exists a $b' \in \fcn{k}{S_g}$ homotopic to $b$ and disjoint from $a$, $b$,  and $u$ such that for any $\gamma \in \fcn{k}{S_g}$  contained in the annulus between $b$ and $b'$,  then $\gamma$ intersects $a$ if and only if $\gamma$ intersects $u$.  This generating relation is detecting whether the subarcs $a_2$ and $u_2$ are the same.  

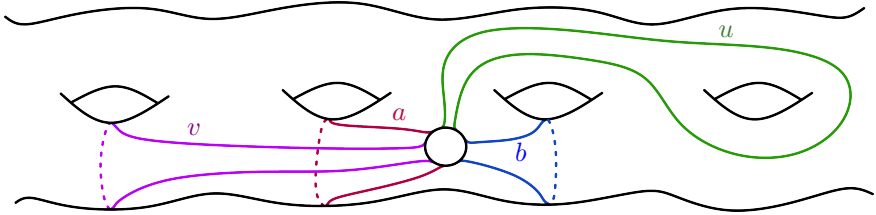
\begin{figure}[h!]
\centering
\begin{tikzpicture}
\small
    \node[anchor=south west, inner sep = 0] at (0,0){\includegraphics[width=4.5in]{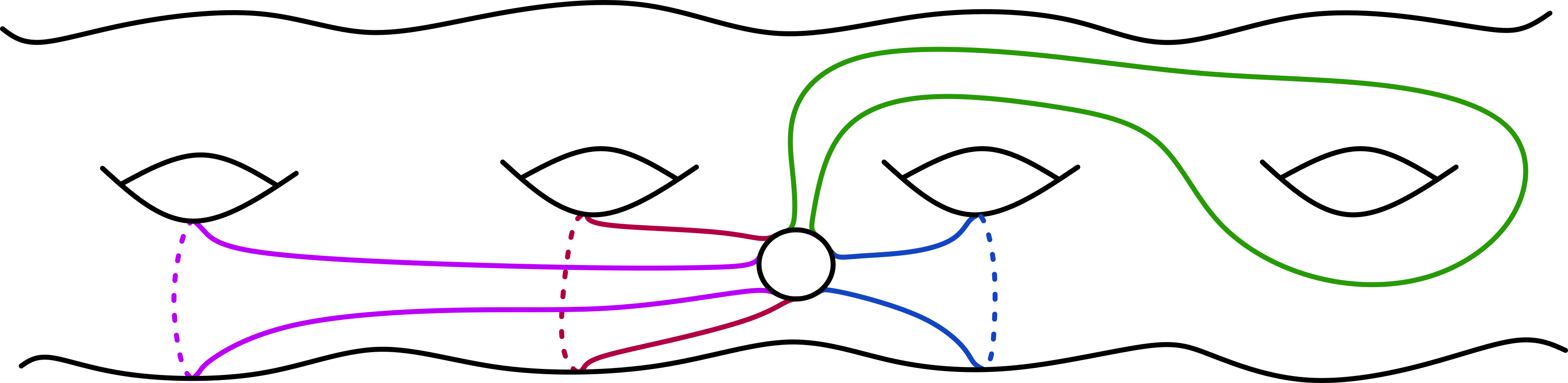}};
    \node at (2.5,1.1) {\color{violet} $v$};
    \node at (5.2,1.3) {\color{purple} $a$};
    \node at (6.8, .8) {\color{blue} $b$};
    \node at (9.5, 2.4) { \color{OliveGreen} $u$};
\end{tikzpicture}
\caption{A couple of simple $k$-smooth pairs with kernel equivalence.}
\end{figure}

Note that $(a,b) = (b,a)$ since simple $k$-smooth pairs are unordered.  Our goal is to show that two simple $k$-smooth pairs are equivalent exactly when they have the same kernel curve.  We start with the following lemma.

\begin{lemma}\label{equivsamecurve}
Let $S_g$ be a surface with $g \geq 2$.  Let $(a,b)$ and $(p,q)$ be simple $k$-smooth pairs on $S_g$. If $(a,b) \kereq (p, q)$, then they have the same kernel curve.
\end{lemma}

\begin{proof}
Since kernel equivalence is a transitive closure, then it is sufficient to show that the pairs in the defining relation $(a,b) \kereq (b,u)$ have the same kernel curve.  Denote the kernel of $(a,b)$ by $e$ and the kernel of $(b,u)$ by $f$. 

Since $(a,b) \kereq (b,u)$, then there exists some $b' \in \fcn{k}{S_g}$ homotopic to $b$ and disjoint from $a$, $b$,  and $u$.  
Note that both $e$ and $f$ are contained in the closure of the annulus bounded by $b$ and $b'$.  Denote this annulus by $B$ and the disks bounded by $e$ and $f$ by $D_e$ and $D_f$.  So $B \setminus a$ will have two components, $D_e$ and an annulus whose boundaries are $b'$ and a second curve made from an arc of $b$ and an arc of $a$.  If $u$ has any point in the interior of this subannulus, then it would be possible to find an curve $\gamma \in \fcn{k}{S_g}$ in this annulus that intersects $u$ but not $a$.  By the defining relation,  this cannot happen. So $u \cap B$ must be contained in the closure of $D_e$.  But this will mean that $D_f \subset D_e$.  Since $a$ and $c$ are symmetric in the defining relation, then $D_e \subset D_f$ as well. Thus $D_e = D_f$, which means that $e = f$. 
\end{proof}

In order to show the other direction, we need to introduce a new graph,  called a line graph.  We then give the connection between the line graph and kernel equivalence.

\medskip

\noindent \textit{Line graphs.  }
For any graph $\Gamma$, define $L(\Gamma)$ to be the \emph{line graph} of $\Gamma$,  which has a vertex for each edge of $\Gamma$ and an edge whenever two edges of $\Gamma$ share a vertex. 

Observe that for any simple $k$-smooth pair $(a,b)$ in $S_g$ with kernel $e$, there exists an open disk $D_e$ whose boundary is the curve $e$.  Set $a'$ and $b'$ to be the closures of the open arcs $a \setminus e$ and $b \setminus e$ in $S_g \setminus D_e$.  Note that $a'$ and $b'$ determine an edge in $\fnatp{S_g \setminus D_e}$, denote this edge by $\overline{ab}$. 

\begin{lemma}\label{equivreltoarcgraph}
Let $S_g$ be a surface with $g \geq 2$.  Let $(a,b)$ and $(p,q)$ be simple $k$-smooth pairs on $S_g$ with the same kernel $e$.  Let $D_e$ be an open disk whose boundary is $e$. Then $(a,b) \kereq (p,q)$ if and only if $\overline{ab}$ and $\overline{pq}$ are connected in $L(\fnatp{S_g \setminus D_e})$.
\end{lemma}

\begin{proof}
We first prove the forward direction.  The first goal is to show that any generating relation will be connected in $L(\fnatp{S_g \setminus D_e})$.  Let $(a, b) \kereq (b, u)$ be the generating relation.  By Lemma \ref{equivsamecurve},  $(a,b)$  and $(b,u)$ will have the same kernel $e$.  So both  $\overline{ab}$ and $\overline{bu}$ are edges in $\fnatp{S \setminus D_e}$ and will be connected to $b$.  So the vertices $\overline{ab}$ and $\overline{bu}$ in the line graph $L(\fnatp{S \setminus D_e})$ are connected by an edge.  Since the relation $\kereq$ is a transitive closure, then any two simple $k$-smooth pairs that are equivalent will be able to be connected by a path  in $L(\fnatp{S \setminus D_e})$.

Next, we prove the backward direction.  If $\overline{ab}$ and $\overline{pq}$ are connected in $L(\fnatp{S \setminus D_e})$, then there is a path in $\fnatp{S \setminus D_e}$ from the arcs $a'$ to $p'$ (this path may, or may not,  pass through $b'$ and $q'$).  Denote the vertices of this path by $\{u_0 = a', u_1, \ldots, u_{n-1}, u_n = p'\}$.   Since $u_1$ is an arc in $S_g \setminus D_e$, it can be extended to a $C^k$ curve $u_1'$ in $S_g$ by a subarc of $e$.  The edge condition of $\fnatp{S \setminus D_e}$ will imply that $(a,u_1')$ is a simple $k$-smooth pair in $S_g$.  Notably, then $e \setminus a$ is an arc contained in $u_1'$ and $b$. 

Consider the surface $\overline{S_g \setminus (D_e \cup a')}$.  Since $b'$ and $u_1$ are disjoint from $a'$ and closed subsets, then the boundary component $a' \cup (e \setminus a)$ has an open neighborhood that does not intersect $b'$ or $u_1$.  So there exists $\alpha \in \fcn{k}{S_g}$ in this neighborhood that are disjoint from $a, b$, and $u_1'$.  Moreover,  $\alpha$ will be homotopic to $a$ since $e$ is inessential.  Since $e \setminus a$ is the only portion of $b$ and $u_1'$ that is contained in the annulus between $a$ and $\alpha$.  So any $\gamma \in \fcn{k}{S_g}$ that is contained in this annulus will intersect $b$ if and only if it intersects $e \setminus a$ if and only if it intersects $u_1'$.  Thus $(a,b) \kereq (a, u_1')$.  

Since $(u_1', u_2')$ is a simple $k$-smooth pair such that $(a, u_1') \kereq,(u_1', u_2')$, then this process can then be continued for each step of the path. Thus by transitivity of the relation $\kereq$,  $(a,b) \kereq (p, q)$. 
\end{proof}

\noindent Utilizing the above result, we get following lemma, which is the reverse of Lemma \ref{equivsamecurve}.

\begin{lemma}\label{samecurveequiv}
Let $S_g$ be a surface with $g \geq 2$.  Let $(a,b)$ and $(p,q)$ be simple $k$-smooth pairs on $S_g$.  If $(a,b)$ and $(p,q)$ have the same kernel curve, then $(a,b) \kereq (p,q)$.
\end{lemma}

\begin{proof}
Since $(a,b)$ and $(p,q)$ have the same kernel curve $e$, then by Lemma \ref{equivreltoarcgraph}, they will be in the same equivalence class exactly if they are connected in the line graph $L(\fnatp{S \setminus D_e})$. But by Lemma \ref{finenonseppairedarcconnected}, $\fnatp{S \setminus D_e}$ is connected when its genus is at least 2. Thus its line graph must also be connected and so $(a,b) \kereq (p,q)$. 
\end{proof}

\noindent We are now ready to state and prove the final result of this section.

\begin{lemma}\label{autopreservesrelation}
Let $S_g$ be a surface with $g \geq 2$.  Let $(a,b)$ and $(p,q)$ be simple $k$-smooth pairs on $S_g$ and let $\alpha \in \aut \fcn{k}{S_g}$.  Then $(a,b)$ and $(p,q)$ have the same kernel curve if and only if $(\alpha(a), \alpha(b))$ and $(\alpha(p), \alpha(q))$ have the same kernel curve. 
\end{lemma}

\begin{proof}
Suppose that $(a,b)$ and $(p,q)$ have the same kernel curve.  Then by Lemma \ref{samecurveequiv},  $(a,b) \kereq (p,q)$.  Since the generating relation for kernel equivalence depends on homotopies, intersections, and containment in annuli, then by Lemmas \ref{autpreservesephomandhull} and \ref{autpreservessubsurfaces} the equivalence relation will be preserved by $\alpha$.  Thus $(\alpha(a),\alpha(b)) \kereq (\alpha(p),\alpha(q))$.  By Lemma \ref{equivsamecurve},  $(\alpha(a),\alpha(b))$ and $(\alpha(p),\alpha(q))$ will have the same kernel curve.  

Since $\alpha$ is an automorphism, then the reverse direction comes from using the automorphism $\alpha^{-1}$ with the exact same reasoning. 
\end{proof}

\subsection{Inessential curves}\label{subsectioninessential}
In this section, we give some ways that inessential curves and their intersections can be detected in both $\efcn{k}{S_g}$ and $\fcn{k}{S_g}$.  This is necessary in the proof of Proposition \ref{ec1toc1iso}.  Recall from Lemma \ref{sepcurveinextgraph}, that any curve in $\efcn{k}{S}$ is separating if and only if its link is a join.

\begin{lemma} \label{inessentialingraph}
Let $S_g$ be a compact surface with $g \geq 1$ and no boundary.  A curve $c \in \efcn{k}{S_g}$ is inessential if and only if $c$ is separating and one side of the join contains only separating curves. 
\end{lemma}

\begin{proof}
Every inessential curve bounds a disk and separates this disk from the rest of the surface.  So $c$ is a separating curve and one side of the join will be all the curves contained in this disk. But these must also be inessential curves and so are all separating curves. 

If $c$ is a separating curve that is essential in $S_g$, then there must be at least one genus in each component of $S_g \setminus c$. But then there will be nonseparating curves in every part of the join. 
\end{proof}

The lemma above guarantees that the restriction map given in Proposition \ref{ec1toc1iso} is well-defined on vertices.  We now identify the graph structures associated to the different possibilities for edges.

\begin{lemma} \label{inessentialandessentialintersect}
Let $S_g$ be a surface with $g \geq 2$.  Let $f$ be an essential curve and $e$ be an inessential curve in $\efcn{k}{S_g}$. Then $f$ and $e$ are disjoint if and only if there exists a simple $k$-smooth pair $(a,b)$ with kernel $e$ such that $a$ is disjoint and not homotopic to $f$ and there is a curve $a' \in \fcn{k}{S_g}$ homotopic to $a$ and disjoint from $a, b, e, $ and $f$. 
\end{lemma}

\begin{proof}
Suppose that $f$ and $e$ are disjoint. Since $f$ is essential, then the subsurface $S_g \setminus f$ will have some genus,  and so there are essential nonseparating curves in $\fcn{k}{S_g}$ in the same component as $e$ that do not intersect $f$.  One of these can be smoothly isotoped to loop around $e$ while still remaining disjoint from $f$.  Let this curve be $a$.  Note that $a$ will not be homotopic to $f$ since it was essential in $S_g \setminus f$.  Since $g \geq 2$, it will be possible to find another nonseparating curve $b \in \fcn{k}{S_g}$ such that $(a,b)$ form a simple $k$-smooth pair with kernel $e$.  Since $e$ and $f$ do not intersect, then $a$ can be homotoped through $e$ to make a new curve isotopic to $a$ and consisting of arcs of $a$ and $e \setminus a$. Since $S_g \setminus (f \cup b)$ is open in $S_g$, then this curve can be pushed off of $a$ and $e$ into a neighborhood and smoothed to make a curve $a' \in \fcn{k}{S_g}$ that has the desired properties. 

On the other hand, suppose that we have a simple $k$-smooth pair $(a,b)$ with kernel $e$ such that $a$ and $a'$ have the given properties in the lemma statement.  Since $a$ and $a'$ are disjoint and homotopic, then they must bound an annulus in $S_g$.  Since $b$ is disjoint from $a'$,  and $b \setminus a$ is two arcs with endpoints on opposite sides of $a$, then the annulus between $a$ and $a'$ must contain one of these arcs.  But $e$ is inessential and $b$ is not homotopic to $a$, so it must be the arc that is contained in $e$.  So $e$ is completely contained in the closed annulus between $a$ and $a'$. Since $f$ is essential, not homotopic to $a$ and does not intersect $a$ or $a'$, then it cannot intersect any curves that are contained in the annulus. Thus $e$ and $f$ do not intersect.
\end{proof}

 \begin{figure}[h!]
\centering
\begin{tikzpicture}
\small
    \node[anchor=south west, inner sep = 0] at (0,0){
\includegraphics[width=1.75in]{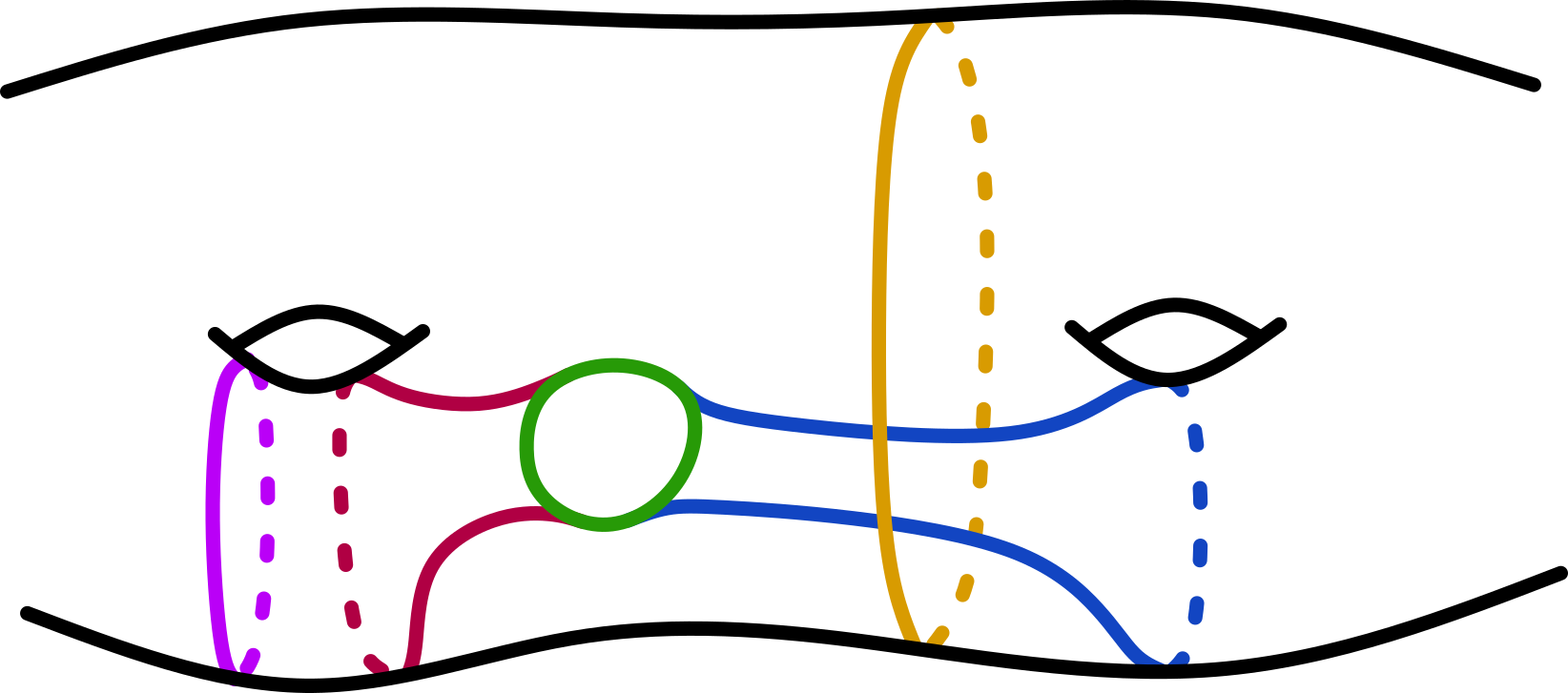} };
    \node[anchor=south west, inner sep = 0] at (5,0) {\includegraphics[width=1.75in]{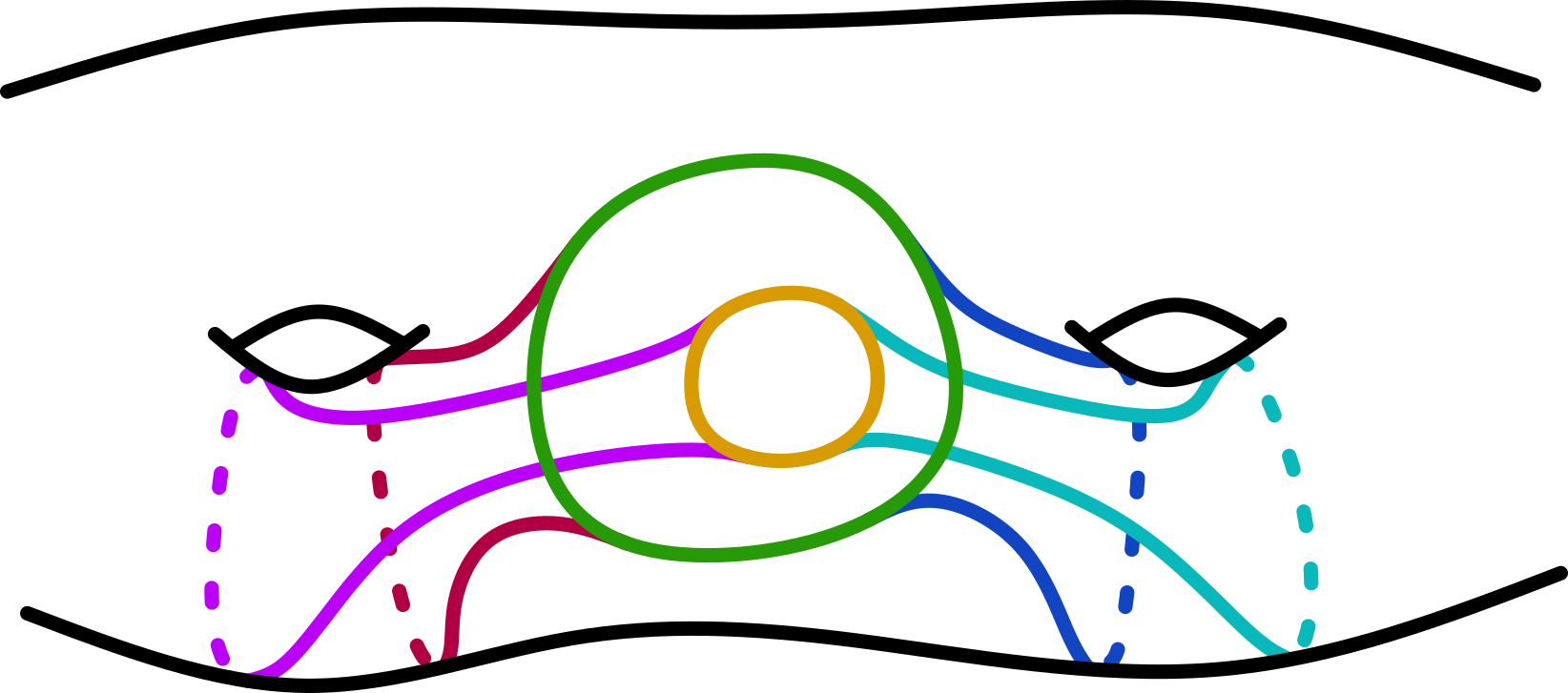}};
    \node[anchor=south west, inner sep = 0] at (10,0) {\includegraphics[width=1.75in]{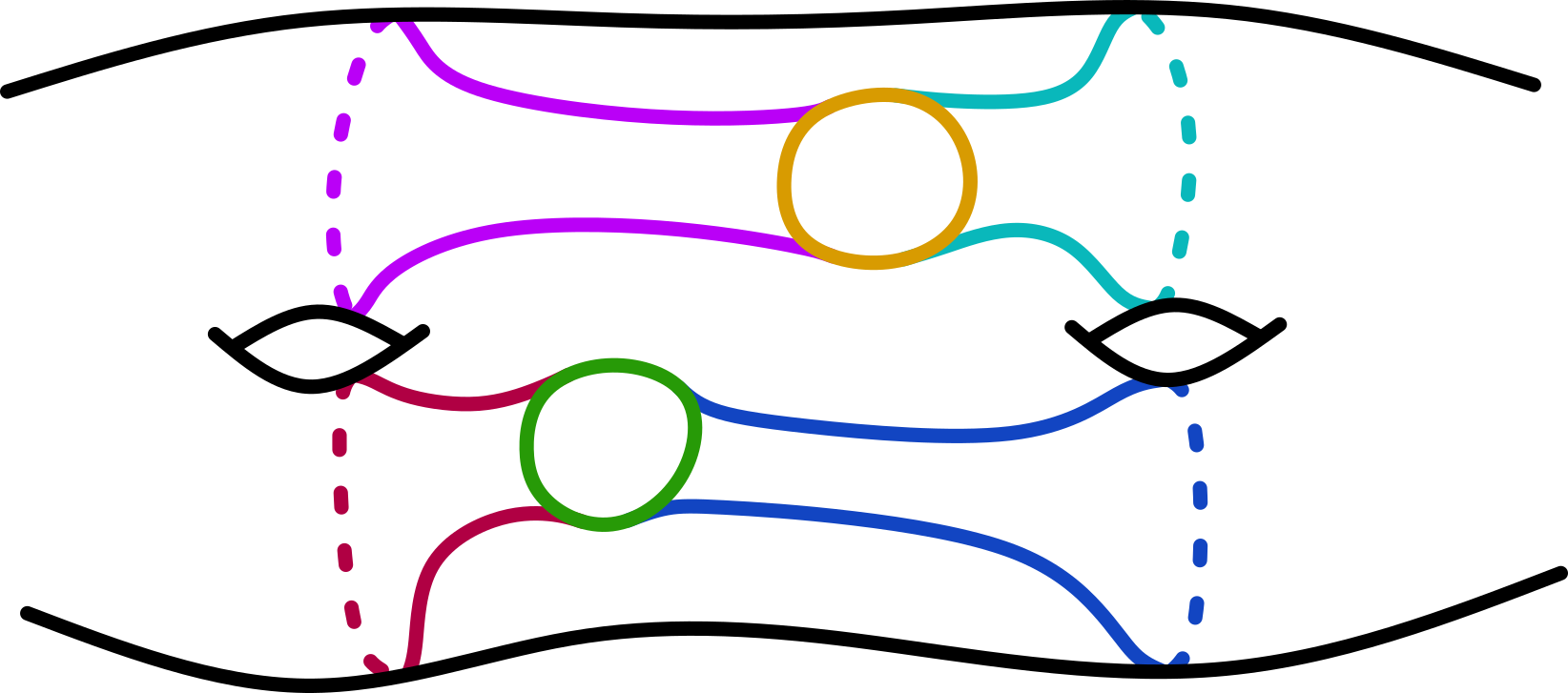}};
\end{tikzpicture}
\caption{The three possible edges given in Lemmas~\ref{inessentialandessentialintersect} (left),  \ref{twoinessentialintersectnest} (middle),  and \ref{twoinessentialintersectnotnest} (right).}
\label{fig-extedges}
\end{figure}

In the proof of \cite[Lemma 5.1]{LMPVY},  Long--Margalit--Pham--Verberne--Yao give a classification of disjoint inessential curves using bigon pairs that has a similar flavor to Lemma \ref{twoinessentialintersectnest} below.  They claim that their classification works for both nested and unnested inessential curves, but in actuality only applies to the nested case.  Instead, the classification for the unnested case is that there will exist two disjoint bigon pairs,  similar to the one given in Lemma \ref{twoinessentialintersectnotnest}.

\begin{lemma} \label{twoinessentialintersectnest}
Let $S_g$ be a surface with $g \geq 2$.  Let $e$ and $f$ be inessential curves in $\efcn{k}{S_g}$.  If necessary, switch the names so that $D_e \not\subset D_f$.  Then $e$ and $f$ are disjoint and nested if and only if for any simple $k$-smooth pair $(a,b)$ with kernel $e$, there exists a simple $k$-smooth pair $(p,q)$ with kernel $f$ such that $p$ is disjoint from $a$ and $q$ is disjoint from $b$.
\end{lemma}

\begin{proof}
Suppose that $e$ and $f$ are disjoint and nested.  Since $D_e \not \subset D_f$, then $f$ is contained in $D_e$.  Let $(a,b)$ be any simple $k$-smooth pair with kernel $e$.  Then $a$ and $b$ can be smoothly isotoped within $D_e$ to make a simple $k$-smooth pair $(p,q)$ with kernel $f$. By smoothly pushing $p$ off of $a$ outside of $D_e$,  it is possible to make the curve $p$ be disjoint from $a$.  Similarly, $q$ can be made disjoint from $b$.  So the desired simple $k$-smooth pair exists. 

To prove the other direction, first suppose that $e$ and $f$ are not disjoint.  Let $x \in e \cap f$.  Let $(a,b)$ be a simple $k$-smooth pair with kernel $e$ such that $x$ is contained in one of the overlap sections $c_1$ or $c_2$.  Thus $f$ will intersect both $a$ and $b$.  So any simple $k$-smooth pair $(p,q)$ with kernel $f$ must have either $p$ or $q$ contain $x$.  But then that curve will intersect both $a$ and $b$ and thus we cannot have the characterization given in the lemma. 

Now suppose that $e$ and $f$ are disjoint and not nested.  Then it is possible to find a simple $k$-smooth pair $(a,b)$ with kernel $e$ such that $a$ and $b$ both intersect $f$ at least twice non-consecutively with respect to an orientation on $f$.  By the construction of simple $k$-smooth pairs,  for any simple $k$-smooth pair $(p,q)$ with kernel $f$,  both $p \cap f$ and $q \cap f$ are both a single arc.  If $p$ is disjoint from $a$, then $q$ must contain any points in $f \cap a$.  Similarly, if $q$ is disjoint from $b$, then $p$ contains all the points in $f \cap b$.  But then if both of these conditions hold, $p \cap f$ and $q \cap f$ must each have at least two arcs, and thus $(p,q)$ cannot be a simple $k$-smooth pair. 
\end{proof}

\begin{lemma} \label{twoinessentialintersectnotnest}
Let $S_g$ be a surface with $g \geq 2$.  Let $e$ and $f$ be inessential curves in $\efcn{k}{S_g}$.  Then $e$ and $f$ are disjoint and not nested if and only if there exist disjoint simple $k$-smooth pairs $(a,b)$ with kernel $e$ and $(p,q)$ with kernel $f$.
\end{lemma}

\begin{proof}
Suppose that $e$ and $f$ are inessential curves that are not nested. Since each pair of simple closed curves lives on a nonseparating pair of pants,  it is possible to find disjoint simple $k$-smooth pairs for $e$ and $f$ on a surface $S_g$ with genus $g \geq 2$. 

To prove the reverse direction, suppose there exist disjoint simple $k$-smooth pairs $(a,b)$ with kernel $e$ and $(p,q)$ with kernel $f$. Thus $e$ and $f$ must be disjoint.  Since the simple $k$-smooth pairs are comprised of essential curves, then $a,  b \not \subset D_f$ and $p,  q \not \subset D_e$.  Thus no part of $a \cup b$ is in $D_f$ and no part of $p \cup q$ is in $D_e$. So $e$ and $f$ are not nested.
\end{proof}

\begin{lemma}\label{c1preservesintersectionsinec1}
Let $e, f \in \efcn{k}{S_g}$ with $g \geq2$ and $\alpha \in \aut \fcn{k}{S_g}$.  Then 
\begin{enumerate}[noitemsep,topsep=0pt, label=$(\roman*)$]
\item if $e$ inessential and $f$ essential,  and $(a,b)$ is any simple $k$-smooth pair with kernel $e$, then $e$ and $f$ are disjoint if and only if the kernel of $(\alpha(a), \alpha(b))$ is disjoint from $\alpha(f)$. 
\item if $e$ and $f$ are both inessential,  and $(a,b)$ is any simple $k$-smooth pair with kernel $e$ and $(p,q)$ is any simple $k$-smooth pair with kernel $f$, then $e$ and $f$ are disjoint if and only if the kernel of $(\alpha(a), \alpha(b))$ is disjoint from the kernel of $(\alpha(p), \alpha(q))$. 
\end{enumerate}
\end{lemma}

\begin{proof}
By Lemma~\ref{autopreservesrelation},  $\alpha$ will be well-defined on kernel curves of simple $k$-smooth pairs.  Thus the choice of the simple $k$-smooth pair $(a,b)$ does not affect the resulting kernel of $(\alpha(a), \alpha(b))$. 

Let $f$ be an essential curve and $e$ be an inessential curve.  By Lemma~\ref{inessentialandessentialintersect},  $e$ and $f$ being disjoint depends only on simple $k$-smooth pairs, intersections, and homotopies.  So by Lemma~\ref{autpreservesephomandhull} and Proposition~\ref{c1preservessimplesmoothpair}, these structures are preserved by $\alpha$.  Thus the kernel of $(\alpha(a), \alpha(b)) $ and $\alpha(f)$ will still have the same relations that will force them to be disjoint. 

Now,  consider the case where $e$ and $f$ are two inessential curves.   By Lemmas~\ref{twoinessentialintersectnest} and \ref{twoinessentialintersectnotnest},  $e$ and $f$ are disjoint, depending characterizations using simple $k$-smooth pairs and intersections.  By Proposition~\ref{c1preservessimplesmoothpair},  this characterization will be preserved.  Thus the kernels of $(\alpha(a), \alpha(b))$ and $(\alpha(p), \alpha(q))$ will also be disjoint.  \end{proof}

\subsection{Proof of proposition}\label{subsectionpropauts}

Finally, we complete this section by showing that restriction to the essential curves gives an isomorphism between the automorphisms of the extended $C^k$-curve graph to the automorphism of the $C^k$-curve graph.

\begin{proof}[Proof of Proposition~\ref{ec1toc1iso}]
Let $\beta \in \aut ( \efcn{k}{S_g})$.  By Lemma~\ref{inessentialingraph}, $\beta$ must send inessential curves to inessential curves and send essential curves to essential curves.  Since $\fcn{k}{S_g}$ is a subgraph of $\efcn{k}{S_g)}$, then $\xi(\beta)$ will be a well-defined restriction of $\beta$ to only the essential curves.  Note that $\xi$ is a homomorphism since the restriction has no affect on the action of $\beta$ on the essential curves. 

Suppose that $\xi(\beta)$ is the identity automorphism.  Then each essential curve in $\efcn{k}{S_g}$ will be mapped to itself.  But for each inessential curve $e$, there is a simple $k$-smooth pair $(a,b)$ of essential curves with $e$ as their kernel.  Since $(a,b)$ will be sent to itself and determines a unique inessential curve, then $e$ must also be sent to itself by $\beta$. Thus $\beta$ is the identity. So $\xi$ is injective. 

Now let $\alpha \in \aut \fcn{k}{S_g}$.  Consider an inessential curve $e \in \efcn{k}{S_g}$.  Pick a simple $k$-smooth pair $(a,b)$ that determines a curve $d$ and whose kernel is $e$.  From Proposition~\ref{c1preservessimplesmoothpair},  $(\alpha(a), \alpha(b))$ must be a simple $k$-smooth pair that determines $\alpha(d)$. Thus it must also have a kernel curve.  Define an extension of $\alpha$, denoted $\widehat{\alpha}$,  to $\efcn{k}{S_g}$ by sending $e$ to the kernel of $(\alpha(a), \alpha(b))$.  By Lemma~\ref{autopreservesrelation},  $\widehat{\alpha}$ is well-defined on inessential curves. 

To check is that $\widehat{\alpha}$ is an automorphism of $\efcn{k}{S_g}$, it needs to preserve the edges between all possible combinations of essential and inessential curves.  On pairs of essential curves, $\widehat{\alpha}$ is exactly automorphism $\alpha$ of $\fcn{k}{S_g}$,  and so it will preserve edges between essential curves.  By Lemma~\ref{c1preservesintersectionsinec1},  $\widehat{\alpha}$ preserves the intersections between inessential and essential curves and the intersections between pairs of essential curves. So $\widehat{\alpha}$  is an automorphism of $\efcn{k}{S_g}$.  

In addition, $\xi ( \widehat{\alpha}) = \alpha$ by construction and so $\xi$ is surjective.  Thus $\xi$ is a bijective homomorphism,  which is an isomorphism.
\end{proof}

\section{Finishing the proof}\label{sectionfinal}
We can now complete Step 3 from the proof outline given in the introduction.  We show that the natural homomorphism $\eta$ between $\thomeo(S_g)$ and $\aut\fcn{k}{S_g}$ is actually an isomorphism. 

\begin{proof}[Proof of Main Theorem] If the composed map $$\textstyle \eta \circ \rho \circ \xi^{-1} : \aut\fcn{k}{S_g} \rightarrow \aut \fcn{k}{S_g}$$ is the identity map,  then $\eta$ will be surjective.  

Let $\alpha \in \aut \fcn{k}{S_g}$. Then $(\eta \circ \rho \circ \xi^{-1})(\alpha)$ is the action on $ \fcn{k}{S_g}$ induced by the homeomorphism $( \rho \circ \xi^{-1})( \alpha)$.  By Proposition \ref{ec1tohomeo1},  this homeomorphism induces the action on $ \efcn{k}{S_g}$ corresponding to $\xi^{-1} (\alpha)$.  But then $\eta \circ \rho$ is the restriction of the action of $\xi^{-1}(\alpha)$ on $ \efcn{k}{S_g}$ to $ \fcn{k}{S_g}$, which by Proposition \ref{ec1toc1iso} is exactly the map $\xi$. So $\eta \circ \rho \circ \xi^{-1}$ is the identity map.  

To show that $\eta$ is injective, consider a homeomorphism $f$ that induces the same action on $\aut \fcn{k}{S_g}$ as the identity map.  For any $x \in S_g$, consider a torus pair $(a,b)$ that intersects exactly at the point $x$. Then $(f(a), f(b)) = (a, b)$ is still a torus pair that intersects only at the point $x$. Thus $f$ must fix the point $x$. Since this holds for every $x \in S_g$, then $f$ must be the identity map on $S_g$. 

Since $\eta$ is a homomorphism and bijective, then it is an isomorphism. 
\end{proof}

\section{Example elements from $\mathbf{\bhomeo^1(S)}$} \label{examples}
We now give some explicit examples of maps that are elements of $\thomeo(S)$ that are not in $\diff^k(S)$.  These examples are given explicitly on $\mathbb{R}^2$.  These non-differentiable points can then be added to any diffeomorphism on a surface using local charts and bump functions.  

\medskip

\subsection*{Example 1: Losing linearity with Le~Roux--Wolff} This first example is based on one given by Le Roux--Wolff \cite[Section 5]{LRW}.  The goal of their example was to give an element of $\homeo^{\infty}(S)$ (which they denote by $\homeo_{\infty}(S)$) that is not in $\diff(S)$.  This is accomplished by adjusting the angles of curves through a point to prevent the induced map between tangent spaces from being linear.  Our example is nearly identical to theirs with only small adjustments to the justifications to account for the use of $C^k$ curves instead of $C^{\infty}$ curves. 

\medskip

Let $f: \mathbb{R} \rightarrow \mathbb{R}$ be a $C^k$ diffeomorphism such that $f(x) = x$ when $x$ is outside $[1/2, 2]$, but $f(1) \neq 1$. Define $F: \mathbb{R}^2 \rightarrow \mathbb{R}^2$ by 
$$F(x,y) = \left\{ \begin{array}{ll} (x,x \, f(y/x)) & \mbox{if } x \neq 0 \\  (x, y) & \mbox{if } x =0 \end{array} \right.$$
The reasoning that Le Roux--Wolff uses  to show that $C^{\infty}$ curves are mapped to $C^{\infty}$ curves relies on the next level of differentiabiliy. So their reasoning will only give that $C^{k}$ curves are mapped to $C^{k-1}$ curves.  Fortunately,  it can be shown by direct computation that the image of any $C^k$ curve under $F$ will also be a $C^k$ curve with nonzero tangent vectors.  But $F$ is not a diffeomorphism. 

 \begin{figure}[h!]
\centering
\begin{tikzpicture}
\small
    \node[anchor=south west, inner sep = 0] at (0,0){
\includegraphics[width=1.5in]{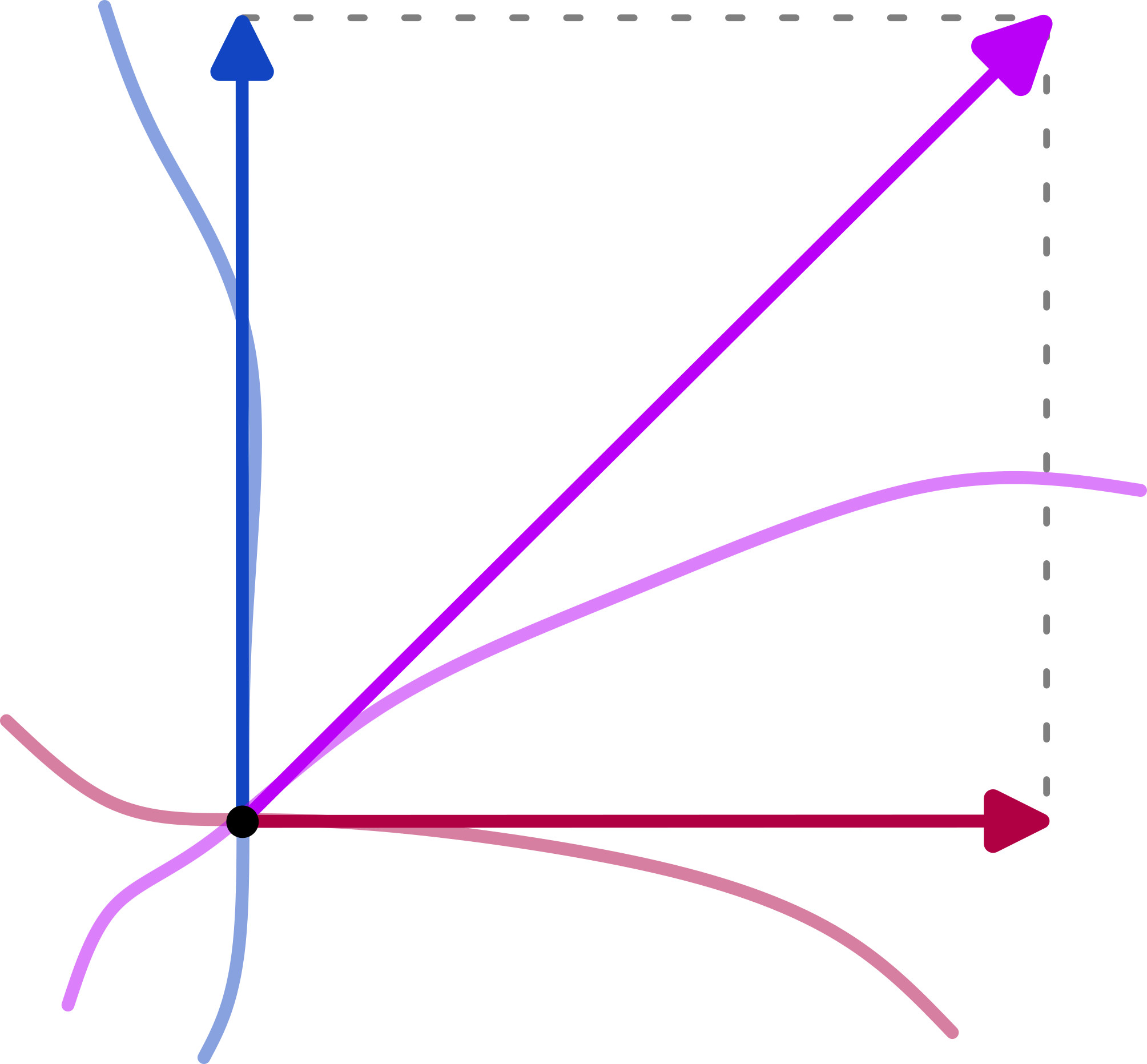} };
    \node at (5,2) {\Large $\xrightarrow{dF_0}$};
    \node[anchor=south west, inner sep = 0] at (6,0) {\includegraphics[width=1.5in]{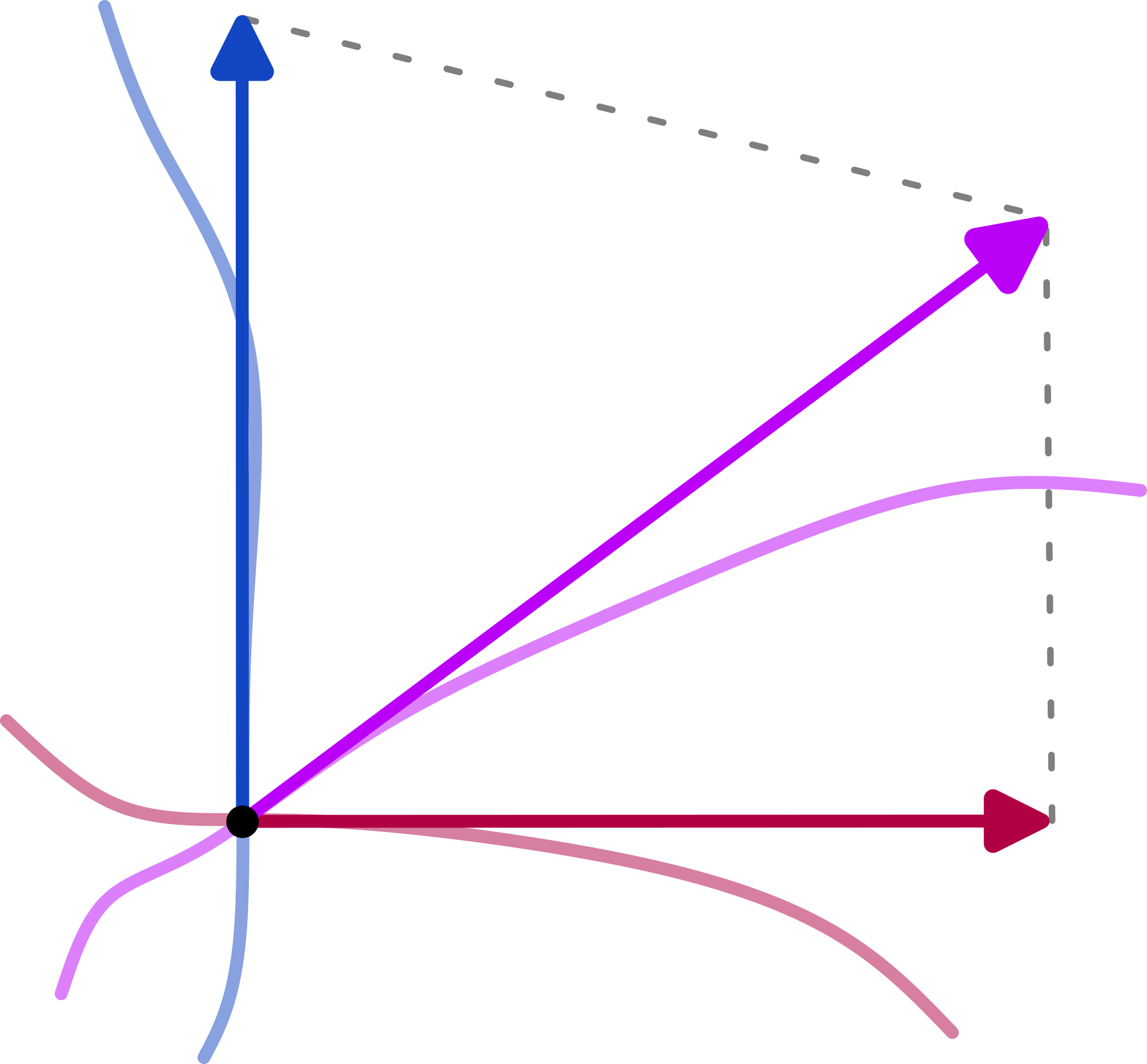}};
\end{tikzpicture}
\caption{The induced map on tangent vectors at the origin. }
\label{fig-ex1LRW}
\end{figure}

The images of unit parameterized curves on the $x$- and $y$-axes with tangent vectors $(0,1)$ and $(1,0)$ at the origin will have those same tangent vectors, while the image of a curve with tangent vector $(1,1)$ at the origin will have the tangent vector of $(1, f(1)) \neq (1,1)$. Thus the induced map on the tangent spaces cannot be linear.

\medskip

\subsection*{Example 2: Containment of $\mathbf{\bhomeo^1(S)}$}
A natural question to ask is if there are any containment relationships between the $\homeo^n(S)$.  The following example shows that $\homeo^1(S) \not\subset \homeo^k(S)$, for any $2 \leq k \leq \infty$. 

To accomplish this, we consider is whether an element of $\homeo^1(S)$ must also map any $C^2$ curve to another $C^2$ curve.  A curve can be determined to be $C^2$ by the existence of continuous curvature, which is independent of the parameterization of the curve.

Let 
$$ m(x, y) = \left\{ \begin{array}{ll} \displaystyle \frac x {\sqrt{x^2+y^2}}+2 & \mbox{if } (x,y) \neq (0,0) \\
 2 & \mbox{if } (x,y) = (0,0) \end{array} \right.$$ Then define the map  $$M(x,y) = \left( x m(x, y), y m(x, y) \right)$$  
 
 \begin{figure}[h!]
\centering
\begin{tikzpicture}
\small
    \node[anchor=south west, inner sep = 0] at (0,0){
\includegraphics[width=2in]{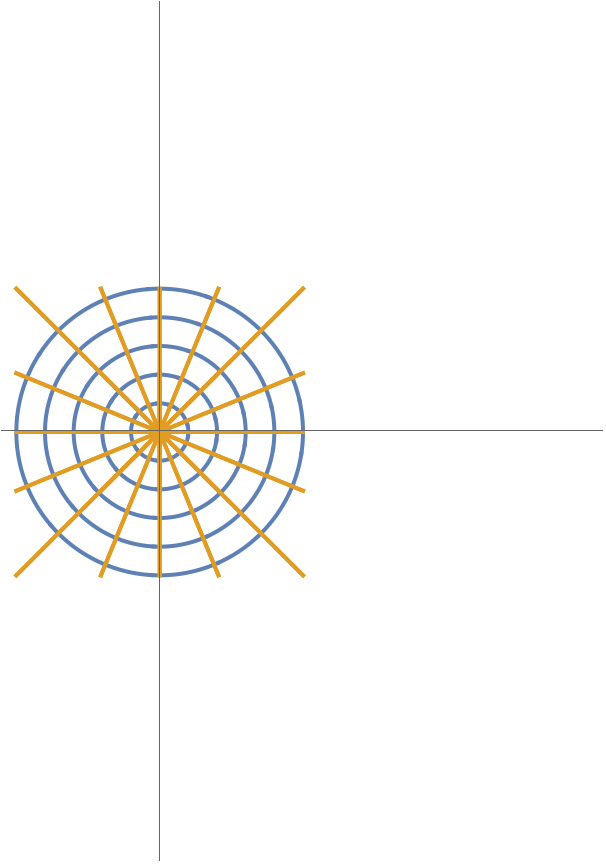} };
    \node at (6,3.8) {\Large $\xrightarrow{M}$};
    \node[anchor=south west, inner sep = 0] at (7,0) {\includegraphics[width=2in]{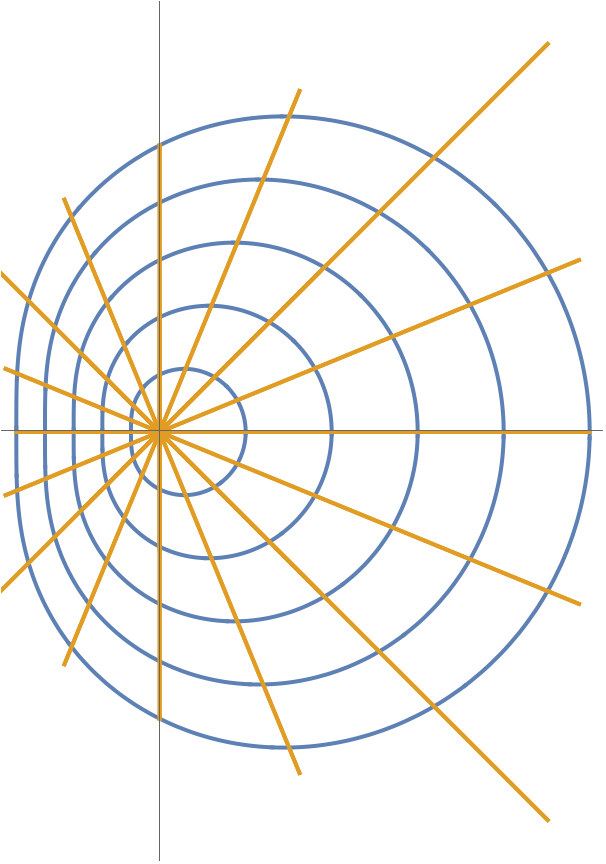}};
\end{tikzpicture}
\caption{The action of the map $M$ on a polar grid.}
\label{fig-ex2}
\end{figure}

Note that the limit of $m(x, y)$ is 1 when approaching the origin along the negative $x$-axis but when approaching from the positive $x$-axis, the limit of $m(x, y)$ is 3.  This difference causes a jump discontinuity in the curvature at the origin for the image under $M$ of some $C^{\infty}$ curves.  

 \begin{figure}[h!]
\centering
\begin{tikzpicture}
\small
    \node[anchor=south west, inner sep = 0] at (0,0){
\includegraphics[width=2.25in]{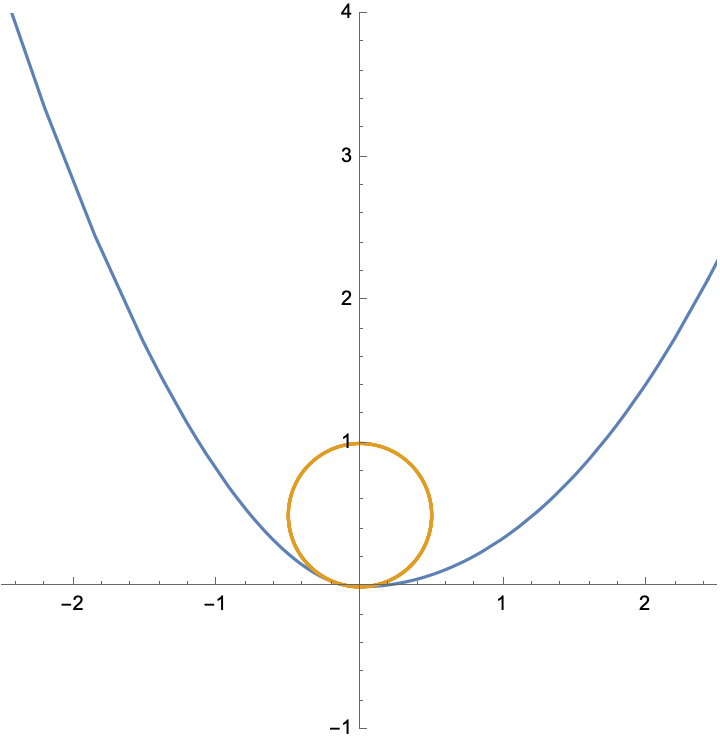} };
    \node[anchor=south west, inner sep = 0] at (7,0) {\includegraphics[width=2.25in]{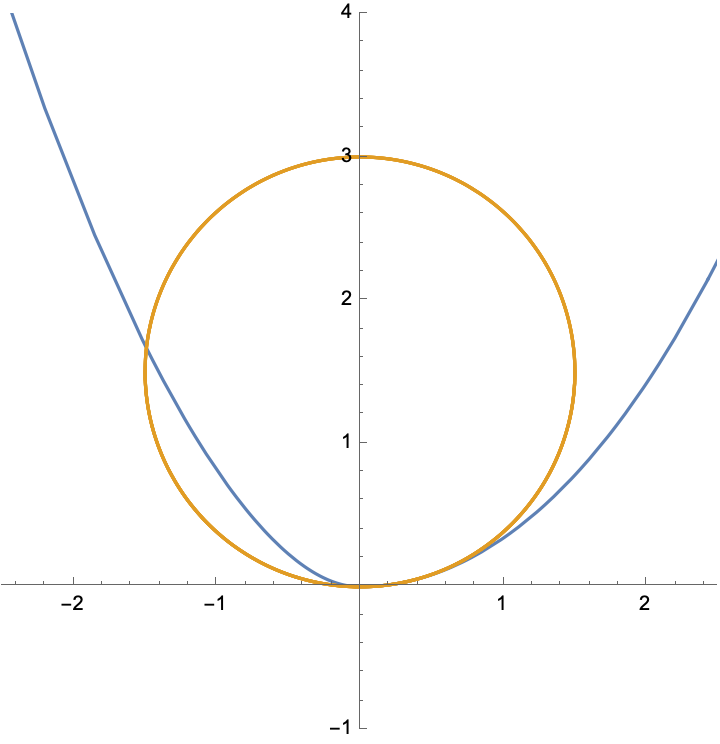}};
\end{tikzpicture}
\caption{Osculating circles for the image of $y=x^2$, showing that $\kappa(0^-)   = 2$ and $\kappa(0^+)  = \frac{2}{3} $.}
\label{fig-curvature}
\end{figure}

Since any $C^{m}$ curve is also a $C^k$ curve, whenever $0 \leq k \leq m \leq \infty$, then $M$ is not an element of $\homeo^k(\mathbb{R}^2)$, for any $2 \leq k \leq \infty$.  But $M$ is an element of $\homeo^1(\mathbb{R}^2)$, and so we have that $\homeo^1(S) \not\subset \homeo^k(S)$ for all $k \geq 2$. 

\subsection*{Example 3: A cautionary non-example for $\mathbf{\bhomeo^2(S)}$}
For the next example, we introduce a homeomorphism $N$ of $\mathbb{R}^2$ that  appears to be an element of $\homeo^2(\mathbb{R}^2)$ but not an element of $\homeo^1(\mathbb{R}^2)$.  This map sends every $C^2$ curve to another $C^2$ curve, but takes some $C^1$ curves to curves that have corners, which are not $C^1$.  

Before explaining how the map $N$ works, we recall some facts about $C^2$ curves.  First, every $C^2$ curve in $\mathbb{R}^2$ has finite curvature at every point.  Thus all $C^2$ curves with horizontal tangent at the origin must lie in a neighborhood of the origin between the curves $y=|x|^{3/2}$ and $y=-|x|^{3/2}$ which have infinite curvature at the origin.  The map $N$ leaves this region unchanged, while squeezing curves that lie between $y=|x|^{3/2}$ and $y=|x|$ towards $y=|x|$ as the curve approaches the origin.  Figure~\ref{fig-nonex2} shows how $N$ preserves $y=x^2$, but forces the image of $y=\frac{10}{9}|x|^{3/2}$ to have a corner at the origin.

\begin{figure}[h!]
\centering
\begin{tikzpicture}
\small
    \node[anchor=south west, inner sep = 0] at (0,0){
\includegraphics[width=2.5in]{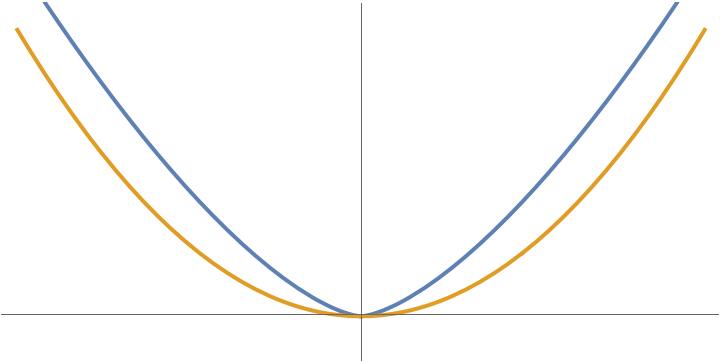} };
    \node at (7.25,2) {\Large $\xrightarrow{N}$};
    \node[anchor=south west, inner sep = 0] at (8,0) {\includegraphics[width=2.5in]{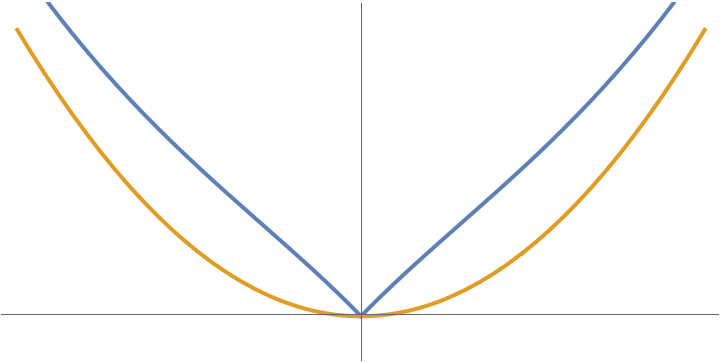}};
\end{tikzpicture}
\caption{The action of the map $N$ on $y=x^2$ and $y=\frac{10}{9}|x|^{3/2}$}
\label{fig-nonex2}
\end{figure}

The squeezing is then also done symmetrically across the $x$-axis so that any $C^2$ curves with non-horizontal tangents are squeezed towards $y=x$ or $y=-x$ with curvature 0, which forces their images to be $C^2$ as well. 

Where $N$ fails to be an example from $\homeo^2(\mathbb{R}^2)$ is that it forces the image of too many curves to have curvature 0 at the origin.  For example, the image of the curve $y=sgn(x)\, |x|^{3/2}$ will have curvature 0 at the origin and thus be $C^2$.  So the action of $N$ on $\fcn{2}{\mathbb{R}^2}$ fails to be invertible.

\bibliographystyle{alpha}
\bibliography{mybib2}

\end{document}